\documentclass[12pt]{article}
\usepackage{amsmath,amssymb,amsfonts,amsthm}
\usepackage{setspace}
\newenvironment{myabstract}{\par\noindent
{\bf Abstract . } \small }
{\par\vskip8pt minus3pt\rm}
\newcounter{item}[section]
\newcounter{kirshr}
\newcounter{kirsha}
\newcounter{kirshb}
\newenvironment{enumarab}{\setcounter{kirshb}{1}
\begin{list}{(\arabic{kirshb})}{\usecounter{kirshb}} }{\end{list}}
\newtheorem{theorem}{Theorem}[section]

\newtheorem{lemma}[theorem]{Lemma}
\newtheorem{corollary}[theorem]{Corollary}
\theoremstyle{definition}

\newtheorem{example}[theorem]{Example}
\newtheorem{definition}[theorem]{Definition}
\def\R{\mathbb{R}}
\def\Q{\mathbb{Q}}
\def\C{{\mathfrak{C}}}
\def\Fm{{\mathfrak{Fm}}}

\def\At{{\bf At}}
\def\Nr{{\mathfrak{Nr}}}

\def\Sg{{\mathfrak{Sg}}}
\def\Fm{{\mathfrak{Fm}}}
\def\A{{\mathfrak{A}}}
\def\B{{\mathfrak{B}}}
\def\C{{\mathfrak{C}}}
\def\D{{\mathfrak{D}}}
\def\M{{\mathfrak{M}}}
\def\N{{\mathfrak{N}}}
\def\Sn{{\mathfrak{Sn}}}
\def\CA{{\bf CA}}
\def\QA{{\bf QA}}

\def\QEA{{\bf QEA}}

\def\Lf{{\bf Lf}}

\def\PA{{\bf PA}}
\def\PEA{{\bf PEA}}

\def\K{{\bf K}}
\def\K{{\bf K}}

\def\RCA{{\bf RCA}}

\def\Rd{{\ Rd}}
\def\(R)RA{{\bf (R)RA}}
\def\RA{{\bf RA}}

\def\R{\mathbb{R}}
\def\Q{\mathbb{Q}}

\def\Sc{{\bf Sc}}

\def\Id{{\bf Id}}
\def\c #1{{\cal #1}}
 \def\CA{{\sf CA}}
\def\B{{\sf B}}
\def\G{{\sf G}}
\def\w{{\sf w}}
\def\y{{\sf y}}
\def\g{{\sf g}}
\def\b{{\sf b}}
\def\r{{\sf r}}
\def\K{{\sf K}}
\def\tp{{\sf tp}}

 \def\Cm{{\mathfrak{Cm}}}
\def\Nr{{\mathfrak{Nr}}}

\def\restr #1{{\restriction_{#1}}}
\def\cyl#1{{\sf c}_{#1}}
\def\diag#1#2{{\sf d}_{#1#2}}

\def\R{\sf R}
\def\Ra{{\mathfrak{Ra}}}
\def\Ca{{\mathfrak{Ca}}}
\def\set#1{\{#1\} }
\def\Ra{{\mathfrak{Ra}}}
\def\Nr{{\mathfrak{Nr}}}
\def\Tm{{\mathfrak{Tm}}}
\def\A{{\mathfrak{A}}}
\def\B{{\mathfrak{B}}}
\def\C{{\mathfrak{C}}}
\def\D{{\mathfrak{D}}}

\def\A{{\mathfrak{A}}}
\def\B{{\mathfrak{B}}}
\def\C{{\mathfrak{C}}}
\def\D{{\mathfrak{D}}}

\def\U{{\mathfrak{U}}}

\def\Bb{{\mathfrak{Bb}}}
\def\L{{\mathfrak{L}}}
\def\Rd{{\mathfrak{Rd}}}
\def\Bb{{\mathfrak{Bb}}}
\def\At{{\mathfrak{At}}}
\def\L{{\mathfrak{L}}}

\def\CA{{\bf CA}}
\def\RA{{\bf RA}}

\def\RCA{{\bf RCA}}

\def\G{{\bf G}}

\def\F{{\mathfrak{F}}}
\def\At{{\sf{At}}}
\def\N{\mathbb{N}}
\def\R{\mathfrak{R}}

\def\Cs{{\sf Cs}}

\def\cyl#1{{\sf c}_{#1}}
\def\diag#1#2{{\sf d}_{#1#2}}

\def\c #1{{\cal #1}}

\def\pa{$\forall$}
\def\pe{$\exists$}

\def\ef{Ehren\-feucht--Fra\"\i ss\'e}

\def\nodes{{\sf nodes}}

\def\restr #1{{\restriction_{#1}}}

\def\Ra{{\mathfrak{Ra}}}
\def\Nr{{\mathfrak{Nr}}}
\def\Z{{\cal Z}}
\def\CA{{\bf CA}}
\def\RCA{{\bf RCA}}
\def\c#1{{\mathcal #1}}

\def\set#1{ \{#1\}}

\def\Ca{{\mathfrak Ca}}
\def\b#1{{\bar{ #1}}}
\def\pe{$\exists$}
\def\pa{$\forall$}
\def\Cm{{\mathfrak Cm}}
\def\Sg{{\mathfrak Sg}}

\def\Rl{{\mathfrak Rl}}
\def\N{{\cal N}}

\def\ls { L\"owenheim--Skolem}
\def\At{{\sf At}}
\def\Uf{{\sf Uf}}

\def\rng{{\sf rng}}
\def\dom{{\sf dom}}
\def\Cm{{\sf Cm}}

\def\w{{\sf w}}
\def\g{{\sf g}}
\def\y{{\sf y}}
\def\r{{\sf r}}
\def\tp{{\sf tp}}
\def\cyl#1{{\sf c}_{#1}}

\def\diag#1#2{{\sf d}_{#1#2}}

\def\ws{winning strategy}
\def\ef{Ehren\-feucht--Fra\"\i ss\'e}

 \def\CA{{\sf CA}}
\def\Cs{{\sf Cs}}
\def\RCA{{\sf RCA}}

\def\RA{{\sf RA}}
\def\PA{{\sf PA}}

\def\PEA{\sf PEA}
\def\QEA{{\sf QEA}}

\def\y{{\sf y}}
\def\g{{\sf g}}
\def\r{{\sf r}}
\def\w{{\sf w}}
\def\Z{{\mathbb{Z}}}
\def\N{{\mathbb{N}}}

\def\CPEA{\sf CPEA}
\def\U{{\mathfrak{U}}}

\def\Gp{{\sf Gp}}

\def\c{{\sf c}}
\def\s{{\sf s}}

\def\Ig{{\sf Ig}}

\def\d{ Dedekind--MacNeille}
\def\Id{{\sf Id}}
\def\Sc{{\sf Sc}}

\def\Lf{{\sf Lf}}
\def\K{{\sf K}}
\def\nodes{{\sf nodes}}

\def\G{{\bold G}}
\def\Sc{{\sf Sc}}

\def\PA{{\sf PA}}
\def\Id{{\sf Id}}
\def\QEA{{\sf QEA}}
\def\s{{\sf s}}

\def\CA{{\sf CA}}
\def\K{{\sf K}}
\def\QA{{\sf QA}}
\def\RCA{{\sf RCA}}

\def\A{{\mathfrak{A}}}
\def\Cs{{\sf Cs}}

\def\V{{\sf V}}

\def\cyl#1{{\sf c}_{#1}}

\def\diag#1#2{{\sf d}_{#1#2}}

\def\la#1{\langle#1\rangle}
\def\G{{\mathfrak{G}}}
\def\Nr{{\sf Nr}}
\def\de{Dedekind-MacNeille}
\def\Cm{{\mathfrak{Cm}}}
\def\M{{\sf M}}
\def\T{{\sf T}}
\def\VT{{\sf VT}}
\def\CRCA{{\sf CRCA}}
\def\PEA{{\sf PEA}}
\title{Finite relation algebras and omitting types in modal fragments of first order logic}
\author{Tarek Sayed Ahmed\\
Department of Mathematics, Faculty of Science,\\
Cairo University, Giza, Egypt.
 }
\date{}
\begin{document}
\maketitle

\begin{myabstract}
Let $2<n\leq l<m< \omega$. Let $L_n$ denote first order logic restricted to the first $n$ variables.  We show 
that the omitting types theorem fails dramatically 
for the $n$--variable fragments of first order logic with respect to clique guarded semantics, and for 
its packed $n$--variable fragments. 
Both are modal fragments of $L_n$. 
As a sample, we show that if there exists 
a finite relation algebra with a so--called strong $l$--blur, and no
$m$--dimensional relational basis, then there exists a countable, atomic and complete $L_n$ theory $T$  
and  type $\Gamma$, such that $\Gamma$ is realizable in every so--called $m$--square model
of $T$,  but any witness isolating $\Gamma$ cannot use less than $l$ variables. 
An $m$--square model $\M$ of $T$ gives a form of clique guarded semantics, where the parameter $m$, measures
how locally well behaved $\M$ is. Every ordinary model is $k$--square for any $n<k<\omega$, 
but the converse is not true.  Any 
model $\M$ is $\omega$--square, and the two notions  
are equivalent if $\M$ is countable.

Such relation algebras are shown to exist 
for certain values of $l$ and $m$ like for $n\leq l<\omega$ and $m=\omega$,  and for $l=n$ and $m\geq n+3$.
The case $l=n$ and $m=\omega$ gives that the omitting types theorem fails for $L_n$ with respect to (usual) 
Tarskian semantics: There is an atomic countable $L_n$ theory $T$  for which the single 
non--principal type consisting of co--atoms cannot be omitted in any model $\M$ of $T$.

For $n<\omega$, positive results on omitting types are obained for $L_n$ by imposing extra conditions 
on the theories and/or the types omitted.  Positive and negative results on omitting types 
are obtained for infinitary variants and extensions of $L_{\omega, \omega}$.
\end{myabstract}

\section{Introduction}

\subsection{Overview}

We follow the notation of \cite{1} which is in conformity with the notation in the monograph 
\cite{HMT2}.  In particular, for any pair of ordinal $\alpha<\beta$, $\CA_{\alpha}$ stands for the class of cylindric algebras of dimension
$\alpha$, $\RCA_{\alpha}$ denotes the class of representable $\CA_{\alpha}$s and $\Nr_{\alpha}\CA_{\beta}(\subseteq \CA_{\alpha})$ 
denotes the class of $\alpha$--neat reducts of $\CA_{\beta}$s. This class is studied in 
the chapter  \cite{Sayedneat} of \cite{1}.   
For $n<\omega$, $L_n$ denotes first order logic restricted
to the first $n$ variables. 
We assume familiarity with the basic notions of the (duality) theory of Boolean algebras with  operators $\sf BAO$s,  like {\it atom structures} and  {\it  complex algebras}.
A more than an adequate reference for our purpose is \cite[\S 2.5, \S 2.6,  \S 2.7]{HHbook}. However, such notions and related ones from the theory of 
$\sf BAO$s used in the sequel 
will be recalled fully in due time. 

{\it Unless otherwise indicated, $n$ will be a finite ordinal $>2$.}

In algebraic logic one can study certain ``particular'' algebraic structures 
(or simply algebras) and constructions on such algebras in isolation, 
as objects of interest in their own right and 
go on to discuss questions which naturally arise independently, regardless of their possible initial 
outcome in connection to  logic.   
But more often than not, such  seemingly purely algebraic investigations   
shed light on the logic side to which they could have owed their birth.   
Indeed, results in algebraic logic are most significant when they have non--trivial impact on (first order) logic, via 
existing `bridge theorems'.  
But sometimes a particular purely algebraic result, 
can dictate  even building a new bridge or more. In this paper, we build new bridges between the algebraic notions of {\it atom--canonicity and relativized complete representations} 
for varieties of $\CA_n$s, 
on the one hand, and the metalogical one of a restrcited version of the Henkin--orey omitting types, namely, Vaught's theorem on existence of countable models for countable atomic theories, 
for the 
so--called packed fragments of $L_n$.

Atom--canonicity is an important {\it persistence property} in various modal logics, that applies to the class of their 
modal algebras; for example the variety $\RCA_n$ viewed as the class of modal algebras of the (modal formalism) of $L_n$ is not
atom--canonical, 
because applying the the complex algebra operator to 
atom structures of $\RCA_n$s, can take us outside $\RCA_n$, more succintly, 
$\Cm(\At\RCA_n)\nsubseteq \RCA_n$.
 
As the name suggests, complete representability is a {\it semantical} notion. A {\it relativized representation} of $\A\in \CA_n$ is an injective homomorphism
$f:\A\to \wp(V)$ where $V\subseteq {}^nU$ for some non--empty set $U$ and the operations on $\wp(V)$ are the concrete operations 
defined
like in cylindric set algebras of dimension $n$ relativized to $V$. In this case we say that $\A$ is {\it represented on $V$ via $f$}, or simply represented on $V$. 
An {\it ordinary representation}, or just a representation  of $\A\in \CA_n$, 
is a relativized representation $f:\A\to \wp(V)$  of $\A$, such that the top element $V$ on which $\A$ is represented, is a disjoint union of cartesian squares, 
that is $V=\bigcup_{i\in I}{}^nU_i$, $I$ is a non-empty indexing set, $U_i\neq \emptyset$  
and  $U_i\cap U_j=\emptyset$  for all $i\neq j$. Such a $V$ is called a {\it generalized cartesian space} of dimension $n$.

A {\it (relativized) complete representation} of $\A\in \CA_n$ on $V\subseteq {}^nU$, is a (relativized) representation of $\A$ on $V$ via $f$, that preserves  arbitrary sums carrying
them to set--theoretic unions, that is the representation $f:\A\to \wp(V)$ is required to satisfy   
$f(\sum S)=\bigcup_{s\in S}f(s)$ for all $S\subseteq \A$ such that $\sum S$ exists. In this case, we 
say that $\A$ is {\it completely represented on $V$ via $f$}, or simply completely represented on $V$. 
If $\A\in \CA_n$, then a relativized representation of $\A$ on $V$ via $f$ is not necessarily a {\it complete} relativized  representation of $\A$ on $V$ via (the same) $f$. 
It is known that if $\A\in \CA_n$ has
a relativized complete representation, then the Boolean reduct of $\A$ is atomic \cite{HH}; below every non-zero element of $\A$ there is an atom (a minimal non-zero element). 
It is also known that there are countable and atomic $\RCA_n$s,
that have no complete representations on generalized cartesian spaces of dimension $n$. So atomicity is necessary but not sufficient for complete 
representability of $\CA_n$s on generalized cartesian spaces of the same dimension $n$.
In fact, the class of completely representable $\CA_\alpha$s on cartesian spaces of dimension $\alpha$, 
when $\alpha>2$ (where generalized cartesian spaces for infinite dimensions is defined like the finite dimensional case), 
is not even elementary \cite[Corollary 3.7.1]{HHbook2}.
It can happen that an atomic $\RCA_n$ may not allow a complete representation on a generalized cartesian space, 
but can be completely represented {\it in  relativized sense} 
on some set $V$ of $n$--ary sequences.

The subtle semantical phenomena of (relativized) complete representability is closely
related to the algebraic notion of  atom--canonicity of (certain supervarieties of) $\RCA_n$ 
(like $\bold S\Nr_n\CA_m$ for $2<n<m<\omega$),
and to the metalogical property of  omitting types in $n$--variable fragments of first order logic
\cite[Theorems 3.1.1-2, p.211, Theorems 3.2.8, 9,10]{Sayed}, 
when non--principal types are omitted with respect to (relativized) semantics.

The typical question is: given a $\A\in \CA_n$ and a family $(X_i: i\in I)$ of subsets of $\A$ ($I$ a non--empty set),  such that $\prod X_i$ exists in $\A$ for each $i\in I$
is there a representation $f:\A\to \wp(V)$, on some $V\subseteq {}^nU$ ($U$ a non--empty set), 
that carries  this set of meets to set theoretic intersections, in the sense that  $f(\prod X_i)=\bigcap_{i\in I}f(x)$ for all $i\in I$? 
When the algebra $\A$ is countable, $|I|\leq \omega$ and $\prod X_i=0$ for all $\in I$, this is an algebraic version of an  omitting types theorem; 
the representation $f$ {\it omits the given set of meets (or non-principal types) on $V$.}  
When it is {\it only one meet} consisting of co-atoms, in an atomic algebra, such a representation $f$ will 
be a {\it complete representation on $V$}, and 
this is equivalent to that $f(\prod X)=\bigcap_{x\in X}f(x)$ for all $X\subseteq \A$ 
for which $\prod X$ exists in $\A$ \cite{HH}. The last condition, when we require that $V$ is a generalized cartesian space of dimension $n$, is an algebraic version 
of Vaught's theorem for first order logic, namely, the unique (up to isomorphism) atomic, 
equivalently prime,  model of a countable 
atomic theory omits all 
non--principal types. 

In proving non--atom canonicity of $\RCA_n$ ($2<n<\omega$), one typically constructs an atom structure $\bf At$, such that 
{\it the term algebra}, in symbols $\Tm\bf At$, is in $\RCA_n$, while its \de\ completion, the complex algebra of its atom structure, in symbols 
$\Cm\bf At$, is outside $\RCA_n$. Here $\Tm\bf At$ is the subalgebra of $\Cm \bf At$ generated by the atoms, and $\Cm\bf At$ is the smallest complete $\CA_n$
which contains $\Tm \bf At$ as a dense subalgebra.

This atom structure is {\it weakly}, but not {\it strongly} representable, in the sense that
it is the atom structure of both a representable and a non--representable algebra demonstrating that the 
representability of an atomic algebra is not determined by the structure of its atom structure.
The algebra $\Tm\bf At$, though both atomic and representable, {\it cannot be completely representable on a cartesian space of dimension $n$}, because a
complete representation of $\Tm \bf At$ on such a space, necessarily induces 
an ordinary representation of $\Cm \bf At$ on the same space.

On the other hand, it is well known that Vaught's theorem is a consequence of the classical Orey--Henkin omitting types theorem. 
Algebraically, the well--known proof (briefly) goes as follows: Let $T$ be a countable atomic theory. Let $\A=\Fm_T$ be 
the Tarski--Lindenbaum quotient algebra corresponding to $T$.   
Then it is not hard to show that,  from the definition of the atomicty of $T$, 
the Boolean algebra $Nr_m\A=\{a\in \A: \c_ia=a,  \forall i\in \omega\setminus m\}$ will be 
atomic for all $m\in \omega$. The  required atomic (prime) model is 
the model omitting the familiy $(X_m: m\in \omega)$
where $X_m$ is the set of co--atoms in $Nr_m\A$. 
Observe that $\prod X_m=0$ 
for all $m\in \omega$, so the $X_m$s are non--principal types that can be omitted 
simultaneously.

Proving failure of Vaught's theorem for $L_n$ which a restricted version of  an omitting types theorem (when one asks for models omiting only the type of co--atoms), one may 
algebraically 
construct a countable atomic $\RCA_n$ that is not completely representable even in a relativized sense. 
To get sharper results, 
one shows that 
varieties like $\bold S\Nr_n\CA_{n+k}(\supsetneq \RCA_n$) for some $k<\omega$ are not atom--canonical.
Both tasks can (and will be done) for $2<n<\omega$. 

For the first task, we construct a countable and atomic $\A\in \RCA_n$ having no complete {\it relativized $n+3$--flat} representations (the notion of {\it flatness} will be defined in a while), 
{\it a fortiori} $\A$ will have no complete
representation.
For the second task, we show that there exists an $\RCA_n$ atom structure $\bf At$ such that 
$\Cm{\bf At}\notin \bold S\Nr_n\CA_{n+3}$. The last result is sharper, because as demonstrated below, 
it implies the former. The implication here is like in the classical case: The existence of a complete $n+3$--flat representation of
$\A(\in \RCA_n)$ implies that $\Cm\At\A$ has an $n+3$--flat representation, where $\At\A$ is the atom structure of $\A$, which is eqiuvalent to 
that $\Cm\At\A\in \bold S\Nr_n\CA_{n+3}$.   

The algebraic methodology and results used here are tailored to studying Vaught's theorem for the so--called {\it clique--guarded fragments} of $n$--variable first order 
logic, which is an equivalent formulation of its packed fragments. Here one {\it guards semantics} so that the semantics considered  are no longer 
the usual Tarskian semantics. Tarskian  semantics  become rather a special or/ and  limiting case. 
Semantics are {\it locally guarded} where `locality' is measured by a parameter $m$ with $2<n<m \leq \omega$.
This $m$--locally well behaved semantics provides an {\it optimal fit} with the syntactically defined variety $\bold S\Nr_n\CA_m$; reflected 
in a  completeness theorem. Any $\A\in \bold S\Nr_n\CA_m$ posses an  {\it $m$--flat 
representation.}  When $m=\omega$, this $m$--flat repesentation can be identified with an ordinary one. 
Henkin's neat embedding theorem stating that  $\bold S\Nr_n\CA_{\omega}=\RCA_n$ is then the limiting case dealing with a natural (algebraizable) extension 
of Godel's completeness theorem 
with respect to usual Tarskian semantics. More succintly, algebras in the variety $\bold S\Nr_n\CA_{\omega}$ 
can be represented on generalized cartesian spaces of dimension $n$.

The first part of the paper is dominated by negative results to the effect that Vaught's theorem 
fails in the situation at hand. Typically: 
{\it The non-principal type of co--atoms an atomic countable $L_n$ theory $T$ ($2<n<\omega$) cannot be omitted with respect to 
a notion of generalized semantics.} 
The semantics dealt with are what 
we call {\it $m$--clique--guarded} semantics, or simply clique guarded semantics.  
Algebraically: For $2<n<m<\omega$, we construct a countable and atomic $\A\in \RCA_n$ lacks  a {\it complete $m$--flat representation} or even a {\it complete $m$--square one}, 
which is a  substantialy weaker notion when $m<\omega$;  an $m$--flat model is $m$--square but the converse is not true.
Finer results are obtained, if we require that the algebra $\A$, in addition, is in $\Nr_n\CA_l$ for some $2<n<l<m$. In this case
the type $\Gamma$ will be realizable in every countable $m$--flat model of $T$, but cannot be isolated by a formula using $\leq l$ variables. 
By the Orey--Henkin omitting types theorem, 
being realizable in every $\omega$--flat (equivalently square) ordinary countable model, we know that $\Gamma$ is isolated by a formula.
So this formula, referred to as {\it a witness}, must use $>l$ variables.

\subsection{Guarding semantics} 

Fix $n<\omega$. One can view modal logics as fragments of first order logic. But on the other hand,  one can 
turn the glass around and give first order logic with $n$ variables a modal formalism, by viewing assignments as worlds, and 
giving the existential quantifier the most
prominent citizen in first order logic the following familar modal pattern (*) :
$$V, s\models \exists x_i\phi \Longleftrightarrow (\exists s)(s\equiv _i t)\&V, s\models \phi,$$
where $V\subseteq {}^n\M$ for some non--empty set $\M$ is the set of states or worlds and $s\in V$ is  a state.
Here existential quantifiers are viewed as modalities (or diamonds).

Guarded semantics are provided by {\it generalized models}, where assignments are allowed from arbitrary sets of $n$--ary sequences $V\subseteq {}^n\M$ (some non empty set $\M$). 
Semantics for the Boolean connectives are defined the usual way,
and the semantics for the existential quantifiers are as specified in (*).
Tarskian semantics becomes the special case when $V$ is `square', namely, $V={}^n\M.$

Guarding is expressed algebraically by {\it relativization}.  
Relativization is a technique that proved extremely potent  in obtaining positive results in both algebraic and modal logic. 
The technique originated 
from research by Istvan N\'emeti  dedicated to the class
of {\it relativized} set algebras, ${\sf Crs}_n$ for short, whose top elements are arbitrary sets of $n$--ary sequences
$V\subseteq {}^n\M$--rather than the whole of $^n\M$ like in cylindric set algebras of dimension $n$-- where $\M$ a non--empty set and $n<\omega$ is the dimension.
The concrete interpretation of the operations are defined like cylindric set algebras; 
Boolean meet is interpreted as intersection, cylindrifiers as projections, diagonal elements 
as  equality \ldots etc,  but  all such operations are {\it relativized} to the top element $V$. 

This type of research was initiated to sidestep resilient wild and unruly behaviour of 
the variety of representable cylindric algebras of dimension at least three, defined as subdirect products of set algebras of the same dimension, 
like undecidability of its equational theory and non-finite axiomatizability.
N\'emeti proved, in a seminal result,  that the universal theory of the variety ${\sf Crs}_n$ is decidable.
The corresponding multi--modal
logic exhibits nice modal behaviour and is regarded as the base for
proposing the so--called 
{\it guarded fragments} of first order logic by Andr\'eka et al. \cite{v}.

Dropping the historically ad-hoc assumption that set algebras should have ` squares' as their top 
elements proved immensely fruitful.
This way of {\it global guarding semantics} has led to the discovery of a whole landscape of multimodal logics having nice modal behaviour (like decidability) 
with the multimodal logic whose modal algebras are the class of relativized set algebras `at the bottom' 
and Tarskian semantics with its undesirable properties (like undecidability) is only the top of an iceberg.
Below the surface a treasure of nice multimodal logics was discovered.  
Relativized semantics  has led to many other nice modal
$n$--variable fragments of first order logic, like the 
loosely guarded, 
clique guarded (a disguised  formalism of packed fragments) 
of first order logic \cite{v}, \cite[Definitions 19.1, 19.2, 19.3,  p. 586-89]{HHbook}.

\subsection{Clique guarding semantics}
 
Fix $2<n<\omega$. Unless otherwise indicated, all theories considered in this paper are countable. Let $T$ be an $L_n$ theory in a certain signature. Let $\A$ be the Tarski--Lindenbaum quotient 
algebra $\Fm_{\sim_T}$ where the quotient is taken relative to the usual semantics, that is $\phi\sim_T \psi\iff T\models \phi\leftrightarrow \phi$.
In particular, $\A\in \RCA_n$. 
Assume that $n<m<\omega$. Let $\M$ be the base of a relativized representation of $\A$, that is, there exists an injective
homomorphism $f:\A\to \wp(V)$, where $V\subseteq {}^n\M$ and $\bigcup_{s\in V} \rng(s)=\M$. Here we identify 
notationally the set algebra with universe $\wp(V)$ with
its universe $\wp(V)$. 
{\it We say that $\M$ is a generalized model of $T$.}

We write $\M\models a(s)$ for $s\in f(a)$. Let  $\L(\A)^m$ be the first order signature using $m$ variables
and one $n$--ary relation symbol for each element in $A$ and $\L(\A)^m_{\infty, \omega}$ be the infinitary extension of $\L(\A)^m$ allowing infinite conjunctions. 
Then {\it an $n$--clique} is a set $C\subseteq M$ such
$(a_1,\ldots, a_{n-1})\in V=1^{\M}$
for distinct $a_1, \ldots, a_{n}\in C.$
Let
${\sf C}^m(\M)=\{s\in {}^m\M :\rng(s) \text { is an $n$--clique}\}.$
Then ${\sf C}^m(\M)$ is called the {\it $n$--Gaifman hypergraph of $\M$}, with the $n$--hyperedge relation $1^{\M}$.

\begin{enumarab}

\item The {\it clique guarded semantics $\models_c$} are defined inductively. For atomic formulas and Boolean connectives they are defined
like the classical case and for existential quantifiers
(cylindrifiers) they are defined as follows:
for $\bar{s}\in {}^m\M$, $i<m$, $\M, \bar{s}\models_c \exists x_i\phi$ $\iff$ there is a $\bar{t}\in {\sf C}^m(\M)$, $\bar{t}\equiv_i \bar{s}$ such that
$\M, \bar{t}\models \phi$.

\item We say that $\M$ is  {\it $m$--square} model of $T$,
if  $\bar{s}\in {\sf C}^m(\M), a\in \A$, $i<n$,
and   $l:n\to m$ is an injective map, $\M\models {\sf c}_ia(s_{l(0)},\ldots, s_{l(n-1)})$
$\implies$ there is a $\bar{t}\in {\sf C}^m(\M)$ with $\bar{t}\equiv _i \bar{s}$,
and $\M\models a(t_{l(0)}, \ldots, t_{l(n-1)})$.

\item $\M$ is said to be {\it (infinitary) $m$--flat} model of $T$ if  it is $m$--square and
for all $\phi\in (\L(\A)_{\infty, \omega}^m) \L(\A)^m$, 
for all $\bar{s}\in {\sf C}^m(\M)$, for all distinct $i,j<m$,
we have
$\M\models_c [\exists x_i\exists x_j\phi\longleftrightarrow \exists x_j\exists x_i\phi](\bar{s}).$
\end{enumarab}

Extending the definition the obvious way one defines an $m$--relativized representation 
$\M$ of $\A\in \CA_n$, whether $m$--square, or $m$--flat. 
Clique guarded fragments are closely related to {\it packed fragments} of first order logic as illustrated next.
Let $2<n<m<\omega$, $\A\in \CA_n$ and 
$\M$ be an $m$--flat representation of $\A$.
Then 
$$\M, s\models_c \phi  \iff\  \M,s\models {\sf packed}(\phi),$$
for all $s\in {\sf C}^n(\M)$ and every $\phi\in \L(\A)^n$, where $\sf packed(\phi$) denotes the translation of $\phi$ to the packed fragment \cite[Definition 19.3]{HHbook}.
In this sense, {\it the clique guraded fragments}, which are the $n$--variable fragments of first order order with clique 
guarded semantics, 
are an alternative formulation of {\it the $n$--variable 
packed fragments}  of first order logic \cite[\S 19.2.3]{HHbook}.

\subsection{Relativized versions of Vaught's theorem}

Fix $2<n\leq l<m\leq \omega$. 
We investigate the validity of a formula $\Phi(l, m)$, which is a relativized version of Vaught's theorem.
The first parameter $l$ measures in an exact sense 
how close we are to $L_{\omega, \omega}$. 
On the other hand, the second parameter $m$ measures how {\it locally well behaved} do we want our  models to be? 
The parameter $m$ measures {\it the degree of squareness} of the required models.  We look for atomic models that are $m$--square. 
(As indicated in the abstract, an $\omega$--square countable model 
is just an ordinary countable model and vice versa.) 

For parameters $n\leq l<m\leq \omega$ 
the formula $\Phi(l, m)$ says:  

{\sf There exists an atomic $\L_l$ countable and complete theory $T$, such that the non--principal type of co--atoms, 
cannot be omitted in an $m$--square model of $T$.}

The formula $\Phi(l, m)$ depends on $n$ which is fixed to be $2<n<\omega$; the reference to $n$ is omitted not to clutter notation. 
By $\L_{l}$ we understand 
an $n$--variable fragment of $L_{\omega, \omega}$ allowing all connectives that are definable by first order formulas using the first 
$l$ variables  with only the first $n$ occuring as 
free variables. This will be explained in more detail in the text, but for the time being one can view $\L_l$ as an $n$--variable logic 
having the same expressive power as $L_l$.   
 
By $T$ atomic, we mean 
that the Boolean reduct of Tarski--Lindenbaum quotent algebra $\Fm_T$ is atomic. The type $\Gamma$ is defined to be 
$\{\phi: (\neg \phi)_T\in \At\Fm_T\}$. The theory $T$ is complete means that for any sentence $\phi$ in the signature of $T$, 
$T\models \phi$ or $T\models \neg \phi$.
This, in turn,  is equivalent to that $\Fm_T(\in \RCA_n$) 
is simple (has no proper congruences). 

While $\Phi(l, m)$ says that for an $\L_l$ atomic theory, the non--principal type consisting of co--atoms may not be omitted in an $m$--square model,
Vaught's theorem for $L_{\omega, \omega}$ says the exact opposite: In an atomic countable theory, the family of non--principal 
type of co--atoms 
can, by the Orey--Henkin omitting types theorem, be omitted in a countable model (the atomic model).
It is therefore reasonable to let $\VT(l, m):=\neg \Phi(l, m)$ and read it as `Vaught's theorem holds at $l$ and $m$', stipulating 
that ${\sf VT}(\omega, \omega)$ is just Vaught's theorem for $L_{\omega, \omega}$:
Countable  atomic theories have countable atomic models. 

In this paper,  we argue that for the permitted values $n\leq l, m\leq \omega$, namely, 
for $n\leq l<m\leq \omega$  and $l=m=\omega$, 
$\VT(l, m)\iff l=m=\omega.$ 
From known algebraic results like non-atom--canonicity of $\RCA_n$ \cite{Hodkinson} and non-first order definability of the class 
of completely representable $\CA_n$s \cite{HH},  it can be easily inferred that
$\VT(n, \omega)$ is false, that is, Vaught's theorem fails for $L_n$ with respect to square Tarskian semantics. One can find countable atomic $L_n$ 
theories having no 
atomic models.
But we can go further. From sharper algebraic results, we prove  
many other special cases for specific values of $l$ and $m$ that support the last equivalence.
Two prominent cases are when $n\leq l<\omega$ and $m=\omega$ and 
when $l=n$ and $m\geq n+3$;  heuristically, when $l\to \omega$ and $m=\omega$ (the $x$ axis) and when $l=n$ and $m\to \omega$ (the $y$ axis).

{\it The first result is proved by showing that there exists an atomic, simple and countable $\A\in \Nr_n\CA_l\cap \RCA_n$ that is not completely representable, 
while the second follows from showing that
there exists an atomic, simple and countable $\A\in \RCA_n$ such that $\Cm\At\A$ lacks an $n+3$--square representation.
This implies, but is not  equivalent to,   that $\Cm{\At\A}\notin \bold S\Nr_n\CA_{n+3}$ and the last is equivalent to that $\Cm\At\A$ does not have an $n+3$--flat representation.}  
In particular,  the varieties $\bold S\Nr_n\CA_m$, for any $m\geq n+3$,  are not atom--canonical. 
Such results show that the omitting types theorem fails dramatically,  even if we 
`clique guard' semantics. 

Conversely, we prove positive omitting types theorems 
for $L_n$ with respect to Tarskian semantics, 
by imposing conditions on theories considered such as quantifier elimination 
and types  we hope to omit, such as  maximality (completeness). 
In such cases, we find, like in the classical Orey--Henkin omitting types theorems, a countable model omitting 
the given family of non--principal types.
It is well known that in the omitting types theorem, countability of the theory and non--principal 
types to be omitted, cannot be dispensed with, due to its intimate connection with the 
Baire category theorem for Stone spaces.  
Nevertheless, we will have occasion to provide models (not necessarly countable), omitting possibly 
uncountably many non--principal types in possibly uncountable $L_n$ 
theories. 
\subsection*{Layout}

\begin{itemize}

\item  For $2<n\leq l<m\leq \omega$, in $\S3$ we present a chain of implications starting from the existence of certain finite relation algebras
having so-called strong $l$--blur and no $m$--flat representations leading up to $\neg {\sf VT}(l, m)$, cf. theorem \ref{main}. Using such implications we prove
$\neg {\sf VT}(l, \omega)$ for $l<\omega$, and $\neg {\sf VT}(n, m)$ for $m\geq n+3$, cf. theorem \ref{can}. 
Our results are next presented in the framework of omitting types theorems for clique guarded fragments of $L_n$ cf. theorem \ref{OTT2}.

\item In $\S4$ we prove several positive omitting types theorems for $L_n$ ($2<n<\omega$) 
and for infinitary algebraizable (in the standard Blok--Pigozzi sense) extensions of $L_{\omega, \omega}$. Several
(counter)examples are presented, mostly in algebraic form, 
marking the boundaries of our results, cf. theorems \ref{i} and \ref{ii}.

\item In $\S5$ Vaught's theorem is approached algebraically, 
where the semantical notion of complete representability for any countable dimension is characterized via 
special neat embeddings, cf. theorem \ref{iii}, 
and 
several positive results on complete representability for reducts of polyadic equality algebras 
are obtained, theorems \ref{pa}, \ref{paa}, \ref{paaa}, and \ref{suf}. 

\item In $\S6$, in contrast to the negative results obtained in $\S3$ with respect to clique guarded semantics, positive results on Vaughts theorem and omitting types are obtained
for $L_n$ and $L_{\omega, \omega}$ with respect to {\it globally guarded semantics}, 
theorems \ref{one}, \ref{two} and \ref{three}. Here top elements of the modal set algebras of dimension $n\leq \omega$, are unions of 
$n$--dimensional cartesian spaces 
that are not necessarily disjoint. 

\item  In $\S7$, our treatment will be purely algebraic. We introduce algebraic counterparts of the modal 
fragments studied in $\S3$,  and elaborate on the metalogical results obtained from an algebraic perspective, cf. 
theorem \ref{main2} and corollary \ref{main3}.

\end{itemize}

\section{Preliminaries}

In our treatment of the notion of neat reducts, and the related one of neat embeddings, 
we follow the terminolgy and notation of \cite{HMT2, Sayedneat}.
The notion of {\it neat reducts} and the related one of {\it neat embeddings} are both important in algebraic logic for the
simple reason that both notions are very much tied
to the notion of representability, via the so--called neat embedding theorem of Henkin's which says that (for any ordinal 
$\alpha$),  we have $\sf RCA_{\alpha}=\bold S\Nr_{\alpha}\CA_{\alpha+\omega}$.
\begin{definition} 
Assume that $\alpha<\beta$ are ordinals and that 
$\B\in \CA_{\beta}$. Then the {\it $\alpha$--neat reduct} of $\B$, in symbols
$\Nr_{\alpha}\B$, is the
algebra obtained from $\B$, by discarding
cylindrifiers and diagonal elements whose indices are in $\beta\setminus \alpha$, and restricting the universe to
the set $Nr_{\alpha}B=\{x\in \B: \{i\in \beta: {\sf c}_ix\neq x\}\subseteq \alpha\}.$
\end{definition}

It is straightforward to check that $\Nr_{\alpha}\B\in \CA_{\alpha}$. 
If $\A\in \CA_\alpha$ and $\A\subseteq \mathfrak{Nr}_\alpha\B$, with $\B\in \CA_\beta$, then we say that $\A$ {\it neatly embeds} in $\B$, and 
that $\B$ is a {\it $\beta$--dilation of $\A$}, or simply a {\it dilation} of $\A$ if $\beta$ is clear 
from context. 
For relation algebra reducts we follow \cite{HHbook2}. 
If $\beta>2$, $\B\in \CA_{\beta}$ and 
$\R=\Ra\B$ is the  algebra with $\RA$ signature defined like in \cite[Definition 5.3.7]{HMT2}, 
we refer to $\B$ also as a $\beta$--dilation of $\R$; 
if $\beta>3$, then it is known that $\Ra\B\in \RA$ \cite[Theorem 5.3.8]{HMT2}.

\subsection{Atom structures and atom--canonicity}
 
We recall the notions of {\it atom structures} and {\it complex algebra} in the framework of Boolean algebras 
with operators of which $\CA$s  are a special case \cite[Definition 2.62, 2.65]{HHbook}. 
The action of the non--Boolean operators in a completely additive (where operators distribute over arbitrary joins componentwise) 
atomic Boolean algebra with operators, $(\sf BAO)$ for short, is determined by their behavior over the atoms, and
this in turn is encoded by the {\it atom structure} of the algebra.

\begin{definition}\label{definition}(\textbf{Atom Structure})
Let $\A=\langle A, +, -, 0, 1, \Omega_{i}:i\in I\rangle$ be
an atomic $\sf BAO$ with non--Boolean operators $\Omega_{i}:i\in I$. Let
the rank of $\Omega_{i}$ be $\rho_{i}$. The \textit{atom structure}
$\At\A$ of $\A$ is a relational structure
$$\langle At\A, R_{\Omega_{i}}:i\in I\rangle$$
where $At\A$ is the set of atoms of $\A$ 
and $R_{\Omega_{i}}$ is a $(\rho(i)+1)$-ary relation over
$At\A$ defined by
$$R_{\Omega_{i}}(a_{0},
\cdots, a_{\rho(i)})\Longleftrightarrow\Omega_{i}(a_{1}, \cdots,
a_{\rho(i)})\geq a_{0}.$$
\end{definition}
\begin{definition}(\textbf{Complex algebra})
Conversely, if we are given an arbitrary first order structure
$\mathcal{S}=\langle S, r_{i}:i\in I\rangle$ where $r_{i}$ is a
$(\rho(i)+1)$-ary relation over $S$, called an {\it atom structure}, we can define its
\textit{complex
algebra}
$$\Cm(\mathcal{S})=\langle \wp(S),
\cup, \setminus, \phi, S, \Omega_{i}\rangle_{i\in
I},$$
where $\wp(S)$ is the power set of $S$, and
$\Omega_{i}$ is the $\rho(i)$-ary operator defined
by$$\Omega_{i}(X_{1}, \cdots, X_{\rho(i)})=\{s\in
S:\exists s_{1}\in X_{1}\cdots\exists s_{\rho(i)}\in X_{\rho(i)},
r_{i}(s, s_{1}, \cdots, s_{\rho(i)})\},$$ for each
$X_{1}, \cdots, X_{\rho(i)}\in\wp(S)$.
\end{definition}
It is easy to check that, up to isomorphism,
$\At(\Cm(\mathcal{S}))\cong\mathcal{S}$ alway. If $\A$ is
finite then of course
$\A\cong\Cm(\At\A)$. 
For algebras $\A$ and $\B$ having the same signature expanding that of Boolean algebras, 
we say that $\A$ is dense in $\B$ if $\A\subseteq \B$ and for all non--zero $b\in \B$, there is a non--zero 
$a\in A$ such that $a\leq b$.
An atom structure will be denoted by $\bf At$.  An atom structure $\bf At$ has the signature of $\CA_\alpha$, $\alpha$ an ordinal, 
if  $\Cm\bf At$ has the signature of $\CA_\alpha$. 

\begin{definition}\label{canonical} 
Let $V$ be a variety of $\sf BAO$s. Then $V$ is {\it atom--canonical} if whenever $\A\in V$ and $\A$ is atomic, then $\mathfrak{Cm}\At\A\in V$.
The {\it  \de\ completion} of  $\A\in \V$, is the unique (up to isomorphisms that fix $\A$ pointwise)  complete  
$\B\in \K_n$ such that $\A\subseteq \B$ and $\A$ is {\it dense} in $\B$. 
\end{definition}

From now on fix $2<n<\omega$. If $\A\in \CA_n$ is atomic,  then 
the complex algebra of its atom structure, in symbols $\mathfrak{Cm}\At\A$ 
is the {\it \de\ completion of $\A$.}
If $\A\in \CA_n$, then its atom structure will be denoted by $\At\A$ with domain the set of atoms of $\A$ denoted by $At\A$.

{\it Non atom--canonicity} can be proved 
by finding {\it weakly representable atom structures} that are not {\it strongly representable}.

\begin{definition} An atom structure $\bf At$ of dimension $n$ is {\it weakly representable} if there is an atomic 
$\A\in \RCA_{n}$ such that $\At\A=\bf At$.  The atom structure  $\bf At$ is {\it strongly representable} if for all $\A\in \CA_{n}$, 
$\At\A={\bf At} \implies \A\in {\sf RCA}_n$.
\end{definition}
These two notions (strong and weak representability) do not coincide for cylindric algebras as proved by Hodkinson \cite{Hodkinson}.
In \cite{mlq} we generalize Hodkinson's result by showing that for there are two atomic $\CA_n$s 
sharing the same atom structure, one is representable and the other its \de\ completion is even outside 
$\bold S{\sf N}r_n\CA_{n+3}(\supsetneq {\sf RCA}_n$). In particular, there is a complete algebra outside $\bold S{\sf Nr}_n\CA_{n+3}$ 
having a dense representable  subalgebra, so that 
$\bold S{\sf Nr}_n\CA_{n+3}$ is not atom--canonical. This construction will be used in theorem \ref{can}.
 
\subsection{Networks and games} 

To define certain games to be used in the sequel,  
we recall the notions of {\it atomic networks} and {\it atomic games} \cite{HHbook, HHbook2}. 
Let $i<n$. For $n$--ary sequences $\bar{x}$ and $\bar{y}$ $\iff \bar{y}(j)=\bar{x}(j)$ for all $j\neq i$
\begin{definition}\label{game} Fix finite $n>1$. 
\begin{enumarab}

\item An {\it $n$--dimensional atomic network} on an atomic algebra $\A\in \QEA_n$  is a map $N: {}^n\Delta\to  \At\A$, where
$\Delta$ is a non--empty set of {\it nodes}, denoted by $\nodes(N)$, satisfying the following consistency conditions for all $i<j<n$: 
\begin{itemize}
\item If $\bar{x}\in {}^n\nodes(N)$  then $N(\bar{x})\leq {\sf d}_{ij}\iff x_i=x_j$,
\item If $\bar{x}, \bar{y}\in {}^n\nodes(N)$, $i<n$ and $x\equiv_i y$, then  $N(\bar{x})\leq {\sf c}_iN(\bar{y})$,
\end{itemize}
For $n$--dimensional atomic networks $M$ and $N$,  we write $M\equiv_i N\iff M(y)=N(y)$ for all $\bar{y}\in {}^{n}(n\sim \{i\})$.

\item   Assume that $\A\in \QEA_n$ is  atomic and that $m, k\leq \omega$. 
The {\it atomic game $G^m_k(\At\A)$, or simply $G^m_k$}, is the game played on atomic networks
of $\A$ using $m$ nodes and having $k$ rounds \cite[Definition 3.3.2]{HHbook2}, where
\pa\ is offered only one move, namely, {\it a cylindrifier move}: 

Suppose that we are at round $t>0$. Then \pa\ picks a previously played network $N_t$ $(\nodes(N_t)\subseteq m$), 
$i<n,$ $a\in \At\A$, $x\in {}^n\nodes(N_t)$, such that $N_t(\bar{x})\leq {\sf c}_ia$. For her response, \pe\ has to deliver a network $M$
such that $\nodes(M)\subseteq m$,  $M\equiv _i N$, and there is $\bar{y}\in {}^n\nodes(M)$
that satisfies $\bar{y}\equiv _i \bar{x}$ and $M(\bar{y})=a$.  

\item We write $G_k(\At\A)$, or simply $G_k$, for $G_k^m(\At\A)$ if $m\geq \omega$.

\item The $\omega$--rounded game $\bold G^m(\At\A)$ or simply $\bold G^m$ is like the game $G_{\omega}^m(\At\A)$ 
except that \pa\ has the option 
to reuse the $m$ nodes in play.
\end{enumarab}
\end{definition}

\section{Vaught's theorem holds  for $L_n$ with respect to clique guarded semantics}

Throughout this section, unless otherwise indicated $n$ is a finite ordinal $>2$. 
Results in algebraic logic are most attractive when they lend themselves to (non--trivial) applications in (first order) logic.
In this section, we adress algebraically a restricted version of the 
omitting types theorems, namely, Vaught's theorem on existence of atomic models for atomic theories, in the framework of 
the {\it clique guarded} $n$--variable fragments
of first order logic.  In the introduction, 
we already  defined the notion of {\it clique guarded semantics}
which we extend the obvious way to all algebras (not just Tarski-Lindenbaum quotient algebras corresponding to $L_n$ theories). 

\begin{definition} Let $\M$ be the base of a relativized representation of $\A\in \CA_n$ witnessed by an injective
homomorphism $f:\A\to \wp(V)$, where $V\subseteq {}^n\M$ and $\bigcup_{s\in V} \rng(s)=\M$. 
We write $\M\models a(s)$ for $s\in f(a)$. Let  $\L(\A)^m$ be the first order signature using $m$ variables
and one $n$--ary relation symbol for each element in $A$. Let ${\sf C}^m(\M)$ be the 
$n$--Gaifman hypergraph, and let $\models_c$ be the clique guarded semantics 
as defined in the introduction.
\begin{itemize}

\item We say that $\M$ is  {\it an $m$--square representation} of $\A$,
if  for all $\bar{s}\in {\sf C}^m(\M), a\in \A$, $i<n$,
and   injective map $l:n\to m$, whenever $\M\models {\sf c}_ia(s_{l(0)},\ldots, s_{l(n-1)})$, then there is a $\bar{t}\in {\sf C}^m(\M)$ with $\bar{t}\equiv _i \bar{s}$,
and $\M\models a(t_{l(0)}, \ldots, t_{l(n-1)})$.

\item We say that $\M$ is {\it a complete $m$--square representation of $\A$ via $f$}, or simply a complete representation of $\A$ if 
$f(\sum X)=\bigcup_{x\in X} f(x)$, for all 
$X\subseteq \A$ for which $\sum X$ exists.

\item We say that $\M$ is an {\it (infinitary) $m$--flat representation} of $\A$ if  it is $m$--square and
for all $\phi\in (\L(\A)_{\infty, \omega}^m) \L(\A)^m$, 
for all $\bar{s}\in {\sf C}^m(\M)$, for all distinct $i,j<m$,
$\M\models_c [\exists x_i\exists x_j\phi\longleftrightarrow \exists x_j\exists x_i\phi](\bar{s})$.
Complete representability is defined like for squareness.
\end{itemize}
\end{definition}

It is straightforward to show, like in the classical case, 
that $\A$ has a complete $m$--square complete representation $\M$ via $f$ 
$\iff$ $\A$ is atomic and $f$ is {\it atomic} in the sense 
that $\bigcup_{x\in \At\A} f(x)=1^{\M}$.

Observe that if $T$ is an $L_n$ theory, then $\M$ is  an {\it $m$--square model of $T$} (as defined in the introduction) $\iff \M$ is an {\it $m$--square 
representation of the Tarski--Lindenbaum quotient algebra $\Fm_T$.}

We will construct cylindric agebras from atomic relation algebras that posses  {\it cylindric basis} using so 
called blow up and blur constructions, an indicative term introduced in \cite{ANT}.

\begin{definition} Let $\R$ be an atomic  relation algebra.  An {$n$--dimensional basic matrix}, or simply a matrix  
on $\R$, is a map $f: {}^2n\to \At\R$ satsfying the 
following two consistency 
conditions $f(x, x)\leq \Id$ and $f(x, y)\leq f(x, z); f(z, y)$ for all $x, y, z<n$. For any $f, g$ basic matrices
and $x, y<m$ we write $f\equiv_{xy}g$ if for all $w, z\in m\setminus\set {x, y}$ we have $f(w, z)=g(w, z)$.
We may write $f\equiv_x g$ instead of $f\equiv_{xx}g$.  
\end{definition}
\begin{definition}\label{b}
An {\it $n$--dimensional cylindric basis} for an atomic relaton algebra 
$\R$ is a set $\cal M$ of $n$--dimensional matrices on $\R$ with the following properties:
\begin{itemize}
\item If $a, b, c\in \At\R$ and $a\leq b;c$, then there is an $f\in {\cal M}$ with $f(0, 1)=a, f(0, 2)=b$ and $f(2, 1)=c$
\item For all $f,g\in {\cal M}$ and $x,y<n$, with $f\equiv_{xy}g$, there is $h\in {\cal M}$ such that
$f\equiv_xh\equiv_yg$. 
\end{itemize}
\end{definition}
For the next lemma, we refer the reader to \cite[Definition 12.11]{HHbook} 
for the definition of hyperbasis for relation algebras.
For a relation algebra $\R$, we let $\R^+$ denotes its canonical extension.

\begin{lemma}\label{i} Let $\R$ be  a relation algebra and $3<n<\omega$.  Then the following hold:
\begin{enumerate}
\item $\R^+$ has an $n$--dimensional infinite basis $\iff\ \R$ has an infinite $n$--square representation.

\item $\R^+$ has an $n$--dimensional infinite hyperbasis $\iff\ \R$ has an infinite $n$--flat representation.
\end{enumerate}
\end{lemma}
\begin{proof} \cite[Theorem 13.46, the equivalence $(1)\iff (5)$ for basis, and the equivalence $(7)\iff (11)$ for hyperbasis]{HHbook}.
\end{proof} 

One can construct a $\CA_n$ in a natural way from an $n$--dimensional cylindric basis which can be viewed as an atom structure of a $\CA_n$
(like in \cite[Definition 12.17]{HHbook} addressing hyperbasis).
For an atomic  relation algebra $\R$ and $l>3$, we denote by ${\sf Mat}_n(\At\R)$ the set of all $n$--dimensional basic matrices on $\R$.
${\sf Mat}_n(\At\R)$ is not always an $n$--dimensional cylindric basis, but sometimes it is,
as will be the case described next. On the other 
hand, ${\sf Mat}_3(\At\R)$ is always a  $3$--dimensional cylindric basis; a result of Maddux's, so that 
$\Cm{\sf Mat}_3(\At\R)\in \CA_3$.

The following definition to be used in the sequel is taken from \cite{ANT}:
\begin{definition}\label{strongblur}\cite[Definition 3.1]{ANT}
Let $\R$ be a relation algebra, with non--identity atoms $I$ and $2<n<\omega$. Assume that  
$J\subseteq \wp(I)$ and $E\subseteq {}^3\omega$.
\begin{enumerate}
\item We say that $(J, E)$  is an {\it $n$--blur} for $\R$, if $J$ is a {\it complex $n$--blur} defined as follows:   
\begin{enumarab}
\item Each element of $J$ is non--empty,
\item $\bigcup J=I,$
\item $(\forall P\in I)(\forall W\in J)(I\subseteq P;W),$
\item $(\forall V_1,\ldots V_n, W_2,\ldots W_n\in J)(\exists T\in J)(\forall 2\leq i\leq n)
{\sf safe}(V_i,W_i,T)$, that is there is for $v\in V_i$, $w\in W_i$ and $t\in T$,
we have
$v;w\leq t,$ 
\item $(\forall P_2,\ldots P_n, Q_2,\ldots Q_n\in I)(\forall W\in J)W\cap P_2;Q_n\cap \ldots P_n;Q_n\neq \emptyset$.
\end{enumarab}
and the tenary relation $E$ is an {\it index blur} defined  as 
in item (ii) of \cite[Definition 3.1]{ANT}.

\item We say that $(J, E)$ is a {\it strong $n$--blur}, if it $(J, E)$ is an $n$--blur,  such that the complex 
$n$--blur  satisfies:
$$(\forall V_1,\ldots V_n, W_2,\ldots W_n\in J)(\forall T\in J)(\forall 2\leq i\leq n)
{\sf safe}(V_i,W_i,T).$$ 
\end{enumerate}
\end{definition}

One formulation of the Orey--Henkin omitting types theorem is. Given a countable consistent first order  theory $T$, 
and a type $\Gamma$ (a set of formulas containing only finitely many free variables), if $\Gamma$ is realizable in every model of $T$, then $\Gamma$ 
has to be isolated, in the sense that there exists a formula $\psi$ consistent with $T$ such that 
$T\models \psi\to \theta$ for all $\theta\in \Gamma$. We call such a formula a {\it a witness}.

Let $n\leq l<m\leq \omega$ and denote $L_{\omega, \omega}$ restricted to $n$ variables by $L_n$. 
Now we formulate formulas that are essentially equivalent to the formulas $\Phi(l, m)$ encountered in the introduction. 
Consider the formula $\Psi(l, m)$: 

{\it There is an atomic, countable and complete $L_n$ theory $T$, such that the type $\Gamma$ consisting of co--atoms is realizable 
in every $m$-- square model, but any formula (witness) isolating this type has
to contain more than $l$ variables.}

For $m<\omega$, $\Psi(l, m)_f$ is obtained from $\Psi(l, m)$ by replacing square by flat. 
$\Psi(l, \omega)_f$ is defined to be $\Psi(l, \omega)$ where the sought for atomic model is an ordinary countable one.

\begin{theorem} \label{main} Let $2<n\leq l<m\leq \omega$. Then
every item implies the immediately following one. 
\begin{enumerate}

\item There exists a finite relation algebra $\R$ algebra with a strong $l$--blur and no infinite $m$--dimensional hyperbasis, 
\item There is a countable atomic $\A\in \Nr_n\CA_l\cap \CA_n$ such that $\Cm\At\A$ does not have an 
$m$--flat representation,

\item There is a countable atomic $\A\in \Nr_n\CA_l\cap \RCA_n$ such that $\A$ has no complete $m$--flat representation,
\item $\Psi(l, m)_f$ is true,

\item $\Psi(l', m')_f$ is true for any $l'\leq l$ and $m'\geq m$.
\end{enumerate}
The same implications hold upon replacing infinite $m$--dimensional hyperbasis by $m$--dimensional basis 
and $m$--flat by $m$--square. Furthermore, in the new chain of implications every item implies the corresponding item in 
theorem \ref{main}. In particular, 
$\Psi(l, m)\implies \Psi(l, m)_f$.  
\end{theorem}

\begin{proof} 

$(1)\implies (2)$: Let $\R$ be as in the hypothesis with strong $l$--blur $(J, E)$. We use the technique in \cite{ANT}. The idea is to `blow up and blur' $\R$ in place of the Maddux algebra 
$\mathfrak{E}_k(2, 3)$ dealt with in \cite[Lemma 5.1]{ANT}   (where $k<\omega$ is the number of non--identity atoms 
and it depends on $l$). 

Let $3<n\leq l$.  The relation algebra $\R$ is blown up by splitting all of the atoms each to infinitely many.
$\R$ is blurred by using a finite set of blurs (or colours) $J$. This can be expressed by the product ${\bf At}=\omega\times \At \R\times J$,
which will define an infinite atom structure of a new
relation algebra, denoted by $\At$ on \cite[p.73]{ANT}. (One can view such a product as a ternary matrix with $\omega$ rows, and for each fixed $n\in \omega$,  we have the rectangle
$\At \R\times J$.)
Then two partitions are defined on $\bf At$, call them $P_1$ and $P_2$.
Composition is re-defined on this new infinite atom structure; it is induced by the composition in $\R$, and a ternary relation $E$
on $\omega$, that `synchronizes' which three rectangles sitting on the $i,j,k$ $E$--related rows compose like the original algebra $\R$.
This relation is definable in the first order structure $(\omega, <)$.
The first partition $P_1$ is used to show that $\R$ embeds in the complex algebra of this new atom structure, namely $\Cm \bf At$, 
The second partition $P_2$ divides $\bf At$ into {\it finitely many (infinite) rectangles}, each with base $W\in J$,
and the term algebra denoted in \cite{ANT} by ${\Bb}(\R, J, E)$ over $\bf At$, consists of the sets that intersect co--finitely with every member of this partition.

On the level of the term algebra $\R$ is blurred, so that the embedding of the small algebra into
the complex algebra via taking infinite joins, do not exist in the term algebra for only finite and co--finite joins exist
in the term algebra. 
The algebra ${\Bb}(\R, J, E)$ with atom structure $\bf At$ is representable using the finite number of blurs. These correspond to non--principal ultrafilters
in the Boolean reduct, which are necessary to
represent this term algebra, for if used alone, the principal ultrafilters alone would give a complete representation,
hence a representation of the complex algebra and this is impossible.
Because $(J, E)$ is a strong $l$--blur, the atom structure $\bf At$ has an $l$--dimensional cylindric basis, 
${\bf At}_{ca}$, namely, ${\sf Mat}_l(\bf At)$
Then there is an algebra $\C$ denoted on \cite[Top of p. 78]{ANT} by ${\Bb}_l(\R, J, E)$ such 
that $\Tm{\sf Mat}_l({\bf At})\subseteq \C\ \subseteq \Cm{\sf Mat}_l(\bf At)$ 
and $\At\C$ is weakly representable.

Now take $\A=\Bb_n(\R, J, E)$. We claim that $\A$ is as required.  Since $\R$ has a strong $l$--blur, then 
$\A\cong \Nr_n\Bb_l(\R, J, E)$ as proved in \cite[item (3) p. 80]{ANT}, so $\A\in \RCA_n\cap \Nr_n\CA_l$. But 
we claim that $\C=\Cm\At\A$  does not an $m$--flat representation.  
Assume for contradicton that $\C$ does have an $m$--flat representation $\M$.
Then $\M$ is infinite of course.  Since $\R$ embeds into $\Bb(\R, J, E)$ which in turn embeds into $\Ra\Cm\At\A$, then 
$\R$ has an $m$--flat representation with base $\M$. But then $\R$ has an  
infinite $m$--dimensional hyperbasis, 
which is contrary to assumption.

$(2)\implies (3)$: A complete $m$--flat representation of (any) 
$\B\in \CA_n$ induces an $m$--flat representation of $\Cm\At\B$. 
To see why, assume that $\B$ has an $m$--flat  complete representable via $f:\B\to \D$, where $\D=\wp(V)$ and 
the base of the represenation $\M=\bigcup_{s\in V} \rng(s)$ is $m$--flat. Let $\C=\Cm\At\B$.
For $c\in C$, let $c\downarrow=\{a\in \At\C: a\leq c\}=\{a\in \At\B: a\leq c\}$; the last equality holds because 
$\At\B=\At\C$. Define, representing $\C$,  
$g:\C\to \D$ by $g(c)=\sum_{x\in c\downarrow}f(x).$ The map $g$ is well defined because $\C$ is complete so arbitrary suprema exist in $\C$. 
Furthermore, it can be easily checked that $g$ is a homomorphism into $\wp(V)$ having base $\M$ (basically because by assumption $f$ is a homomorphism).

$(3)\implies (4)$. By \cite[\S 4.3]{HMT2}, we can (and will) assume that $\A= \Fm_T$ for a countable, atomic theory $L_n$ theory $T$.  
Let $\Gamma$ be the $n$--type consisting of co--atoms of $T$. Then $\Gamma$ is realizable in every $m$--flat model, for if $\M$ is an $m$--flat model omitting 
$\Gamma$, then $\M$ would be the base of a complete $m$--flat  representation of $\A$.
But $\A\in {\sf Nr}_n\CA_l$,  so  using exactly the same  (terminology and) argument in \cite[Theorem 3.1]{ANT} we get that  
any witness isolating $\Gamma$  needs more 
than $l$--variables. 

$(4)\implies (5)$ follows from the definitions.
\end{proof}
The following lemma is proved in \cite{mlq}:
\begin{lemma}\label{n} 
Assume that $2<n<m<\omega$ and $\A\in \CA_n$ be atomic. 
Then the following hold:
\begin{enumerate}
\item $\A\in \bold S_c\Nr_n\CA_m\implies$ \pa\ has a \ws\ in $\bold G^m(\At\A)$,
\item If $\A$ is finite and \pa\ has a \ws\ in $G^m_{\omega}(\At\A)$, 
then $\A\notin \bold S\Nr_n\CA_m$.
\end{enumerate} 
\end{lemma}
In the next theorem we use a rainbow constructions for $\CA$s \cite{HH, HHbook2}.
Fix $2<n<\omega$. Given relational structures 
$\sf G$ (the greens) and $\sf R$ (the reds) the rainbow 
atom structure of a $\QEA_n$  consists of equivalence classes of surjective maps $a:n\to \Delta$, where $\Delta$ is a coloured graph.

A coloured graph is a complete graph labelled by the rainbow colours, the greens $\g\in \sf G$,  reds $\r\in \sf R$, and whites; and some $n-1$ tuples are labelled by `shades of yellow'.
In coloured graphs certain triangles are not allowed for example all green triangles are forbidden.
A red triple $(\r_{ij}, \r_{j'k'}, \r_{i^*k^*})$ $i,j, j', k', i^*, k^*\in \sf R$ is not allowed,  unless $i=i^*,\; j=j'\mbox{ and }k'=k^*$, 
in which case we say that the  red indices match, cf.\cite[4.3.3]{HH}.
The equivalence relation relates two such maps $\iff$  they essentially define the same graph \cite[4.3.4]{HH}. We let $[a]$ denote the equivalence class containing $a$. 
The accessibilty (binary relations) corresponding to 
cylindric operations are like in \cite{HH}. 

Special coloured graphs typically used by \pa\ during implementing his \ws\ are called {\it cones}: 
{\it Let $i\in {\sf G}$, and let $M$ be a coloured graph consisting of $n$ nodes
$x_0,\ldots,  x_{n-2}, z$. We call $M$ {\it an $i$ - cone} if $M(x_0, z)=\g_0^i$
and for every $1\leq j\leq n-2$, $M(x_j, z)=\g_j$,
and no other edge of $M$
is coloured green.
$(x_0,\ldots, x_{n-2})$
is called  the {\bf base of the cone}, $z$ the {\bf apex of the cone}
and $i$ the {\bf tint of the cone}.}

For $2<n<\omega$, we use the graph version of the games $G_{\omega}^m(\beta)$ and $\bold G^m(\beta)$ where $\beta$ is a $\CA_n$ 
rainbow atom structure, cf. \cite[4.3.3]{HH}. The (complex) rainbow algebra based on $\sf G$ and $\sf R$ is denoted by $\A_{\sf G, R}$.
The dimension $n$ will always be clear from context.

\begin{theorem}\label{can}
Let $n\leq l<\omega$. Then $\Psi(l,\omega)$ and $\Psi(n, n+3)$ are true. In particular, from both statements we get that $\Psi(n, \omega)$ is true.
\end{theorem}
\begin{proof}
{\bf Proving $\Psi(l, \omega)$ \cite{ANT}:}. Let $l\geq 2n-1$, $k\geq (2n-1)l$, $k\in \omega$. 
One takes the finite
integral relation algebra $\R_l={\mathfrak E}_k(2, 3)$ 
where $k$ is the number of non-identity atoms in
$\R_l$. Then $\R_l$ has a strong $l$--blur, and it can only be represented 
on infinite basis \cite{ANT}.  The rest follows using the same reasoning in the first item of the last 
theorem.

{\bf Proof for flatness:} We give first a relatively simple 
proof for $\Psi(n, n+3)_f$. Then we give another more involved proof for squareness (and flatness). 
Both proofs use a rainbow construction. 

By theorem \ref{main} 
it suffices to show that there an algebra with countably many atoms 
outside $\bold S_c\Nr_n\CA_{\omega}$.

Take the a rainbow--like $\CA_n$, call it $\C$, based on the ordered structure $\Z$ and $\N$.
The reds ${\sf R}$ is the set $\{\r_{ij}: i<j<\omega(=\N)\}$ and the green colours used 
constitute the set $\{\g_i:1\leq i <n-1\}\cup \{\g_0^i: i\in \Z\}$. 
In complete coloured graphs the forbidden triples are like 
the usual rainbow constructions based on $\Z$ and $\N$,   
but now 
the triple  $(\g^i_0, \g^j_0, \r_{kl})$ is also forbidden if $\{(i, k), (j, l)\}$ is not an order preserving partial function from
$\Z\to\N$.

We show that  \pa\ has a \ws\ in the graph version of the game 
$\bold G^{n+3}(\At\C)$ played on coloured graphs \cite{HH}.
The rough idea here, is that, as is the case with \ws's of \pa\ in rainbow constructions, 
\pa\ bombards \pe\ with cones having distinct green tints demanding a red label from \pe\ to appexes of succesive cones.
The number of nodes are limited but \pa\ has the option to re-use them, so this process will not end after finitely many rounds.
The added order preserving condition relating two greens and a red, forces \pe\ to choose red labels, one of whose indices form a decreasing 
sequence in $\N$.  In $\omega$ many rounds \pa\ 
forces a win, 
so by lemma \ref{n}, $\C\notin \bold S_c{\sf Nr}_n\CA_{n+3}$.

More rigoroudly, \pa\ plays as follows: In the initial round \pa\ plays a graph $M$ with nodes $0,1,\ldots, n-1$ such that $M(i,j)=\w_0$
for $i<j<n-1$
and $M(i, n-1)=\g_i$
$(i=1, \ldots, n-2)$, $M(0, n-1)=\g_0^0$ and $M(0,1,\ldots, n-2)=\y_{\Z}$. This is a $0$ cone.
In the following move \pa\ chooses the base  of the cone $(0,\ldots, n-2)$ and demands a node $n$
with $M_2(i,n)=\g_i$ $(i=1,\ldots, n-2)$, and $M_2(0,n)=\g_0^{-1}.$
\pe\ must choose a label for the edge $(n+1,n)$ of $M_2$. It must be a red atom $r_{mk}$, $m, k\in \N$. Since $-1<0$, then by the `order preserving' condition
we have $m<k$.
In the next move \pa\ plays the face $(0, \ldots, n-2)$ and demands a node $n+1$, with $M_3(i,n)=\g_i$ $(i=1,\ldots, n-2)$,
such that  $M_3(0, n+2)=\g_0^{-2}$.
Then $M_3(n+1,n)$ and $M_3(n+1, n-1)$ both being red, the indices must match.
$M_3(n+1,n)=r_{lk}$ and $M_3(n+1, r-1)=r_{km}$ with $l<m\in \N$.
In the next round \pa\ plays $(0,1,\ldots n-2)$ and re-uses the node $2$ such that $M_4(0,2)=\g_0^{-3}$.
This time we have $M_4(n,n-1)=\r_{jl}$ for some $j<l<m\in \N$.
Continuing in this manner leads to a decreasing
sequence in $\N$. By the previous theorem we are through.

{\bf Proof for squareness:} \cite{mlq}. 
It also suffices to show by the previous theorem that there is a countable atomic $\A\in \RCA_n$ 
such that $\Cm\At\A\notin \bold S\Nr_n\CA_{n+3}$.
Take the finite rainbow cylindric algebra $R(\Gamma)$
as defined in \cite[Definition 3.6.9]{HHbook2},
where $\Gamma$ (the reds) is taken to be the complete irreflexive graph $m$, and the greens
are  $\{\g_i:1\leq i<n-1\}
\cup \{\g_0^{i}: 1\leq i\leq n+1\}$ so that $\sf G$ is the complete irreflexive graph $n+1$.

Call this finite rainbow $n$--dimensional cylindric algebra, based on ${\sf G}=n+1$ and ${\sf R}=n$,
$\CA_{n+1, n}$  and denote its finite atom structure by $\bf At_f$.
One  then replaces each  red colour
used in constructing  $\CA_{n+1, n}$ by infinitely many with superscripts from $\omega$, 
getting a weakly representable atom structure $\bf At$, that is,
the term algebra $\Tm\bf At$ is representable.
The resulting atom structure (with $\omega$--many reds),  call it $\bf At$, 
is the rainbow atom structure that is like the atom structure of the (atomic set) algebra denoted by $\A$ in \cite[Definition 4.1]{Hodkinson} except that we have $n+1$ greens
and not infinitely many as is the case in \cite{Hodkinson}.
Everything else is the same. In particular, the rainbow signature \cite[Definition 3.6.9]{HHbook2} now consists of $\g_i: 1\leq i<n-1$, $\g_0^i: 1\leq i\leq n+1$,
$\w_i: i<n-1$,  $\r_{kl}^t: k<l< n$, $t\in \omega$,
binary relations, and $n-1$ ary relations $\y_S$, $S\subseteq n+1$.
There is a shade of red $\rho$; the latter is a binary relation that is {\it outside the rainbow signature}.
But $\rho$ is used as a  label for  coloured graphs built during a `rainbow game', and in fact, \pe\ can win the rainbow $\omega$--rounded game
and she builds an $n$--homogeneous (coloured graph) model $M$ as indicated in the above outline by using $\rho$ when
she is forced a red \cite[Proposition 2.6, Lemma 2.7]{Hodkinson}.
Then $\Tm\At$ is representable as a set algebra with unit $^nM$; this can be proved exactly as in \cite{Hodkinson}.
In fact, $\Tm\bf At\subseteq \A$, with $\A$ as described in the preceding outline.

We next embed $\CA_{n+1, n}$ into  the complex algebra $\Cm\bf At$, the \de\ completion of $\Tm\bf At$.
Let ${\sf CRG}_f$ denote  the class of coloured graphs on 
$\bf At_f$ and $\sf CRG$ be the class of coloured graph on $\bf At$. We 
can assume that  ${\sf CRG}_f\subseteq \sf CRG$.
Write $M_a$ for the atom that is the (equivalence class of the) surjection $a:n\to M$, $M\in \sf CGR$.
Here we identify $a$ with $[a]$; no harm will ensue.
We define the (equivalence) relation $\sim$ on $\At$ by
$M_b\sim N_a$, $(M, N\in {\sf CGR})$ $\iff$ they are everywhere identical exacept possibly at red edges
$$M_a(a(i), a(j))=\r^l\iff N_b(b(i), b(j))=\r^k,  \text { for some $l,k$}\in \omega.$$
We say that $M_a$ is a {\it copy of $N_b$} if $M_a\sim N_b$. 
Now we define a map $\Theta: \CA_{n+1, n}=\Cm{\bf At_f}$ to $\Cm\At$,
by  specifing  first its values on ${\sf At}_f$,
via $M_a\mapsto \sum_jM_a^{(j)}$; where $M_a^{(j)}$ is a copy of $M_a$; each atom maps to the suprema of its 
copies.  (If $M_a$ has no red edges,  then by $\sum_jM_a^{(j)}$,  we understand $M_a$).
This map is extended to $\CA_{n+1, n}$ the obvious way. The map
$\Theta$ is well--defined, because $\Cm\At$ is complete. 
It is not hard to show that the map $\Theta$ 
is an injective homomorphim. 

One next proves that \pa\ has  a 
\ws\ for \pe\ in $\bold G^{n+3}\At(\CA_{n+1, n})$
using the usual rainbow strategy by bombarding \pe\ with cones having the same base and distinct green tints.
He needs $n+3$ nodes to implement his \ws. In fact he need $n+3$ nodes to force a win in the weaker game $G^{n+3}_{\omega}$ without the need to
resue the nodes in play.
Then by lemma \ref{n}, this implies that  $\CA_{n+1,n}\notin
\bold S{\sf Nr}_n\CA_{n+3}$. Since $\CA_{n+1,n}$ embeds into $\Cm\bf At$, 
hence $\Cm\bf At$  is outside 
$\bold S{\sf Nr}_n\CA_{n+3}$, too.
\end{proof}

For $2<n\leq l<m\leq \omega$, the statement $\Psi(l, m)$ is the negation of a special case of an omitting types theorem which we define next:

\begin{definition}  Let $2<n\leq l<m\leq \omega$.  Let $T$ be an $L_n$ theory in a signature $L$ and $\Sigma$ be a set of $L$--formulas.

\begin{enumarab}

\item  We say that $T$ {\it $m$--omits $\Sigma$}, if there 
exists an injective homomorphism  $f:\A\to \wp(V)$ where $M=\bigcup_{s\in V}\rng(s)$ is an 
$m$--square representation of $\Fm_T$, and $\bigcap_{\phi\in \Sigma}f(\phi_T)=\emptyset$.

\item We say that $T$ {\it $l$-isolates $\Sigma$}, if there exists a formula $\phi$ using $l$ variables, such that
$\phi$ is consistent with $T$,  and
$T\models \phi\to \Sigma$. If not, we say that $T$ {\it $l$--locally} omits $\Sigma$.
 
\end{enumarab}
\end{definition}

Let $\lambda$ be a cardinal. 
Then ${\sf OTT}(l, m, \lambda)$ is the statement: 

{\it If $T$ is a countable $L_n$ theory, $\bold X=(\Gamma_i: i<\lambda)$ is family of types, 
such that  $T$, $l$--locally omits $\Gamma_i$ using at most $l$ variables for each $i\in \lambda$, 
then $T$,  $m$--omits $\Sigma_i$ for each $i<\lambda$.}

Let $n<2\leq l<m\leq \omega$.  Then it is not hard to see
that if $T$ is atomic, then  $\neg {\sf OTT}(l, m, 1) \iff \Psi(l, m))$  by 
taking the one type consisting of co--atoms. Henceforth, by ${\sf OTT}(l, m, 1)$ we {\it exclusively mean that the theory involved is 
atomic and the {\it one non--principal type} to be omitted is that 
consisting of the co--atoms (of $\Fm_T)$.}  
For $2<n\leq l<m\leq \omega$,  we denote $\neg {\sf OTT}(l, m, 1)$ by ${\sf VT}(l, m)$, which we read as 
{\bf Vaught's theorem holds at $l$ and $m$.} 
By $\sf VT(\omega, \omega)$ we undertand Vaught's theorem for first order logic in the following form 
(which can be seen as a limiting case of ${\sf VT}(l,m)$ when $l\to \omega$ and $m=\omega$):
 
{\sf Any countable atomic $L_{\omega, \omega}$ theory in a signature having countably many relation symbols of arity $\leq n$ 
has an atomic model.}

From theorem \ref{can} we immediately get:

\begin{corollary} Let $2<n\leq l<\omega$ and $k\geq 3$. Then ${\sf VT}(l, \omega)$ and ${\sf VT}(n, n+3)$ 
are false.
\end{corollary}

Next we formulate the logical counterpart of \cite[Problem 2.12]{HMT2} and 
theorem \ref{main}. But first a definition. By a logic $\L$ we understand a fragment of first order logic.
\begin{definition}
A {\it formula schema} in a logic $\L$ is a formula $\sigma(R_1,\ldots, R_k)$
where $R_1,\ldots, R_k$ are relation symbol.
An {\it $\omega$ instance} or even simply an instance of $\sigma$
is a formula of the form $\sigma(\chi_1, \ldots, \chi_k)$
where $\chi_1, \ldots, \chi_k$
are formulas and each $R_i$ is replaced by $\chi_i$ .A formula schema is called {\it type--free valid} if
all of its instances are valid.
\end{definition}
Here  {\it type--free valid formula schema} is a new notion of validity defined by Henkin et 
al. \cite[Remark 4.3.65, Problem 4.16]{HMT2}, \cite[p. 487]{HHbook}. 
For provability we use the proof system in \cite[p. 157, \S 4.3]{HMT2}. For $l\in \omega$, we write $\vdash_l$ for 
`provability using $l$ variables.'

\begin{theorem}\label{OTT2} 
\begin{enumerate}
\item For any $m>3$ there exists a $3$--variable type free valid formula schema that cannot be proved using $m-1$ variables, but can 
be proved using $m$ variables.

\item Assume that for any $m>3$, there exists a finite relation algebra $\R_m$ with strong $m-1$ blur but no
infinite $m$--dimensional hyperbasis. Then for every $m>3$, 
there is an $L_3$ atomic and countable theory $T$ such that the non-principal type $\Gamma$ consisting of co--atoms is realizable
in every $m$--flat-model of $T$, but cannot be isolated using only $m-1$ variables.
That is any witness of $\Gamma$ must contain at least $m$ variables. 
\end{enumerate}
\end{theorem}
\begin{proof}  
The two items depend on the existence of certain finite relation algebras.
 
For the first item such relation algebras are constructed by Hirsch and Hodkinson
\cite{HHbook}. 
Fix finite $m>3$. For each $r\in \omega$, a finite relation algebra denoted by $\A(m-1, r)$ in \cite[Definition 15.2]{HHbook} is proved to have 
an $m-1$--dimensional hyperbasis but no 
$m$--dimensional hyperbasis,  cf. \cite[Theorem 15.1]{HHbook}. 
Let $\C_r$ be the algebra $\Ca(\bold H)$, where $\bold H=H_3^{m}(\A(m-1,r), \omega))$
is the $\CA_3$ atom structure consisting of all $m$--wide $3$--dimensional
wide $\omega$ hypernetworks \cite[Definition 12.21]{HHbook}
on $\A(m-1,r)$, so that $\C_r\in \CA_3$.  
It is proved in \cite{HHbook}, that for any $r\in \omega$, $\C_r\in {\sf Nr}_3{\sf CA}_{m-1}$, $\C_r\notin {\bold  S}{\sf Nr}_3{\sf CA}_{m}$
and $\Pi_{r/U}\C_r\in {\sf RCA}_3$, cf. \cite[Corollaries 15.7, 5.10, Exercise 2, pp. 484, Remark 15.13]{HHbook}. In particular, we have (+):
for any positive $k$, $\bold S\Nr_3\CA_{3+k+1}$ is not finitely axiomatizable over $\bold S\Nr_3\CA_{3+k}$.

Now using this non--finite axiomatizability result, we prove the required in the first item. If $\phi$ is a $3$--variable formula and $\tau(\phi)$
is its corresponding term in the language of $\CA_{3}$ as defined in \cite[Definition 4.3.55]{HMT2},
then $\bold S\Nr_{3}\CA_{3+k}\models \tau(\phi)={\bf 1}\Longleftrightarrow \vdash_{3+k}\phi.$
So for any $k\geq 1$, there is no finite set of formulas of $L_3$ whose set $\Sigma$ of instances satisfies
$\Sigma \vdash_{3+k} \phi\Longleftrightarrow \vdash _{3+k+1}\phi,$
for the existence of such a set of formulas
would imply finite axiomatizability by equations of
$\bold S\sf Nr_{3}\CA_{3+k+1}$ over $\bold S\sf Nr_{3}\CA_{3+k}$
since every schema translates into an equation in the language of
$\CA_{3}$ and this contradicts (+). This plainly implies the required.

For the second item,  asume tha $\R_m$ as in the hypothesis exist.  Since $\R_m$ has an 
$m-1$ blur $(J, E)$, then $\Bb(\R_m, J, E)$ has an $m-1$-dimensional cylindric basis, 
but   $\Cm\At\Bb(\R_m, J, E)$
does not have an infinite
$m$--dimensional hyperbasis, seeing as how $\R_m$ does not have such a hyperbasis, and $\R_m$ embeds into
$\Cm\At\Bb(\R_m, J, E).$

If there is such an $\R_m$, with $(J, E)$ being a {\it strong} $m-1$ blur, then the algebra $\Bb_3(\R_m, J, E)\in \RCA_3$ will give the required theory. 
Any atomic and countable (consistent) theory  
$T$ such that $\Fm_T\cong \Bb_3(\R_m, J, E)$ 
will fit the bill.
\end{proof}

{\bf An equivalent formulation of ${\sf VT}(l, m)$:}
Fix $2<n\leq l<m\leq \omega$. We now give an equivalent formulation of ${\sf VT}(l, m)$, which is  the one described briefly in the introduction, addressing the 
the logic $\L_l$ which is a {\it first order definable expansions} of $L_n$.
Such fragments were initially approached  by Jonsson in the context of relation algebras, 
and  further studied by B\'iro, Givant, N\'emeti, Tarski, S\'agi and others. 
The original purpose was to tame unruly behaviour of $\sf RRA$ and $\RCA_n$ like non--finite axiomatizability, 
but Biro \cite{Biro} showed that such (finite first order definable expansions) are inadequate to achieve this aim. 

The logic $\L_l$ is an $n$--variable logic that is obtained from 
$L_n$ by {\it adding new first order definable connectives}.  The analogous expansions for the calculas of relations $\L^{\times}$ is approached in \cite{Biro, HHM}. 
Each such connective is {\it definable by a first order formula using $\leq l$ variables with at most $n$ free.} 
The formation rule of formulas is defined inductively the expected way.
For example if $c$ is an $n$--ary connective and $\phi_0,\ldots \phi_{n-1}$ are $n$ formulas, then
$c(\phi_0,\dots, \phi_{n-1})$ is a formula. 
{\it Now take the special case when $\L_l$ is the first order definable expansion of $L_n$ by adding $\leq n$--ary connectives for every first order formula using $l$ variables
with free variables among the first $n$.}
The logics $\L_l$ and $\L_n$ 
have the same signature; {\it its only that $\L_l$ has more first order definable connectives.}
Assume without loss that the signature consists of a single binary relation $\bold R$.

Accordingly, the models of $\L_l$ are the same as the models of $L_n$.
If $\M=(M, R)$, with $R\subseteq M\times M$ is such a model,  then $\bold R$ is interpreted to be the same binary relation $R\subseteq M\times M$ 
in $\L_l$ and $L_n$. 
The semantics of a newly introduced (first order definable) connective is defined by its defining formula. 
For example if $\Psi$ defines the unary connective $c$, say, 
then the semantics of $c$ in a model $\M$, 
is (inductively) defined for $s\in {}^nM$ by:  
$c(\phi)[s]\iff (\M, \Psi^{\M})\models \phi[s],$, where $\Psi^{\M}$ is the set of all 
$n$--ary assignments satisfying $\Psi$ in $\M$, 
that is, $\Psi^{\M}=\{s\in {}^nM: \M\models \Psi[s]\}$.

We denote the statement `any  countable atomic $\L_l$ theory has an $m$--square atomic model' by ${\sf VT}(\L_l, m).$
Then ${\sf VT}(\L_l, m)$ is equivalent $\neg \Phi(l, m)$, with $\Phi(l, m)$ 
as defined in the introduction.
To unify notation, denote $L_{\omega, \omega}$ by $\L_{\omega}$. Identifying ${\sf VT}(\L_{\omega}, \omega)$ with ${\sf VT}(\omega, \omega)$, 
the `square version' of the second item of theorem \ref{OTT2} can be reformulated as:

\begin{theorem} If for each $3<m<\omega$, there exist a finite relation algebra $\R_m$ having $m-1$ strong blur and no
$m$--dimensional relational basis (equivalently $m$--square representation), then
for $3\leq l<m$ and $l=m=\omega$, 
${\sf VT}(\L_l, m)$ holds  $\iff\l=m=\omega$.
\end{theorem}

\section{Positive results}

Throughout this section, unless otherwise indicated, $n$ is finite ordinal $>2$. 
We prove several positive omitting types theorem for $L_n$ which, as indicated in the introduction, 
can be viewed as a multi--dimensional modal logic.
For any ordinal $\alpha$, following \cite{HMT2}, ${\sf Cs}_\alpha$ denotes the class of cylindric set algebras of dimension $\alpha$ and 
${\sf Gs}_\alpha$ denotes the class of generalized cylindric set algebras of dimension $n$, whose top elements are generalized $\alpha$--dimensional 
catesian spaces.

We start by recalling certain cardinals that play a key role in (positive) omitting types theorems for $L_{\omega, \omega}$.
Let $\sf covK$ be the cardinal used in \cite[Theorem 3.3.4]{Sayed}.
The cardinal $\mathfrak{p}$  satisfies $\omega<\mathfrak{p}\leq 2^{\omega}$
and has the following property:
If $\lambda<\mathfrak{p}$, and $(A_i: i<\lambda)$ is a family of meager subsets of a Polish space $X$ (of  which Stone spaces of countable Boolean algebras are examples)  
then $\bigcup_{i\in \lambda}A_i$ is meager. For the definition and required properties of $\mathfrak{p}$, witness \cite[pp. 3, pp. 44-45, corollary 22c]{Fre}. 
Both cardinals $\sf cov K$ and $\mathfrak{p}$  have an extensive literature.
It is consistent that $\omega<\mathfrak{p}<\sf cov K\leq 2^{\omega}$ \cite{Fre},
so that the two cardinals are generally different, but it is also consistent that they are equal; equality holds for example in the Cohen
real model of Solovay and Cohen.  Martin's axiom implies that  both cardinals are the continuum.

Next we formulate the algebraic version of an omitting types theorem counting in countable infinite dimensions:

\begin{definition}\label{defott} Let $\alpha$ be any ordinal and $\lambda$ be a cardinal. 
If $\A\in \sf RCA_{\alpha}$ and $\bold X=(X_i: i<\lambda)$ is  family of subsets of $\A$, we say that {\it $\bold X$ is omitted in $\C\in \sf Gs_{\alpha}$}, 
if there exists an isomorphism 
$f:\A\to \C$ such that $\bigcap f(X_i)=\emptyset$ for all $i<\lambda$. When we want to stress the role of $f$, 
we say that $\bold X$ is omitted in $\C$ via $f$. If $X\subseteq \A$ and $\prod X=0$, 
then we refer to $X$ as a {\it non-principal type} of $\A$.
\end{definition}
To prove the main result on positive omitting types theorems, we need the following lemma due to Shelah:
\begin{lemma} \label{sh} Assume that $\lambda$ is an infinite regular cardinal. 
Suppose that $T$ is a first order theory,
$|T|\leq \lambda$ and $\phi$ is a formula consistent with $T$,  then there exist models $\M_i: i<{}^{\lambda}2$, each of cardinality $\lambda$,
such that $\phi$ is satisfiable in each,  and if $i(1)\neq i(2)$, $\bar{a}_{i(l)}\in M_{i(l)}$, $l=1,2,$, $\tp(\bar{a}_{l(1)})=\tp(\bar{a}_{l(2)})$,
then there are $p_i\subseteq \tp(\bar{a}_{l(i)}),$ $|p_i|<\lambda$ and $p_i\vdash \tp(\bar{a}_ {l(i)})$ ($\tp(\bar{a})$ denotes the complete type realized by
the tuple $\bar{a}$)
\end{lemma}
\begin{proof} \cite[Theorem 5.16, Chapter IV]{Shelah}.
\end{proof}

In the next theorem $n<\omega$:
\begin{theorem}\label{i} Let $\A\in \bold S_c\Nr_n\CA_{\omega}$ be countable.  Let $\lambda< 2^{\omega}$ and let 
$\bold X=(X_i: i<\lambda)$ be a family of non-principal types  of $\A$.
Then the following hold:
\begin{enumerate}
\item If $\A\in \Nr_n\CA_{\omega}$ and the $X_i$s are maximal non--principal ultrafilters,  then $\bold X$ can be omitted in a ${\sf Gs}_n$.
Furthrmore, the condition of maximality cannot be dispensed with,
\item Every subfamily of $\bold X$ of cardinality $< \mathfrak{p}$ can be omitted in a ${\sf Gs}_n$; in particular, every countable 
subfamily of $\bold X$ can be omitted in a ${\sf Gs}_n$,
\item  If $\A$ is simple, then every subfamily 
of $\bold X$ of cardinlity $< \sf covK$ can be omitted in a ${\sf Cs}_n$, 
\item It is consistent, but not provable (in $\sf ZFC$), that $\bold X$ can be omitted in a ${\sf Gs}_n$,
\item If $\A\in \Nr_n\CA_{\omega}$ and $|\bold X|<\mathfrak{p}$, then $\bold X$ can be omitted $\iff$ every countable subfamily of $\bold X$ can be omitted.   
If $\A$ is simple, we can replace $\mathfrak{p}$ by $\sf covK$. 
\item If $\A$ is atomic, {\it not necessarily countable}, but have countably many atoms, 
then any family of non--principal types can be omitted in an atomic $\sf Gs_n$; in particular, 
$\bold X$ can be omitted in an atomic ${\sf Gs}_n$; if $\A$ is simple, we can replace $\sf Gs_n$ by 
$\sf Cs_n$.
\end{enumerate}
\end{theorem}
\begin{proof}

For the first item we prove the special case when $\kappa=\omega$. The general case follows from the fact that (*) below holds for any infinite regular cardinal.
We assume that $\A$ is simple (a condition that can be easily removed). We have $\prod ^{\B}X_i=0$ for all $i<\kappa$ because,
$\A$ is a complete subalgebra of $\B$. Since $\B$ is a locally finite,  we can assume 
that $\B=\Fm_T$ for some countable consistent theory $T$.
For each $i<\kappa$, let $\Gamma_i=\{\phi/T: \phi\in X_i\}$.

Let ${\bold F}=(\Gamma_j: j<\kappa)$ be the corresponding set of types in $T$.
Then each $\Gamma_j$ $(j<\kappa)$ is a non-principal and {\it complete $n$-type} in $T$, because each $X_j$ is a maximal filter in $\A=\mathfrak{Nr}_n\B$.
(*) Let $(\M_i: i<2^{\omega})$ be a set of countable
models for $T$ that overlap only on principal maximal  types; these exist by lemma \ref{sh}.

Asssume for contradiction that for all $i<2^{\omega}$, there exists
$\Gamma\in \bold F$, such that $\Gamma$ is realized in $\M_i$.
Let $\psi:{}2^{\omega}\to \wp(\bold F)$,
be defined by
$\psi(i)=\{F\in \bold F: F \text { is realized in  }\M_i\}$.  Then for all $i<2^{\omega}$,
$\psi(i)\neq \emptyset$.
Furthermore, for $i\neq j$, $\psi(i)\cap \psi(j)=\emptyset,$ for if $F\in \psi(i)\cap \psi(j)$, then it will be realized in
$\M_i$ and $\M_j$, and so it will be principal.

This implies that $|\bold F|=2^{\omega}$ which is impossible. Hence we obtain a model $\M\models T$ omitting $\bold X$
in which $\phi$ is satisfiable. The map $f$ defined from $\A=\Fm_T$ to ${\sf Cs}_n^{\M}$ (the set algebra based on $\M$ \cite[4.3.4]{HMT2})
via  $\phi_T\mapsto \phi^{\M},$ where the latter is the set of $n$--ary assignments in
$\M$ satisfying $\phi$, omits $\bold X$. Injectivity follows from the facts that $f$ 
is non--zero and $\A$ is simple. 

For the second part of (1), we construct an atomic $\B\in \Nr_n\CA_{\omega}$ with uncountably many atoms that is not completely representable. This implies that the maximality condition 
cannot be dispensed with; else the set  of co--atoms of $\B$ call it $X$ will be a non--principal type that cannot be omitted, 
because any ${\sf Gs}_n$ omitting $X$ yields a complete representation of 
$\B$, witness the last paragraph in \cite{Sayed}.  
The construction is taken from \cite{bsl}. 

Let $\kappa$ be an infinite cardinal. 
We specify the atoms and forbidden triples. The atoms are $\Id, \; \g_0^i:i<2^{\kappa}$ and $\r_j:1\leq j<
\kappa$, all symmetric.  The forbidden triples of atoms are all
permutations of $({\sf Id}, x, y)$ for $x \neq y$, \/$(\r_j, \r_j, \r_j)$ for
$1\leq j<\kappa$ and $(\g_0^i, \g_0^{i'}, \g_0^{i^*})$ for $i, i',
i^*<2^{\kappa}.$  In other words, we forbid all the monochromatic triangles.

Write $\g_0$ for $\set{\g_0^i:i<2^{\kappa}}$ and $\r_+$ for
$\set{\r_j:1\leq j<\kappa}$. Call this atom
structure $\alpha$. Let $\R=\Tm(\alpha)$ We claim that $\R$ has no complete representation.
Assume for contradiction that $\R$ has a complete representation $M$.  Let $x, y$ be points in the
representation with $M \models \r_1(x, y)$.  For each $i< 2^{\kappa}$, there is a
point $z_i \in M$ such that $M \models \g_0^i(x, z_i) \wedge \r_1(z_i, y)$.
Let $Z = \set{z_i:i<2^{\kappa}}$.  Within $Z$ there can be no edges labelled by
$\r_0$ so each edge is labelled by one of the $\kappa$ atoms in
$\r_+$.  The Erdos-Rado theorem forces the existence of three points
$z^1, z^2, z^3 \in Z$ such that $M \models \r_j(z^1, z^2) \wedge \r_j(z^2, z^3)
\wedge \r_j(z^3, z_1)$, for some single $j<\kappa$.  This contradicts the
definition of composition in $\R$ (since we avoided monochromatic triangles).

Let $S$ be the set of all atomic $\R$-networks $N$ with nodes
$\kappa$ such that $\{\r_i: 1\leq i<\kappa: \r_i \text{ is the label
of an edge in $N$}\}$ is finite.
Then it is straightforward to show $S$ is an amalgamation class, that is for all $M, N
\in S$ if $M \equiv_{ij} N$,  there is $L \in S$ with
$M \equiv_i L \equiv_j N$.
So $\Ca(S)\in \CA_\omega$.
 Now let $X$ be the set of finite $\R$-networks $N$ with nodes
$\subseteq\kappa$ such that:
(1) each edge of $N$ is either (a) an atom of
$\A$ or (b) a cofinite subset of $\r_+=\set{\r_j:1\leq j<\kappa}$ or (c)
a cofinite subset of $\g_0=\set{\g_0^i:i<2^{\kappa}}$ and

(2)  $N$ is `triangle-closed', i.e. for all $l, m, n \in \nodes(N)$ we
have $N(l, n) \leq N(l,m);N(m,n)$.  That means if an edge $(l,m)$ is
labelled by $\sf Id$ then $N(l,n)= N(m,n)$ and if $N(l,m), N(m,n) \leq
\g_0$ then $N(l,n)\cdot \g_0 = 0$ and if $N(l,m)=N(m,n) =
\r_j$ (some $1\leq j<\omega$) then $N(l,n)\cdot \r_j = 0$.
For $N\in X$ let $N'\in\Ca(S)$ be defined by
$$\set{L\in S: L(m,n)\leq
N(m,n) \mbox{ for } m,n\in \nodes(N)}.$$
For $i\in \omega$, let $N\restr{-i}$ be the subgraph of $N$ obtained by deleting the node $i$.
If $N\in X, \; i<\omega$, then $\cyl i N' =
(N\restr{-i})'$. We omit this part of the proof.

Now let $X' = \set{N':N\in X} \subseteq \Ca(S)$.
Then the subalgebra of $\Ca(S)$ generated by $X'$ is obtained from
$X'$ by closing under finite unions.
Clearly all these finite unions are generated by $X'$.  We must show
that the set of finite unions of $X'$ is closed under all cylindric
operations.  Closure under unions is given.  For $N'\in X$ we have
$-N' = \bigcup_{m,n\in \nodes(N)}N_{mn}'$ where $N_{mn}$ is a network
with nodes $\set{m,n}$ and labeling $N_{mn}(m,n) = -N(m,n)$. $N_{mn}$
may not belong to $X$ but it is equivalent to a union of at most finitely many
members of $X$.  The diagonal $\diag ij \in\Ca(S)$ is equal to $N'$
where $N$ is a network with nodes $\set{i,j}$ and labeling
$N(i,j)=\sf Id$.  Closure under cylindrification is given.

Let $\C$ be the subalgebra of $\Ca(S)$ generated by $X'$.
Then we claim that $\B = \Ra(\C)$.
To see why, each element of $\R$ is a union of a finite number of atoms,
possibly a co-finite subset of $\g_0$ and possibly a co-finite subset
of $\r_+$.  Clearly $\A\subseteq\Ra(\C)$.  Conversely, each element
$z \in \Ra(\C)$ is a finite union $\bigcup_{N\in F}N'$, for some
finite subset $F$ of $X$, satisfying $\cyl i z = z$, for $i > 1$. Let $i_0,
\ldots, i_k$ be an enumeration of all the nodes, other than $0$ and
$1$, that occur as nodes of networks in $F$.  Then, $\cyl
{i_0} \ldots
\cyl {i_k}z = \bigcup_{N\in F} \cyl {i_0} \ldots
\cyl {i_k}N' = \bigcup_{N\in F} (N\restr{\set{0,1}})' \in \R$.  So $\Ra(\C)
\subseteq \A$.
$\R$ is relation algebra reduct of $\C\in\CA_\omega$ but has no complete representation.
Let $\B=\mathfrak{Nr}_n \C$ ($2<n<\omega$). Then
$\B\in \Nr_n\CA_{\omega}$ is atomic, 
but it has no complete representation.

For (2) and (3), we can assume that $\A\subseteq_c \Nr_n\B$, $\B\in \Lf_{\omega}$. 
We work in $\B$. Using the notation on \cite[p. 216 of proof of Theorem 3.3.4]{Sayed} replacing $\Fm_T$ by $\B$, we have $\bold H=\bigcup_{i\in \lambda}\bigcup_{\tau\in V}\bold H_{i,\tau}$
where $\lambda <\mathfrak{p}$, and $V$ is the weak space ${}^{\omega}\omega^{(Id)}$,  
can be written as a countable union of nowhere dense sets, and so can 
the countable union $\bold G=\bigcup_{j\in \omega}\bigcup_{x\in \B}\bold G_{j,x}$.  

So for any $a\neq 0$,  there is an ultrafilter $F\in N_a\cap (S\setminus \bold H\cup \bold G$)
by the Baire category theorem. This induces a homomorphism $f_a:\A\to \C_a$, $\C_a\in {\sf Cs}_n$ that omits the given types, such that
$f_a(a)\neq 0$. (First one defines $f$ with domain $\B$ as on p.216, then restricts $f$ to $\A$ obtaining $f_a$ the obvious way.) 
The map $g:\A\to \bold P_{a\in \A\setminus \{0\}}\C_a$ defined via $x\mapsto (g_a(x): a\in \A\setminus\{0\}) (x\in \A)$ is as required. 
In case $\A$ is simple, then by properties of $\sf covK$, $S\setminus (\bold H\cup \bold G)$ is non--empty,  so
if $F\in S\setminus (\bold H\cup \bold G)$, then $F$ induces a non--zero homomorphism $f$ with domain $\A$ into a $\Cs_n$ 
omitting the given types. By simplicity of $\A$, $f$ is injective.

To prove independence, it suffices to show that $\sf cov K$ many types may not be omitted because it is consistent that $\sf cov K<2^{\omega}$. 
Fix $2<n<\omega$. Let $T$ be a countable theory such that for this given $n$, in $S_n(T)$, the Stone space of $n$--types,  the isolated points are not dense.
It is not hard to find such theories. One such (simple) theory is the following: 
Let $(R_i:i\in \omega)$ be a countable family of unary relations and for each disjoint and finite subsets
 $J,I\subseteq \omega$, let $\phi_{I,J}$ be the formula expressing `there exists $v$ such that $R_i(v)$ holds for all $i\in I$ and $\neg R_j(v)$ holds for all
$j\in J$. Let $T$ be the following countable theory $\{\phi_{I, J}: I, J\text { as above }\}.$  
Using a simple compactness argument one can show that 
$T$ is consistent. Furthermore, for each $n\in \omega$, $S_n(T)$ does not have isolated types at all, 
hence of course the isolated types are not dense in $S_n(T)$ for all $n$. 
Algebraically, this means that if $\A=\Fm_T$, then for all $n\in \omega$, $\Nr_n\A$ is atomless.
(Another example, is the  theory of random graphs.)

This condition excludes the existence of a prime model for $T$ because $T$ has a prime model $\iff$ the isolated points in 
$S_n(T)$ are dense for all $n$. A prime model which in this context is an atomic model,  omits any family of non--principal types (see the proof of the last item).
We do not want this to happen.
Using exactly the same argument in \cite[Theorem 2.2(2)]{CF}, one can construct 
a family $P$ of non--principal $0$--types (having no free variable) of  $T$,
such that $|P|={\sf covK}$ and $P$ cannot be omitted.
Let $\A=\Fm_T$ and for $p\in P$, let $X_p=\{\phi/T:\phi\in p\}$. Then $X_p\subseteq \Nr_n\A$, and $\prod X_p=0$,
because $\Nr_n\A$ is a complete subalgebra of $\A$.
Then we claim that for any $0\neq a$, there is no set algebra $\C$ with countable base
and $g:\A\to \C$ such that $g(a)\neq 0$ and $\bigcap_{x\in X_p}f(x)=\emptyset$.
To see why, let $\B=\Nr_n\A$. Let $a\neq 0$. Assume for contradiction, that there exists
$f:\B\to \D'$, such that $f(a)\neq 0$ and $\bigcap_{x\in X_i} f(x)=\emptyset$. We can assume that
$B$ generates $\A$ and that $\D'=\Nr_n\B$, where $\B\in \Lf_{\omega}$.
Let $g=\Sg^{\A\times \B}f$.
We show that $g$ is a homomorphism with $\dom(g)=\A$, $g(a)\neq 0$, and $g$ omits $P$, and for this, 
it suffices to show by symmetry that $g$ is a function with domain $A$. It obviously has co--domain $B$. 
Settling the domain is easy: $\dom g=\dom\Sg^{\A\times \B}f=\Sg^{\A}\dom f=\Sg^{\A}\Nr_{n}\A=\A.$ 

Let $\bold K=\{\A\in \CA_{\omega}: \A=\Sg^{\A}\Nr_{n}\A\}(\subseteq \sf Lf_{\omega}$). 
We show that $\bold K$ is closed under finite dirct products. Assume that $\C$, $\D\in \bold K$, then we have 
$\Sg^{\C\times \D}\Nr_{\alpha}(\C\times \D)= \Sg^{\C\times \D}(\Nr_{n}\C\times \Nr_{n}\D)=
\Sg^{\C}\Nr_{n}\C\times  \Sg^{\D}\Nr_{n}\D=\C\times \D$.
Observe that 
 $$(a,b)\in g\land \Delta[(a,b)]\subseteq n\implies f(a)=b.$$
Indeed $(a,b)\in \Nr_{n}\Sg^{\A\times \B}f=\Sg^{\Nr_{n}(\A\times \B)}f=\Sg^{\Nr_{\alpha}\A\times \Nr_{n}\B}f=f.$
Now suppose that $(x,y), (x,z)\in g$.
Let $k\in \omega\setminus n.$ Let $\oplus$ denote `symmetric difference'. Then
$(0, {\sf c}_k(y\oplus z))=({\sf c}_k0, {\sf c}_k(y\oplus z))={\sf c}_k(0,y\oplus z)={\sf c}_k((x,y)\oplus(x,z))\in g.$
Also,
${\sf c}_k(0, {\sf c}_k(y\oplus z))=(0,{\sf c}_k(y\oplus z)).$
Thus by (*) we have  $f(0)={\sf c}_k(y\oplus z)$ for any  $k\in \omega\setminus n.$
Hence ${\sf c}_k(y\oplus z)=0$ and so $y=z$. 
We have shown that $g$ is a function, $g(a)\neq 0$,  $\dom g=\A$ and $g$ omits $P$.
This contradicts that  $P$, by its construction,  cannot be omitted.
Assuming Martin's axiom, 
we get $\sf cov K=\mathfrak{p}=2^{\omega}$; together with the above arguments this proves (4).

We now prove (5). Let $\A=\Nr_n\D$, $\D\in {}\sf Lf_{\omega}$ is countable. Let $\lambda<\mathfrak{p}.$ Let $\bold X=(X_i: i<\lambda)$ be as in the hypothesis.
Let $T$ be the corresponding first order theory, so
that $\D\cong \Fm_T$. Let $\bold X'=(\Gamma_i: i<\lambda)$ be the family of non--principal types in $T$ corresponding to $\bold X$. 
If $\bold X'$ is not omitted, then there is a (countable) realizing tree
for $T$, hence there is a realizing tree for a countable subfamily of $\bold X'$ in the sense of \cite[Definition 3.1]{CF}, 
hence a countable subfamily of $\bold X'$ 
cannot be omitted. Let $\bold X_{\omega}\subseteq \bold X$ be the corresponding countable 
subset of $\bold X$. Assume that $\bold X_{\omega}$ can be omitted in a $\sf Gs_n$, via $f$ say. 
Then  by the same argument used in proving item (4)  $f$ can be lifted to $\Fm_T$ omitting $\bold X'$, 
which is a contradiction.  We leave the part when $\A$ is simple to the reader. 

For (6): If $\A\in {\bold S}_n\Nr_n\CA_{\omega}$, is atomic and has countably many atoms, 
then any complete representation of $\A$, equivalently, an atomic representation of $\A$, equivalently, a representation of $\A$ 
omitting the set of co--atoms is as required. 
If $\A$ is simple and completely representable, 
then it is completely represented  
by a ${\sf Cs}_n$, and we are done. 
\end{proof}

\subsection{Omitting types for infinitary logics}

Here we adress omitting types theorems for certain infinitary extensions of first order logic.
Our treatment remains to be purely algebraic. 
We start with a definition:

\begin{definition} Let $\alpha$ be an ordinal. 
\begin{enumarab}
\item  A {\it weak space of dimension $\alpha$} is a set $V$ of the form $\{s\in {}^{\alpha}U: |\{i\in \alpha: s_i\neq p_i\}|<\omega\}$
where $U$ is a non--empty set and $p\in {}^{\alpha}U$.
We denote $V$ by $^{\alpha}U^{(p)}$.

\item We write $\sf Gws_{\alpha}$ short hand for  the class of {\it generalized weak set algebras}
as defined in \cite[Definition 3.1.2, item (iv)]{HMT2}. By definition $\sf Gws_{\alpha}={\bf SP}{\sf Ws}_{\alpha}$,
where $\sf Ws_{\alpha}$ denotes the class of weak set algebra of dimension $\alpha$. The top elements of $\sf Gws_{\alpha}$s are
{\it generalized weak spaces} of dimension $\alpha$; these are disjoint unions of weak spaces of the same dimension.

\item For $\alpha\geq \omega$, we let $\sf Dc_{\alpha}$ denote the class of {\it dimension complemented $\CA_{\alpha}$s}, 
so that $\A\in {\sf Dc}_{\alpha}\iff\ \alpha\setminus \Delta x$ 
is infinite for every $x\in \A$. 

\item Let $\A\in \CA_{\alpha}$. Let $\bold K\in \{\sf Gs_{\alpha}, \sf Gws_{\alpha}\}$. Then $\A$ is {\it completely representable} with respect to $\bold K$, if there exists 
$\C\in \bold K$, and an isomorphism $f:\A\to \C$ such that for all $X\subseteq \A$, 
$f(\sum X)=\bigcup_{x\in X}f(x)$, whenever $\sum X$ exists in $\A$. In this case, we say that $\A$ is completely representable via $f$.
\end{enumarab}
\end{definition}
It can be easily proved that $\A$ is completely representable via $f:\A\to \C$, where $\C$ has top element $V$ 
$\iff$ $\A$ is atomic and $f$ is {\it atomic} in the sense that 
$f(\sum \At\A)=\bigcup_{x\in \At\A}f(x)=V$, where $\At\A$ denotes the set of atoms of $\A$.
  
Let $\alpha$ be a countable infinite ordinal. In the first item of the next theorem we show that Vaught's theorem, hence the omitting types theorem 
fails for an algebraizable formalism of first order logic denoted by  $L_{\alpha}$ in \cite[\S 4.3]{HMT2}. In the second item imposing an algebraic condition 
and {\it guarding semantics} using ${\sf Gws}_{\alpha}$ instead of ${\sf Gs}_{\alpha}$, 
we get instead a positive result.
The third and fourth items show 
that  the countability condition cannot be omitted from the stipulated conditions in the second item.
In fact they say more:  Vaught's theorem fails for uncountable $L_{\alpha}$ theories.

We need a piece of notation: Let $\beta>\alpha$ be any ordinals. 
If $\A\in \CA_{\beta}$, then $\Rd_\alpha\A$ is the $\CA_\alpha$ with the same universe as $\A$ and operations, besides the Boolean operations, 
are restricted to 
(the cylindrifiers) ${\sf c}_i$, (and the diagonal elements) ${\sf d}_{ij}$ with $i, j<\alpha$.

\begin{theorem} \label{ii} Let $\alpha$ be a countable infinite ordinal.
\begin{enumerate} 
\item There exists a countable 
atomic $\A\in \RCA_\alpha$ such that the non--principal types of co--atoms cannot be omitted in a $\sf Gs_{\alpha}$, 

\item If $\A\in \bold S_c\Nr_{\alpha}\CA_{\alpha+\omega}$ is countable, $\lambda$ a cardinal $<\mathfrak{p}$ 
and $\bold X=(X_i: i<\lambda)$ is a family of non--principal types, 
then $\bold X$ can be omited in a $\sf Gws_{\alpha}$ (in the sense of the above definition upon replacing $\sf Gs_{\alpha}$ by ${\sf Gws}_{\alpha}$). 
If $\A$ is simple, and $\A\in \Nr_{\alpha}\CA_{\alpha+\omega}$, 
then we can replace $\mathfrak{p}$ by $\sf covK$,

\item There exists an atomic $\A\in \RCA_{\alpha}$ such that its \d\ completion, namely, $\Cm\At\A$ is not in $\bold S\Nr_{\alpha}\CA_{\alpha+k}$ 
for any $k\geq 3$. Furthermore, $\A$ is not completely representable with respect to $\sf Gws_{\alpha}$,

\item There exists an atomic $\A\in \Nr_{\alpha}\CA_{\alpha+\omega}$ with uncountably 
many atoms such that $\A$ has no complete representation 
with respect to ${\sf Gws}_{\alpha}$. 
\end{enumerate}
\end{theorem}
\begin{proof}

(1) Using exactly the same argument in \cite{HH}, 
one shows that if $\C\in \CA_{\omega}$ is completely representable 
$\C\models {\sf d}_{01}<1$, then $\At\C=2^{\omega}$. The argument is as follows: 
Suppose that $\C\models {\sf d}_{01}<1$. Then there is $s\in h(-{\sf d}_{01})$ so that if $x=s_0$ and $y=s_1$, we have
$x\neq y$. Fix such $x$ and $y$. For any $J\subseteq \omega$ such that $0\in J$, set $a_J$ to be the sequence with
$i$th co-ordinate is $x$ if $i\in J$, and is $y$ if $i\in \omega\setminus J$.
By complete representability every $a_J$ is in $h(1^{\C})$ and so it is in
$h(x)$ for some unique atom $x$, since the representation is an atomic one.
Let $J, J'\subseteq \omega$ be distinct sets containing $0$.   Then there exists
$i<\omega$ such that $i\in J$ and $i\notin J'$. So $a_J\in h({\sf d}_{0i})$ and
$a_J'\in h (-{\sf d}_{0i})$, hence atoms corresponding to different $a_J$'s with $0\in J$ are distinct.
It now follows that $|\At\C|=|\{J\subseteq \omega: 0\in J\}|=2^{\omega}$.

Take $\D\in {\sf Cs}_{\omega}$ with universe $\wp({}^{\omega}2)$. 
Then $\D\models {\sf d}_{01}<1$ and plainly $\D$ is completely representable. 
Using the downward \ls--Tarski theorem, take a countable  elementary subalgebra $\B$ of $\D$.
This is possible because the signature of $\CA_{\omega}$ is countable.
Then in $\B$ we have $\B\models {\sf d}_{01}<1$ because $\B\equiv \C$. But $\B$ 
cannot be completely 
representable, because if it were then by the above argument, we get that $|\At\B|=2^{\omega}$,
which is impossible because $\B$ is countable. 

(2) Now we prove the second item, which is a generalization of \cite[Theorem 3.2.4]{Sayed}. 
Though the generalization is strict, in the sense that $\sf Dc_{\omega}\subsetneq \bold S_c\Nr_{\omega}\CA_{\omega+\omega}$
\footnote{It is not hard to see that the full set algebra with universe $\wp({}^{\omega}\omega)$ 
is in $\Nr_{\omega}\CA_{\omega+\omega}\subseteq \bold S_c\Nr_{\omega}\CA_{\omega+\omega}$ but it is not
in $\sf Dc_{\omega}$ because for any $s\in {}^{\omega}U$, $\Delta\{s\}=\omega$.}
the proof is the same. Without loss, we can take $\alpha=\omega$. Let $\A\in \CA_{\omega}$ be as in the hypothesis. 
For brevity, let $\beta=\omega+\omega$.  By hypothesis, we have 
$\A\subseteq_c\Nr_{\alpha}\D$, with $\D\in \CA_{\beta}$.
 
We can also assume that $\D\in {\sf Dc}_{\beta}$ by replacing, if necessary, $\D$ by
$\Sg^{\D}\A$.   Since $\A$ is a complete sublgebra of $\Nr_{\omega}\D$  which in turn is a complete subalgebra 
of $\D$, we have $\A\subseteq_c \D$. Thus given $< \mathfrak{p}$ non--principal types in $\A$ 
they stay non--principal in $\D$.  Next one proceeds like in {\it op.cit} since 
$\D\in {\sf Dc}_{\beta}$ is countable;  this way omitting any $\bold X$ consisting of $<\mathfrak{p}$ 
non--principal types.

For all non-zero $a\in \D$, there exists $\B\in \sf Ws_{\beta}$  and a homomorphism $f_a:\D\to \B$ (not necessarily injective) 
such that $f_a(a)\neq \emptyset$ and $f_a$ omits $\bold X$. 
Let $\C=\bold P_{a\in \D, a\neq 0}\B_a\in \sf Gws_{\beta}$. Define $g:\D\to  \C$ by 
$g(x)=(f_a(x): a\in \D\setminus \{0\})$, and then relativize $g$ to $\A$ as follows:

Let $W$ be the top element of $\C$. Then  $W=\bigcup_{i\in I} {}^{\beta}U_i^{(p_i)}$, where 
$p_i\in  {}^{\beta}U_i$ and $^{\beta}U_i^{(p_i)}\cap {}^{\beta}U_j^{(p_j)}$, for $i\neq j\in I$.
Let   $V=\bigcup _{i\in I}{}^{\alpha}U_i^{(p_i\upharpoonright \alpha)}$. 
For $s\in V$, $s\in {}^{\alpha}U_i^{(p_i\upharpoonright \alpha)}$ (for a unique $i$), let $s^+=s\cup p_i\upharpoonright \beta\setminus \alpha$. 
Now define $f:\A\to \wp(V)$, via $a\mapsto \{s\in V: s^+\in g(a)\}$.
Then $f$ is as required.

Now assume that $\A$ is simple, with $\A=\Nr_{\alpha}\D$ and $\D\in {\sf Dc}_{\beta}$.
It suffices to show that  $\D$ is simple, too. Consider the function $F(I)=I\cap \A$, $I$ an ideal in $\D$. 
It is straightforward to check that $F$ establishes an isomorphism between
the lattice of ideals in $\D$ and the lattice of ideals of $\A$ (the order here is  of course $\subseteq$),  
with inverse $G(J)= {\sf Ig}^{\D}(J)$,  $J$ an ideal in $\A$, cf. \cite[Theorem 2.6.71, Remark 2.6.72]{HMT2}. Here 
$\Ig^{\D}J$ denotes the ideal of $\D$  generated by $J$. 
Thus $\A$ is simple $\iff$ $\D$ is simple.

(3)  Throughout this item and the next one $F$ denote a non--principal ultrafilter on $\omega\setminus 3$.
We lift the construction in theorem \ref{can} proving $\Psi(n, n+3)$ 
to the transfinite using ultraproducts. 

For each finite $k\geq 3$, let $\A(k)$
be an atomic countable simple representable
$\CA_k$ such that  $\B(k)=\Cm\At\A(k)\notin \bold S\Nr_k\CA_{k+3}.$ We know by that such algebras exist by theorem \ref{can}.
Let $\A_k$ be an (atomic) algebra having the signature of $\CA_{\omega}$
such that $\Rd_k\A_k=\A(k)$.
Analogously, let $\B_k$ be an algebra having the signature
of $\CA_{\omega}$ such that $\Rd_k\B_k=\B(k)$, and we require in addition that $\B_k=\Cm(\At\A_k)$.
Let $\B=\Pi_{i\in \omega\setminus 3}\B_i/F$.
It is easy to show that 
$\A=\Pi_{i\in \omega\setminus 3}\A_i/F\in \RCA_{\omega}$.
Furthermore, a direct computation gives:
$$\Cm\At\A=\Cm(\At[\Pi_{i\in \omega\setminus 3 }\A_i/F])
=\Cm[\Pi_{i\in \omega\setminus 3}(\At\A_i)/F)]$$
$$=\Pi_{i\in \omega\setminus 3}(\Cm(\At\A_i)/F)
=\Pi_{i\in \omega\setminus 3}\B_i/F
=\B.$$
By the same token, we have $\B\in \CA_{\omega}$. We now show that $\B$ is outside $\bold S\Nr_{\omega}\CA_{\omega+3}$
proving the required.
Assume for contradiction that $\B\in \bold S\Nr_{\omega}\CA_{\omega+3}$.
Then $\B\subseteq \Nr_{\omega}\C$ for some $\C\in \CA_{\omega+3}$.
Let $3\leq m<\omega$ and  let $\lambda:m+3\rightarrow \omega+3$ be the function defined by $\lambda(i)=i$ for $i<m$
and $\lambda(m+i)=\omega+i$ for $i<3$.

Then we get (*) $\Rd^\lambda\C\in \CA_{m+3}$ and $\Rd_m\B\subseteq \Rd_m\Rd^\lambda\C$.
It is straighforward to show that $\B_m$ embeds into $\Rd_m\B_{t}$, whenever $3\leq m<t<\omega$. Call
this embedding $I_t$, so that $I_t: \B_m\to \Rd_m\B_t$ is an injective homomorphism.
Let $\iota( b)=(I_{t}b: t\geq m )/F$ for  $b\in \B_m$.
Then $\iota$ is an injective homomorphism that embeds $\B_m$ into
$\Rd_m\B$.  By (*)  we know that $\Rd_{m}\B\in {\bf S}\Nr_m\CA_{m+3}$, hence  $\B_m\in \bold S\Nr_{m}\CA_{m+3}$, too.
This is a contradiction, and we are done.

(4) This time we lift the construction used in the first item of theorem \ref{i} to the transfinite
using essentially the same idea as in the previous item.
For each finite $k\geq 3$, let $\C(k)$ be an 
algebra in $\Nr_k\CA_{\omega}$ 
having uncountably many atoms that is not completely representable; such algebras 
were constructed in the first item of theorem \ref{i}. 

Let ${\C}_{k}$ be an algebra having the signature of $\CA_{\omega}$ such that
$\Rd_k{\C}_k={\C}(k)=\Nr_{k}\D_{k}$,
where $\D_{k}\in \CA_{\omega}$. 
Let $\B=\Pi_{k\in \omega\setminus 3}\C_{i}/F.$
We will prove that $\B\in \bold \Nr_\omega\CA_{\omega+\omega},$ 
and that $\B$ is not completely representable with respect to ${\sf Gws}_{\omega}$.

Let $\A_{k}$ be an algebra having the signature of
$\CA_{\omega+\omega}$ such that
$\Rd_{\omega}\A_{k}=\D_{k}$.
Then from \cite[Lemma 3.2]{t}, we have
$\Pi_{k\in \omega\setminus 3}\A_{k}/F\in \CA_{\omega+\omega}$.
We now prove that $\B\cong\mathfrak{Nr}_\omega(\Pi_{k\omega\setminus 3}\A_k/F)$.  
Computing for each finite $k\geq 3$:
 $$\Rd_k\C_{k}=\C(k)\cong\Nr_{k}\D_{k}=\Nr_{k}\Rd_{\omega}\A_{k}= \Rd_k\Nr_\omega\A_k.$$
Then from \cite[Lemma 3.3]{t}, using a standard Lo\'s argument we get:
$\Pi_{k\in \omega\setminus 3}\C_k/F\cong\Pi_{k\in \omega\setminus 3}[\mathfrak{Nr}_\omega\A_k/F]\cong\mathfrak{Nr}_\omega(\Pi_{k\in \omega\setminus 3}\A_k/F).$
We are through with the first required.

For the second part, we proceed as follows.  Assume for contradiction
that $\B$ (which is atomic) is completely representable with respect to $\sf Gws_{\omega}$, with isomorphism $f$ establishing the complete representation.
Identifying set algebras with their domain, we have $f: \B\to \wp(V)$, where $V$ is a generalized weak space.
Let $3\leq m<\omega$. Then  $\C=\Rd_m\B$ is completely representable, via $f\upharpoonright \C$, by noting that 
$\Rd_m\wp(V)\cong \wp(W)$ for some $W$; a generalized space of dimension $m$, and that this isomorphism preserves infinite
intersections. 
In more detail, let $U$ be the base of $\D=\wp(V)$, that is $U=\bigcup_{s\in V}\rng(s)$ and let $d\in \D$, $d\neq 0$. 
Fix $x\in d$. For each $s\in {}^nU$, let $s^+=x\upharpoonright (\alpha\setminus n)\cup s$, 
then $s^+\in {}^{\alpha}U$. For $Y\in \D$, let $g_d(Y)=\{y\in {}^nU: y^+\in Y\}$. 
Then $g_d:\Rd_m\wp(V)\to \wp(^nU)$ is a homomorphism that preserves infinite intersections, and 
$g_d(d)\neq 0$. By taking the direct product of images (varying non-zero $d\in \wp(V)$), we get an embedding  
$g:\Rd_m\wp(V)\to \C$, where $\C\in {\sf Gs}_m$ and $g$ 
preserves infinite intersections.  Then $g\circ (f\upharpoonright \C)$ 
is the desired complete representation of $\C=\Rd_m\B$.

We have shown that for $3\leq m<\omega$, 
$\C=\Rd_m\B$ is completely representable, with the complete representation induced by the complete represenation of 
$\C$.

From the finite dimensional case we have $\sf R=\Ra\C$, $\C\in \CA_{\omega}$. Then, for any $2<m<n<\omega$, the identity map 
is  a complete embedding from $\C(m)=\Nr_m\C$ to $\Rd_{m}\Nr_n\C=\Rd_m\C(n)$.
Fix $2<m<\omega$. For each finite $k>m$ let $I_{k}$ be a complete embedding 
from ${\C}(m)$ into $\Rd_m {\C}(k)$; we know that such an embedding exists. Let $\iota( b)=(I_k(b): k\in \omega)/F$ 
for  $b\in \C(m)$.

Then $\iota$ is a complete embedding  from $\C(m)$ into $\Rd_m\B=\Rd_m(\Pi_{k}\C_{k}/F$).
We can assume that $\C(m)\subseteq_c \Rd_m\B$. We show that 
$\C(m)$  is completely representable which leads to a contradiction finishing the proof.
For brevity denote $\Rd_m\B$ by $\D$ and we identify set algebras with their universe.
Let $f:\D\to \wp(V)$ be a complete representation of
$\D$  via $f$, where $V$ is an $n$--dimensional generalized space. 
We claim that $g=f\upharpoonright \C(m)$ is the required complete representation of $\C(m)$. Let $X\subseteq \C(m)$ be such that $\sum^{\C(m)}X=1$. 
Then by $\C(m)\subseteq_c \D$, we have  $\sum ^{\D}X=1$. Furthermore, for all $x\in X(\subseteq \C(m))$ we have $f(x)=g(x)$, so that 
$\bigcup_{x\in X}g(x)=\bigcup_{x\in X} f(x)=V$ 
(since $f$ is a complete representation).  This finishes the proof. 
\end{proof}

\section{Vaught's theorem and complete representations}

Fix $2<n<\omega$. Here we approach Vaught's theorem for fragments of $L_n$ and variants of $L_{\omega, \omega}$ by
proving positive results on complete representability for various classes of cylindric--like algebras. 

\subsection{Complete representations}

We denote the class of completely representable algebras
with respect to 
$\sf Gws_{\alpha}$ by ${\sf CRCA}_{\alpha}.$ 

The next lemma tells us that the notions of atomicity and complete representation of an algebra 
are inherited by complete (hence dense) sublgebras. The second part is formulated only for ${\sf CA}_n$s, but it is valid for all 
classes of algebras considered in what follows (with the same proof).

\begin{lemma}\label{join} Let $\D$ be a Boolean algebra.
Assume that $\A\subseteq_c \D$. If $\D$ is atomic, then $\A$ is atomic
\cite[Lemma 2.16]{HHbook}.  If $\D\in \sf CRCA_\alpha$, $\alpha$ any ordinal, 
is completely representable, then $\A\in \sf CRCA_{\alpha}$.
\end{lemma}
\begin{proof}
Let $\A$ and $\D$ be as indicated. We show that $\A$ is atomic. Let $a\in A$ be non--zero. Then since $\D$ is atomic, 
there exists an atom $d\in D$, such that $d\leq a$. 
Let $F=\{x\in A: x\geq d\}$. Then $F$ is an ultrafilter of $\A$. 
It is clear that $F$ is a filter. To prove maximality, assume that $c\in A$ and $c\notin F$, then $-c\cdot d\neq 0$, so $0\neq -c\cdot d\leq d$, hence $-c\cdot d=d$, 
because $d$ is an atom in $\B$, thus $d\leq -c$, and we get by definition that $-c\in F$.  We have shown that $F$ is an ultrafilter. 
We now show  that $F$ is a principal ultrafilter in $\A$, 
that is, it is generated by an atom. Assume for contradiction that it is not, so that $\prod^{\A} F$ exists, because $F$ is an ultrafilter 
and $\prod ^{\A}F=0$, because it is non--principal. 
But $\A\subseteq_c \D$, so we obtain $\prod^{\A}F=\prod^{\D}F=0$. This contradicts that $0<d\leq x$ for all $x\in F$.
Thus $\prod^{\A}F=a'$, $a'$ is an atom in $\A$, $a'\in F$ 
and $a'\leq a$, because $a\in F$.  We have proved the first required.

Let  $\A\subseteq_c \D$ and 
assume that $\D$ is completely representable.  We will show that $\A$ is completely representable. 
Identifying set agebras with their domain, let $f:\D\to \wp(V)$ be a complete representation of
$\D$ where $V$ is a ${\sf Gws}_{\alpha}$ unit. 
We claim that $g=f\upharpoonright \A$ is a complete representation of $\A$. 
The argument used is the same argument used in the proof of the last item of theorem \ref{ii}.
Let $X\subseteq \A$ be such that $\sum^{\A}X=1$. 
Then by $\A\subseteq_c \D$, we have  $\sum ^{\D}X=1$. Furthermore, for all $x\in X(\subseteq \A)$ we have $f(x)=g(x)$, so that 
$\bigcup_{x\in X}g(x)=\bigcup_{x\in X} f(x)=V$, since $f$ is a complete representation, 
and we are done. 
\end{proof}
For a class $\bold K$ of $\sf BAO$s, we write $\bold K\cap \bf At$ for the class of atomic algebras in $\bold K$.
\begin{theorem}\label{iii}
Let $\alpha$ be an arbitary countable ordinal.
Then the following hold:
\begin{enumerate}
\item $\sf CRCA_{\alpha} \subseteq \bold S_c\Nr_{\alpha}(\CA_{\alpha+\omega}\cap {\bf At})\cap \bf At\subseteq \bold S_c\Nr_{\alpha}\CA_{\alpha+\omega}\cap \bf At,$
\item If  $\alpha<\omega$, and $\A\in {\sf CRCA}_{\alpha}$, then \pe\ has  a \ws\ in $G_{\omega}(\At\A)$ and $\bold G^{\omega}(\At\A),$ 
\item All  reverse inclusions and implications in the previous two items hold, 
if $\A$ has countably many atoms, or $\alpha\leq 2$,
\item For $\alpha>2$, $\sf CRCA_{\alpha}\subsetneq \Nr_{\alpha}\CA_{\alpha+\omega}\cap \bf At\subsetneq 
\bold S_c\Nr_{\alpha}\CA_{\alpha+\omega}\cap \bf At,$  
\item If $2<\alpha<\omega$, then ${\sf CRCA}_{\alpha}\subsetneq {\bf El}{\sf CRCA}_{\alpha}$ \cite{HH}. Furthermore, for any $\alpha<\omega$, 
${\bf El}\sf CRCA_{\alpha}={\bf El}\bold S_c\Nr_{\alpha}(\CA_{\omega}\cap {\bf At})={\bf El}\bold S_c\Nr_{\alpha}\CA_{\omega}\cap \bf At
={\bf El}(\bold S_c\Nr_{\alpha}\CA_{\omega}\cap \bf At),$
\item For $2<\alpha<\omega$, ${\bf EL}{\sf CRCA}_{\alpha}$ is not finitely 
axiomatizable.

\end{enumerate}
\end{theorem}
\begin{proof} 
The first inclusion in the first item: Let $\A\in {\sf CRCA}_\alpha$. Then $\A$ is atomic. Assume that $\M$ is the base of a complete representation of $\A$, whose
unit is a weak generalized space,
that is, $1^{\M}=\bigcup {}^nU_i^{(p_i)}$ $p_i\in {}^{\alpha}U_i$, where $^{\alpha}U_i^{(p_i)}\cap {}^{\alpha}U_j^{(p_j)}=\emptyset$ for distinct $i$ and $j$, in some
index set $I$ and there is an isomorphism $t:\B\to \C$, where $\C\in \sf Gs_{\alpha}$ 
has unit $1^{\M}$, and $t$ preserves arbitrary meets carrying
them to set--theoretic intersections.
For $i\in I$, let $E_i={}^{\alpha}U_i^{(p_i)}$. Take  $f_i\in {}^{\alpha+\omega}U_i^{(q_i)}$ where $q_i\upharpoonright \alpha=p_i$
and let $W_i=\{f\in  {}^{\alpha+\omega}U_i^{(q_i)}: |\{k\in \alpha+\omega: f(k)\neq f_i(k)\}|<\omega\}$.
Let ${\C}_i=\wp(W_i)$. Then $\C_i$ is atomic; indeed the atoms are the singletons. 
 
Let $x\in \Nr_{\alpha}\C_i$, that is ${\sf c}_ix=x$ for all $\alpha\leq i<\alpha+\omega$.
Now if  $f\in x$ and $g\in W_i$ satisfy $g(k)=f(k) $ for all $k<\alpha$, then $g\in x$.
Hence $\Nr_{\alpha}\C_i$
is atomic;  its atoms are $\{g\in W_i:  \{g(i):i<\alpha\}\subseteq U_i\}.$
Define $h_i: \A\to \Nr_{\alpha}\C_i$ by
$h_i(a)=\{f\in W_i: \exists a'\in \At\A, a'\leq a;  (f(i): i<\alpha)\in t(a')\}.$
Let $\D=\bold P _i \C_i$. Let $\pi_i:\D\to \C_i$ be the $i$th projection map.

Now clearly  $\D$ is atomic, because it is a product of atomic algebras,
and its atoms are $(\pi_i(\beta): \beta\in \At(\C_i))$.  
Now  $\A$ embeds into $\Nr_{\alpha}\D$ via $J:a\mapsto (\pi_i(a) :i\in I)$. If $x\in \Nr_{\alpha}\D$,
then for each $i$, we have $\pi_i(x)\in \mathfrak{Nr}_{\alpha}\C_i$, and if $x$
is non--zero, then $\pi_i(x)\neq 0$. By atomicity of $\C_i$, there is an $\alpha$--ary tuple $y$, such that
$\{g\in W_i: g(k)=y_k\}\subseteq \pi_i(x)$. It follows that there is an atom
of $b\in \A$, such that  $x\cdot  J(b)\neq 0$, and so the embedding is atomic, hence complete.
We have shown that $\A\in \bold S_c{\sf Nr}_{\alpha}\CA_{\alpha+ \omega}\cap \bf At$
and we are done.  The second inclusion is straightforward, since  
$\bold S_c\Nr_{\alpha}(\CA_{\alpha+\omega}\cap {\bf At})\subseteq \bold S_c\Nr_{\alpha}\CA_{\alpha+\omega}.$

(2) follows from first item, together with lemma \ref{n}. 

The first part of (3) follows by observing that, for any countable ordinal, 
$\alpha$, the class $\sf CRCA_{\alpha}$ coincides with the class of atomic algebras in $\bold S_c\Nr_{\alpha}\CA_{\alpha+\omega}$ 
having countably many atoms. This can be proved like the second item in theorem \ref{ii}: Assume that $\A\in \bold S_c\Nr_{\alpha}\CA_{\alpha+\omega}$ 
is atomic having countably many atoms. Take $\B$ to be $\Tm\At\A$, that is, $\B$ is the term algebra; the subalgebra of $\A$ generated by the atoms. 
Then $\B$ is countable and atomic, and  it is easy to see that $\A\in \sf CRCA_{\alpha}\iff \B\in \sf CRCA_{\alpha}$, because 
$\A$ and $\B$ share 
the same completely representable atom structure. 
Take $X$ to be the non--principal type consisting of co--atoms in $\B$.
Then,  like in the second item of theorem \ref{ii}, $\B$ is completely representable by any $\sf Gws_{\alpha}$ omitting $X$, 
hence $\A\in \sf CRCA_{\alpha}$. 

It is known \cite[Proposition 3.8.1]{HHbook2} that for $\alpha\leq 2$, $\sf CRCA_{\alpha}=\RCA_{\alpha}\cap \bf At$
hence, from the first item, we get that $\RCA_{\alpha}\cap \bf At\subseteq \bold S_c\Nr_{\alpha}(\CA_{\alpha}\cap \bf At)\cap \bf At\subseteq
\bold S_c\Nr_{\alpha}\CA_{\alpha}\cap \bf At\subseteq  \RCA_{\alpha}\cap \bf At$. 
This takes care of the second and last required in  the third item.

Item (4): Strictness of the first inclusion for $2<\alpha<\omega$, follows from the example in the first item of theorem \ref{i}. The infinite dimensional case is proved in the last part of theorem 
\ref{ii} by lifting the construction in the finite dimensional case to the transfinite. 
The strictness of the second inclusion follows from the construction in \cite{SL}
which we briefly recall. 
Let $\alpha>1$. Let $V$ be the weak space $^{\alpha}\Q^{(\bold 0)}$, where $\bold 0$ $\alpha$-ary zero sequence and let 
${\A}\in {\sf Ws}_{\alpha}$ have  universe $\wp(V)$. Then 
$\A\in \Nr_{\alpha}{\sf CA}_{\alpha+\omega}$. To see why  let 
$W={}^{\alpha+\omega}\Q^{(\bold 0)}$, where $\bold 0$ is the $\alpha+\omega$--ary zero sequence
and let $\D=\wp(W)$.
Then the map $\theta: \A\to \wp(\D)$ defined via $a\mapsto \{s\in W: (s\upharpoonright \alpha)\in a\}$, 
is an injective homomorphism from $\A$ into $\Rd_{\alpha}\D$ that is onto 
$\Nr_{\alpha}\D$.
Let $y=\{s\in V: s_0+1=\sum_{i>0} s_i\}$
and for $s\in y$, let $y_s$ be the singleton containing $s$, i.e. $y_s=\{s\}.$
Let ${\B}=\Sg^{\A}\{y,y_s:s\in y\}.$ Then it is proved in \cite{SL} that $\B$ is atomic, in fact 
$\B$ contains every singleton $\{s\}$ with $s\in V$. 
Sharing the same atom structure (consisting of the singletons $\{s\}$, $s\in V$), $\B\subseteq_d \A$ and so $\B\subseteq_c \A$.
Thus $\B\in \bold S_c\Nr_{\alpha}\CA_{\alpha+\omega}\cap \bf At$ 
because $\A\in  \Nr_{\alpha}\CA_{\alpha+\omega}$ is atomic.  As proved in \cite{SL},  
$\B\notin {\bf  El}\Nr_{\alpha}{\sf CA}_{\alpha+1}(\supseteq \Nr_{\alpha}\CA_{\alpha+\omega}\cap \bf At)).$

For the strictness of the first inclusion in item (5) we use the 
$\CA_{\alpha}$ based on $\Z$ and $\N$ denoted by $\C$ in the first item of theorem \ref{can}.  In {\it op.cit}, 
it is shown that $\C\notin \bold S_c\Nr_{\alpha}\CA_{\alpha+3}(\supseteq \bold S_c\Nr_{\alpha}\CA_{\omega}$), so $\C$ is not completely representable by the first item.
In \cite{mlq}, it is proved that $\C\equiv \B$ for some countable $\B\in \sf CRCA_{\alpha}$, 
so $\C\in {\bf El}\sf CRCA_{\alpha}\sim \CRCA_{\alpha}$.

We prove the required equalities of the given classes for $\alpha<\omega$ in item (5). Define the class $\sf LCA_{\alpha}$ as follows: 
$\A\in \sf LCA_{\alpha}\iff \A$ is atomic and  \pe\ has a \ws\ in $G_k(\At\A)$ for all $k<\omega$. 
Then this class is elementary because a \ws\ for \pe\ in $G_k$ can be coded in  first order sentence $\rho_k$.
We show that all the given classes coincide with this class, getting the required. 
The class of atom structures of this class is studied in \cite{HHbook2} under the name of atom structures 
satisfying the `Lyndon conditions'. In the present context, 
working on the `atomic algebras level' the Lyndon conditions are just the first order sentences
$\{\rho_k: k\in \omega\}$.

Assume that $\A\in {\sf LCA}_\alpha$.
Take a countable elementary subalgebra $\C$ of $\A$.
Since $\sf LCA_{\alpha}$ is elementary, then $\C\in \sf LCA_{\alpha}$, so for 
$k<\omega$, \pe\ has a \ws\ $\rho_k$,  in $G_k(\At\C)$.
Let $\D$ be a non--principal ultrapower of $\C$.  Then \pe\ has a \ws\ $\sigma$ in $G_{\omega}(\At\D)$ --- essentially she uses
$\rho_k$ in the $k$'th component of the ultraproduct so that at each
round of $G_{\omega}(\At\D)$,  \pe\ is still winning in co--finitely many
components, this suffices to show she has still not lost. Now one can use an
elementary chain argument to construct countable elementary
subalgebras $\C=\A_0\preceq\A_1\preceq\ldots\preceq\ldots \D$ in the following way.
One defines  $\A_{i+1}$ to be a countable elementary subalgebra of $\D$
containing $\A_i$ and all elements of $\D$ that $\sigma$ selects
in a play of $G_{\omega}(\At\D)$ in which \pa\ only chooses elements from
$\A_i$. Now let $\B=\bigcup_{i<\omega}\A_i$.  This is a
countable elementary subalgebra of $\D$, hence necessarily atomic,  and \pe\ has a \ws\ in
$G_{\omega}(\At\B)$, so $\B$ is completely representable.
Thus $\A\equiv \C\equiv \B$, hence $\A\in {\bf El}{\sf CRCA}_\alpha$. We have shown that ${\sf LCA}_{\alpha}\subseteq {\bf El}\sf CRCA_{\alpha}.$

Now if $\A\in \bold S_c{\sf Nr}_\alpha\CA_{\omega}\cap {\bf At}$, then by lemma \ref{n}, 
\pe\ has a \ws\ in $\bold G^{\omega}(\At\A)$, hence in $G_{\omega}(\At\A)$, {\it a fortiori}, in  $G_k(\At\A)$ for all $k<\omega$, 
so, by definition,  $\A\in \sf LCA_{\alpha}$.
Since ${\sf LCA}_\alpha$ is elementary,  we get that ${\bf El}(\bold S_c{\sf Nr}_\alpha\CA_{\omega}\cap {\bf At})\subseteq {\sf LCA}_\alpha$.
But ${\sf CRCA}_\alpha\subseteq \bold S_c{\sf Nr}_\alpha\CA_{\omega}\cap \bf At$,  hence 
${\sf LCA}_\alpha={\bf El}{\sf CRCA}_\alpha\subseteq {\bf El}(\bold S_c{\sf Nr}_\alpha\CA_{\omega}\cap \bf At)\subseteq {\sf LCA}_{\alpha}$. 
Now $\bold S_c\Nr_{\alpha}\CA_{\omega}\cap \bf At\subseteq {\bf El}\bold S_c\Nr_{\alpha}\CA_{\omega}\cap \bf At$,
and the latter class is elementary (if $\bold K$ is elementary, then $\bold K\cap \bf At$ is elementary because atomicity is a first order property), 
so ${\bf El}(\bold S_c\Nr_{\alpha}\CA_{\omega}\cap \bf At)\subseteq {\bf El}\bold S_c\Nr_{\alpha}\CA_{\omega}\cap \bf At.$
Conversely, if $\C$ is in the last class, then $\C$ is atomic and $\C\equiv \D$, for some $\D\in \bold S_c\Nr_{\alpha}\CA_{\omega}$.
Hence $\D$ is atomic, so $\D\in  \bold S_c\Nr_{\alpha}\CA_{\omega}\cap \bf At$, 
thus $\C\in  {\bf El}(\bold S_c\Nr_{\alpha}\CA_{\omega}\cap \bf At)$.

We have shown that 
${\bf El}\bold S_c\Nr_{\alpha}\CA_{\omega}\cap \bf At={\bf El}(\bold S_c{\sf Nr}_\alpha\CA_{\omega}\cap \bf At)={\sf LCA}_{\alpha}={\bf El}\sf CRCA_{\alpha}$. 
Finally, by lemma \ref{n}, $\bold S_c\Nr_{\alpha}(\CA_{\omega}\cap {\bf At})\subseteq {\sf LCA}_{\alpha}$, so
${\bf El}\bold S_c\Nr_{\alpha}(\CA_{\omega}\cap {\bf At})\subseteq {\sf LCA}_{\alpha}$. The other inclusion follows from 
$\sf CRCA_{\alpha}\subseteq \bold S_c\Nr_{\alpha}(\CA_{\omega}\cap \bf At)$, so 
$\sf LCA_{\alpha}={\bf El}\CRCA_{\alpha}\subseteq {\bf El}\bold S_c\Nr_{\alpha}(\CA_{\omega}\cap \bf At)$.   
We have shown that all classes coincide with ${\sf LCA}_{\alpha}$, 
which is the elementary closure of $\sf CRCA_{\alpha}$, and we are done.

Let $2<\alpha<\omega$. We prove the last required, namely, the non--finite axiomatizability of ${\bf El}\sf CRCA_{\alpha}(=\sf LCA_{\alpha})$.
For each $2<\alpha\leq l<\omega$, 
let $\R_l$ be the finite Maddux algebra $\mathfrak{E}_{f(l)}(2, 3)$ with strong $l$--blur
$(J_l,E_l)$ and $f(l)\geq l$ as specified in the proof of the first item of theorem \ref{main}.  
Let ${\cal R}_l={\Bb}(\R_l, J_l, E_l)\in \sf RRA$ and let 
$\A_l=\Nr_\alpha{\Bb}_l(\R_l, J_l, E_l)\in \RCA_\alpha$. 
Then  the sequence $(\At\A_l: l\in \omega\setminus \alpha)$ is a a sequence of weakly representable atom structures 
that are not strongly representable with a completely representable 
ultraproduct.  This sequence witnesses non--finite
axiomatizability of ${\sf LCA}_\alpha$,  because 
$\Cm\At\A_l\notin \RCA_\alpha(\supseteq {\sf LCA}_\alpha)$, but 
$\Pi_{l\in \omega\setminus \alpha}\Cm\At\A_l/F\in {\sf CRCA}_\alpha\subseteq {\sf LCA}_\alpha$, 
since $\Pi_{l\in \omega\setminus \alpha}\Cm\At\A_l/F\cong \Cm[\Pi_{l\in \omega\setminus \alpha}\At\A_l/F]$,
and $\Pi_{l\in \omega\setminus \alpha}\At\A_l/F$  
is a completely representable atom structure. 
\end{proof}

\subsection{Characterizing 
complete representations via atomic dilations}

In the first item of theorem \ref{i}, 
we showed that, for $2<n<\omega$, there is an atomic $\A\in {\sf Nr}_n\CA_{\omega}$ with uncountably many atoms such that $\A$ is not completely representable 
witnessing the strictness of the inclusion in item (4) of theorem \ref{iii}.
But the $\C\in \CA_{\omega}$ for which $\A=\Nr_n\C$ is atomless. 
Indeed, using the notation in {\it op.cit}, for any $N\in X$, we can add an extra node 
extending
$N$ to $M$ such that $\emptyset\subset M'\subset N'$, so that $N'$ cannot be an atom.
So can $\C$ be atomic?  We call $\C$ an {\it $\omega$--dilation of $\A$.} 

In what follows we adress complete representability of a given algebra 
in connection to the existence of an $\omega$--dilation of this algebra  that is atomic. We shall deal with many classes of cylindric--like algebras for which the neat reduct 
operator can be defined. In particular, for such classes, and regardless of atomicity, we can (and will) 
talk about an $\omega$--dilation of a given algebra.

For an ordinal $\alpha$, let $\PA_{\alpha}(\PEA_{\alpha})$ denote the class of $\alpha$--dimensional polyadic (equality) algebas as defined in \cite[Definition 5.4.1]{HMT2}. 
For $\alpha\geq \omega$, we let ${\sf CPA}_{\alpha}$ (${\sf CPEA}_{\alpha}$) denote the reduct of $\PA_{\alpha}$($\PEA_{\alpha}$) whose signature is 
obained from that of $\PA_{\alpha}$ ($\PEA_{\alpha}$) by discarding all infinitary cylindrifiers, and its axiomatization is that of 
$\PA_{\alpha}$($\PEA_{\alpha})$ restricted to the new signature.
$\QA_{\alpha}$ ($\QEA_{\alpha}$) denotes the class of quasi--polyadic (equality) algebras obtained by restricting the signature and axiomatization of $\PA_{\alpha}(\PEA_{\alpha}$) to
only finite substitutions and cylindrifiers.  So here the signature does not contain {\it infinitary} substitutions, the $\s_{\tau}$s are defined only for 
those maps $\tau:\alpha\to \alpha$ that move 
only finitely many points.   With cylindrifiers defined only on finitely many indices, the neat reduct operator $\Nr$ for $\QA_{\alpha}$, $\QEA_{\alpha}$, $\sf CPA_{\alpha}$ and 
$\sf CPEA_{\alpha}$ is defined analogous 
to the $\CA$ case.

We recall from \cite[Definition~5.4.16]{HMT2},  the notion of neat reducts of polyadic algebras. We shall be dealing with infinite dimensional such
algebras. Because infinite cylindrfication is allowed,  the definition of neat reducts is different from the $\CA$ case. 
We define the neat reduct operator 
only $\PA$s; the $\PEA$ case is entirely analgous considering diagonal elements.
 \begin{definition} Let $J\subseteq \beta$ and
$\A=\langle A,+,\cdot ,-, 0, 1,{\sf c}_{(\Gamma)}, {\sf s}_{\tau} \rangle_{\Gamma\subseteq \beta ,\tau\in {}^{\beta}\beta}$
be a $\PA_{\beta}$.
Let $Nr_J\B=\{a\in A: {\sf c}_{(\beta\sim J)}a=a\}$. Then
$${\bf Nr}_J\B=\langle Nr_{J}\B, +, \cdot, -, {\sf c}_{(\Gamma)}, {\sf s}'_{\tau}\rangle_{\Gamma\subseteq J,  \tau\in {}^{\alpha}\alpha}$$
where ${\sf s}'_{\tau}={\sf s}_{\bar{\tau}}.$ Here $\bar{\tau}=\tau\cup Id_{\beta\sim \alpha}$.
The structure ${\bf Nr}_J\B$ is an algebra, called the {\it $J$--compression} of $\B$.
When $J=\alpha$, $\alpha$ an ordinal, then ${\bf Nr}_{\alpha}\B\in \PA_{\alpha}$ and it is
called the neat $\alpha$ reduct of $\B$. 
\end{definition}
Assume that $\B\in \PEA_{\beta}$ for some infinite ordinal $\beta$. Then for $n<\omega$, ${\bf Nr}_n\B\subseteq \Nr_n\Rd_{qea}\B$, where 
$\Rd_{qea}$ denotes the quasi--polyadic reduct of $\B$, obtained by discarding infinitary
substitutions and  the definition of the neat reduct opeartor $\Nr_n$ here is like the $\CA$ case not involving infinitary cylindrifiers.
Indeed, if $x\in {\bf Nr}_n\B$, then ${\sf c}_{(\beta\setminus n)}x=x$, so for any $i\in \beta\setminus n$,
${\sf c}_ix\leq {\sf c}_{(\beta\setminus n)}x=x\leq {\sf c}_ix$,
hence ${\sf c}_ix=x$.   However, the converse might not be true. If $x\in \Nr_n\Rd_{qea}\B$, then ${\sf c}_ix=x$ for all $i\in \beta\setminus n$, but 
this does not imply that ${\sf c}_{(\beta\setminus n)}x=x;$ it can happen that ${\sf c}_{(\beta\setminus n)}x>x={\sf c}_ix$ (for all $i\in \beta\setminus n$).

We will show in a moment that if 
$\C\in {\sf PEA}_{\omega}$ is atomic and $n<\omega$, then both $\Nr_n\Rd_{qea}\C$ and ${\bf Nr}_n\C$ are 
completely representable $\PEA_n$s.  
This gives a plethora of completely representable ${\sf PEA}_n$s whose $\CA$ reducts are (of course) 
also completely representable. 
We present analogous positive results typically of the form:  If $\bold K$ is any of the classes defined above 
(like $\CPEA, CPA, QEA, QA$), $\D\in \bold K_{\omega}$ is atomic, $n<\omega$, 
then (under certain conditions on $\D$)  $\Nr_n\D$ is completely representable. The `certain conditions' will be 
formulated only for the dilation $\D$ and will not depend on $n$. For example for 
$\PEA$, mere atomicity  of $\D$ will suffice, 
for $\PA$ we will need complete additivity of $\D$ too. 

We need a crucial lemma. But first a definition:
\begin{definition} A {\it transformation system} is a quadruple of the form $(\A, I, G, {\sf S})$ where $\A$ is an 
algebra of any signature, 
$I$ is a non--empty set (we will only be concerned with infinite sets),
$G$ is a subsemigroup of $(^II,\circ)$ (the operation $\circ$ denotes composition of maps) 
and ${\sf S}$ is a homomorphism from $G$ to the semigroup of endomorphisms of $\A$. 
Elements of $G$ are called transformations. 
\end{definition}

The next lemma says that, roughly, in the presence of all substitution operators in the infinite dimensional case, one can 
form dilations in any higher dimension.
\begin{lemma}\label{dilation} Let $\alpha$ be an infinite ordinal and $\bold K\in \{\PA, \PEA\}$. Let $\D\in \bold K_{\alpha}.$ 
Then for any ordinal $\mathfrak{n}>\alpha$,  there exists $\B\in \bold K_{\mathfrak{n}}$ 
such that $\D={\bf Nr}_{\alpha}\B$. 
Furthermore, if $\D$ is atomic (complete), then $\B$ can be chosen to be atomic (complete). An entirely analogous result holds for 
$\sf CPA$ and $\sf CPEA$ replacing the operator ${\bf Nr}$ by the neat reduct operator $\Nr$.
\end{lemma}
\begin{proof} 
Let $\bold K\in \{\sf PA, PEA, CPA, CPEA\}.$ Assume that $\D\in \bold K_{\alpha}$ and that $\mathfrak{n}>\alpha$. 
If $|\alpha|=|\mathfrak{n}|$, then one fixes a bijection $\rho:\mathfrak{n}\to \alpha$, and defines the $\mathfrak{n}$-dimensional dilation of the diagonal free
reduct of $\D$, having the same universe as $\D$, by re-shuffling the operations of $\D$ along $\rho$ \cite{DM}. 
Then one defines diagonal elements in the $\mathfrak{n}$-dimensional dilation of the diagonal free reduct of $\D$, 
by using the diagonal elements in  $\D$ \cite[Theorem 5.4.17]{HMT2}.

Now assume that $|\mathfrak{n}|>|\alpha|$.  
Let ${\sf End}(\D)$ be the semigroup of Boolean endomorphisms on
$\D$.  Then the map $\sf S:{}^\alpha\alpha\to {\sf End}(\A)$ defined  via $\tau\mapsto {\sf s}_{\tau}$ is a homomorphism of semigroups.
The operation on both semigroups is composition of maps, so that $(\D, \alpha, {}^{\alpha}\alpha, \sf S)$ is a transformation system. 
For any set $X$, let $F(^{\alpha}X,\A)$
be the set of all maps from $^{\alpha}X$ to $\A$ endowed with Boolean  operations defined pointwise and for
$\tau\in {}^\alpha\alpha$ and $f\in F(^{\alpha}X, \A)$, put ${\sf s}_{\tau}f(x)=f(x\circ \tau)$.

This turns $F(^{\alpha}X,\A)$ to a transformation system as well.
The map $H:\A\to F(^{\alpha}\alpha, \A)$ defined by $H(p)(x)={\sf s}_xp$ is
easily checked to be an embedding of transfomation systems. Assume that $\beta\supseteq \alpha$. Then $K:F(^{\alpha}\alpha, \A)\to F(^{\beta}\alpha, \A)$
defined by $K(f)x=f(x\upharpoonright \alpha)$ is an embedding, too.
These facts are fairly straightforward to establish
\cite[Theorems 3.1, 3.2]{DM}.
Call $F(^{\beta}\alpha, \D)$ a minimal functional dilation of $F(^{\alpha}\alpha, \D)$.
Elements of the big algebra, or the (cylindrifier free)
functional dilation, are of form ${\sf s}_{\sigma}p$,
$p\in F(^{\alpha}\alpha, \D)$ where $\sigma\upharpoonright \alpha$ is injective \cite[Theorems 4.3-4.4]{DM}.
Let $\B^{-c, -d}=F({}^{\mathfrak{n}}\alpha, \D).$

For the $\PA$ case one defines cylindrifiers on $\B^{-c, -d}$ by setting for each $\Gamma\subseteq \mathfrak{n}:$
$${\sf c}_{(\Gamma)}{\sf s}_{\sigma}^{(\B^{-c. -d})}p={\sf s}_{\rho^{-1}}^{(\B^{-c, -d})} {\sf c}_{\rho(\{(\Gamma)\}\cap \sigma \alpha)}^{\D}{\sf s}_{(\rho\sigma\upharpoonright \alpha)}^{\D}p.$$
For the cases $\sf CPA$ case, one defines cylindrifiers on $\B^{-c, -d}$ by restricting $\Gamma$ to singletons, setting for each  $i\in  \mathfrak{n}:$
$${\sf c}_{i}{\sf s}_{\sigma}^{(\B^{-c, -d})}p={\sf s}_{\rho^{-1}}^{(\B^{-c. -d})} {\sf c}_{(\rho(i)\cap \sigma(\alpha))}^{\D}
{\sf s}_{(\rho\sigma\upharpoonright \alpha)}^{\D}p.$$
In both cases $\rho$ is any permutation such that $\rho\circ \sigma(\alpha)\subseteq \sigma(\alpha).$
The definition is sound, that is, it is independent of $\rho, \sigma, p$; furthermore, it agrees with the old cylindrifiers in $\D$.
Denote the resulting algebra by $\B^{-d}$.

When $\D\in \PA_{\alpha}$, identifying algebras with their transformation systems
we get that $\D\cong {\bf Nr}_{\alpha}\B^{-d}$, via the isomorphism $H$ defined
for $f\in \D$ and $x\in {}^{\mathfrak{n}}\alpha$ by,
$H(f)x=f(y)$ where $y\in {}^{\alpha}\alpha$ and $x\upharpoonright \alpha=y$,
\cite[Theorem 3.10]{DM}. In \cite[Theorems 4.3. 4.4]{DM} 
it is shown that $H(\D)={\bf Nr}_{\alpha}\B^{-d}$ where $\B^{-d}=\{\s_{\sigma}^{(\B^{-d})}p: p\in \D: \sigma\upharpoonright \alpha \text { is injective}\}.$
 
When $\D\in {\sf CPA}_{\alpha}$, identifying $\D$ with $H(\D)$, where $H$ is defined like in the $\PA$ case, 
we get that $\D\subseteq \Nr_{\alpha}\B^{-d}$ with $\B^{-d}=\{\s_{\sigma}^{(\B^{-d})}p: p\in \D: \sigma\upharpoonright \alpha \text { is injective}\}.$
We show that $\Nr_{\alpha}\B^{-d}\subseteq \D$. 
Let $x\in \Nr_{\alpha}\B^{-d}$. Then there exist $y\in \D$ and $\sigma:\beta\to \beta$ with $\sigma\upharpoonright \alpha$ injective, such that 
$x=\s_{\sigma}^{(\B^{-d})}y$. Choose any $\tau:\beta\to \beta$ such that 
$\tau(i)=i$ for all $i\in \alpha$ and $(\tau\circ \sigma)(i)\in \alpha$ for all $i\in \alpha$. Such a $\tau$ clearly exists. Since $\Delta x\subseteq \alpha$, 
and $\tau$ fixes $\alpha$ pointwise, we have $\s_{\tau}^{(\B^{-d})}x=x$.
Then $x=\s_{\tau}^{(\B^{-d})}x=\s_{\tau}^{(\B^{-d})}\s_{\sigma}^{(\B^{-d})}y=\s_{\tau\circ \sigma}^{(\B^{-d})}y= \s_{\tau\circ \sigma\upharpoonright \alpha}^{\D}y\in \D$.

Having at hand $\B^{-d}$, for all 
$i<j<\mathfrak{n}$,  the diagonal element ${\sf d}_{ij}$ (in $\B^{-d}$) can be defined, using
the diagonal elements in $\D$, as in
\cite[Theorem 5.4.17]{HMT2}, obtaining the expanded required structure $\B$.  

In all cases the expanded structure $\B$ has Boolean reduct isomorphic to $F({}^\mathfrak{n}\alpha, \D)$. 
In particular, $\B$ is atomic (complete) if 
$\D$ is atomic (complete), 
because a product of an atomic (complete) Boolean algebras is atomic (complete).
\end{proof}
The proof of the following lemma follows from the definitions.
\begin{lemma}\label{join2} If $\A$, $\B$ and $\C$ are Boolean algebras, such that 
$\A\subseteq \B\subseteq \C$, $\B\subseteq_c \C$ and $\A\subseteq_c \C$, 
then $\A\subseteq_c \B$. 
\end{lemma}

For simplicity of notation, if $\beta\geq \omega$, $\B\in \PA_{\beta}(\PEA_{\beta})$, and $n<\omega$,  then we write $\Nr_n\B$ 
for $\Nr_n\Rd_{qa}\B$ $(\Nr_n\Rd_{qea}\B)$, where $\Rd_{qa}$ denotes `quasi-polyadic reduct'.

{\it In this section we understand complete representability for $\alpha$--dimensional algebras, $\alpha$ any ordinal,
in the classical sense with respect to generalized cartesian $\alpha$--dimensional spaces.}

It is shown in in \cite{au}, that for any infinite ordinal $\alpha$, if  $\A\in \PA_{\alpha}$ is atomic and completely additive, then it is completely representable. 
From this it can be concluded that for any $n<\omega$, 
any complete subalgebra of $\Nr_n\D$ is completely representable using lemma \ref{join}, 
because $\Nr_n\D\subseteq_c \D$ (as can be easily distilled from the next proof).

The result in \cite{au} holds for $\sf CPA$'s, cf. theorem \ref{paaa}, but it does not hold for $\PEA_{\omega}$s and $\sf CPEA_{\omega}$s.
It is not hard to construct atomic algebras in the last two classes that are not even representable, let alone completely representable. 
But for such (non--representable) algebras the $n$--neat reduct, for any $n<\omega$, 
will be completely representable 
as proved next (in theorems \ref{pa} and \ref{paa}):
\begin{theorem}\label{pa} If $2<n<\omega$ and $\D\in \PEA_{\omega}$ is atomic, then
any complete subalgebra of $\Nr_n\D$ is completely representable. In particular, 
${\bf Nr}_n\D$ is completely representable. 
\end{theorem}\begin{proof}
We identify notationally set algebras with their domain. 
Assume that  $\A\subseteq_c{\Nr}_n\D$, where $\D\in \PEA_{\omega}$ is atomic. 
We want to completely represent $\A$.  Let $a\in \A$ be non--zero. We will find  a homomorphism $f:\A\to \wp(^nU)$ 
such that $f(a)\neq 0$, and $\bigcup_{y\in Y}f(y)={}^nU$, whenever $Y\subseteq \A$ satisfies $\sum^{\A}Y=1$.

Assume for the moment  (to be proved in a while) that $\A\subseteq_c \D$. Then by lemma \ref{join} $\A$ is atomic, because $\D$. For brevity, let $X=\At\A$. 
Let $\mathfrak{m}$ be the local degree of $\D$, $\mathfrak{c}$ its effective cardinality 
and let $\beta$ be any cardinal such that $\beta\geq \mathfrak{c}$
and $\sum_{s<\mathfrak{m}}\beta^s=\beta$; such notions are defined in \cite{DM, au}.

We can assume by lemma \ref{dilation}, that $\D=\Nr_{\omega}\B$, with $\B\in \PEA_{\beta}$.  
For $\tau\in {}^{\omega}\beta$, we write $\tau^+$ for $\tau\cup Id_{\beta\setminus \omega}(\in {}^\beta\beta$).
Consider the following family of joins evaluated in $\B$,
where $p\in \D$, $\Gamma\subseteq \beta$ and
$\tau\in {}^{\omega}\beta$:
(*) $ {\sf c}_{(\Gamma)}p=\sum^{\B}\{{\sf s}_{{\tau^+}}p: \tau\in {}^{\omega}\beta,\ \  \tau\upharpoonright \omega\setminus\Gamma=Id\},$ and (**):
${\sf s}_{{\tau^+}}^{\B}X=1.$
The first family of joins exists \cite[Proof of Theorem 6.1]{DM}, \cite{au}, and the second exists, 
because $\sum ^{\A}X=\sum ^{\D}X=\sum ^{\B}X=1$ and $\tau^+$ is completely additive, since
$\B\in \PEA_{\beta}$. 

The last equality of suprema follows from the fact that $\D=\Nr_{\omega}\B\subseteq_c \B$ and the first
from the fact that $\A\subseteq_c \D$. We prove the former, the latter is exactly the same replacing
$\omega$ and $\beta$, by $n$ and $\omega$, respectivey, proving that $\Nr_n\D\subseteq_c \D$, hence $\A\subseteq_c \D$.  

We prove that $\Nr_{\omega}\B\subseteq_c \B$. Assume that $S\subseteq \D$ and $\sum ^{\D}S=1$, and for contradiction, that there exists $d\in \B$ such that
$s\leq d< 1$ for all $s\in S$. Let  $J=\Delta d\setminus \omega$ and take  $t=-{\sf c}_{(J)}(-d)\in {\D}$.
Then  ${\sf c}_{(\beta\setminus \omega)}t={\sf c}_{(\beta\setminus \omega)}(-{\sf c}_{(J)} (-d))
=  {\sf c}_{(\beta\setminus \omega)}-{\sf c}_{(J)} (-d)
=  {\sf c}_{(\beta\setminus \omega)} -{\sf c}_{(\beta\setminus \omega)}{\sf c}_{(J)}( -d)
= -{\sf c}_{(\beta\setminus \omega)}{\sf c}_{(J)}( -d)
=-{\sf c}_{(J)}( -d)
=t.$
We have proved that $t\in \D$.
We now show that $s\leq t<1$ for all $s\in S$, which contradicts $\sum^{\D}S=1$.
If $s\in S$, we show that $s\leq t$. By $s\leq d$, we have  $s\cdot -d=0$.
Hence by ${\sf c}_{(J)}s=s$, we get $0={\sf c}_{(J)}(s\cdot -d)=s\cdot {\sf c}_{(J)}(-d)$, so
$s\leq -{\sf c}_{(J)}(-d)$.  It follows that $s\leq t$ as required. Assume for contradiction that 
$1=-{\sf c}_{(J)}(-d)$. Then ${\sf c}_{(J)}(-d)=0$, so $-d =0$ which contradicts that $d<1$. We have proved that $\sum^{\B}S=1$,
so $\D\subseteq_c \B$.

Let $F$ be any Boolean ultrafilter of $\B$ generated by an atom below $a$. We show that $F$
will preserve the family of joins in (*) and (**).
We use a simple topological argument  used by the author in \cite{au}. 
One forms nowhere dense sets in the Stone space of $\B$ 
corresponding to the aforementioned family of joins 
as follows:

The Stone space of (the Boolean reduct of) $\B$ has underlying set,  the set of all Boolean ultrafilters
of $\B$. For $b\in \B$, let $N_b$ be the clopen set $\{F\in S: b\in F\}$.
The required nowhere dense sets are defined for $\Gamma\subseteq \beta$, $p\in \D$ and $\tau\in {}^{\omega}\beta$ via:
$A_{\Gamma,p}=N_{{\sf c}_{(\Gamma)}p}\setminus N_{{\sf s}_{\tau^+}p}$; here we require that $\tau\upharpoonright (\omega\setminus \Gamma)=Id$, 
and $A_{\tau}=S\setminus \bigcup_{x\in X}N_{{\sf s}_{\tau^+}x}.$
The principal ultrafilters are isolated points in the Stone topology, so they lie outside the nowhere dense sets defined above.
Hence any such ultrafilter preserve the joins in (*) and (**). 
Fix a principal ultrafilter $F$ preserving (*) and (**) with $a\in F$. 
For $i, j\in \beta$, set $iEj\iff {\sf d}_{ij}^{\B}\in F$.

Then by the equational properties of diagonal elements and properties of filters, it is easy to show that $E$ is an equivalence relation on $\beta$.
Define $f: \A\to \wp({}^n(\beta/E))$, via $x\mapsto \{\bar{t}\in {}^n(\beta/E): {\sf s}_{t\cup Id}^{\B}x\in F\},$
where $\bar{t}(i/E)=t(i)$ and $t\in {}^n\beta$. 
It is not hard to show that $f$ is well--defined, a homomorphism (from (*)) and atomic (from (**)), such that $f(a)\neq 0$ $({\bar{Id}}\in f(a)$).

The complete representability of ${\bf Nr}_n\D$ follows from lemmata, \ref{join}, 
\ref{join2}, by observing that  ${\bf Nr}_n\D\subseteq_c \D$, hence ${\bf Nr}_n\D\subseteq_c \Nr_n\D$. 
\end{proof}
For $\CPEA$s, we have a slightly weaker result:
\begin{theorem}\label{paa} If $n<\omega$ and  $\D\in {\sf CPEA}_{\omega}$ is atomic,  then any complete subalgebra 
of $\Nr_n\Cm\At\D$  is completely representable. In particular, if $\D$ is complete and atomic, then $\Nr_n\D$ is completely representable.
\end{theorem}
\begin{proof}
Let $\D\in {\sf CPEA}_{\omega}$ be atomic. Let $\D^*=\Cm\At\D$. Then $\D^*$ is complete and atomic and $\Nr_n\D^*\subseteq_c \D^*$.
To prove the last $\subseteq_c$, assume for contradiction that 
there is some $S\subseteq \Nr_{n}\D^*$,  $\sum ^{\Nr_{n}\D^*}S=1$, and there exists $d\in \D^*$ such that
$s\leq d< 1$ for all $s\in S$. Take $t=-\bigwedge_{i\in \omega\setminus n}(-{\sf c}_i-)d$.
This infimum is well defined because $\D^*$ is complete. Like in the previous proof it can be proved that ${\sf c}_it=t$ for all $i\in \omega\setminus n$,
hence $t\in \Nr_{n}\D^*$
 and that $s\leq t<1$ for all $s\in S$, 
which contradicts that $\sum ^{\Nr_{n}\D^*}S=1$.

Let $\beta$ be a regular cardinal $>|\D^*|$ and by lemma \ref{dilation}, 
let $\B\in {\sf CPEA}_{\beta}$ be   complete and atomic such that $\D^*=\Nr_{\omega}\B$. 
Then we have the following chain of complete embeddings:
 $\Nr_n\D^*\subseteq_c \D^*=\Nr_{\omega}{\B}\subseteq_c \B$; the last $\subseteq_c$ follows like above using that $\B$ is complete.
From the first $\subseteq_c$, since $\D^*$ is atomic, we get by lemma \ref{join}, 
that $\Nr_n\D^*$ is atomic. Let $X=\At\Nr_n\D^*$. 
Then also from the first $\subseteq_c$, we get that $\sum^{\D^*}X=1$, so $\sum^{\B}X=1$ because $\D^*\subseteq_c \B$. 
For $k<\beta$, $x\in \D^*$ and $\tau\in {}^\omega\beta$,
the following joins hold in $\B$: 
(*) ${\sf c}_kx=\sum_{l\in \beta}^{\B}{\sf s}_l^kx$ and (**) $\sum {\sf s}_{\tau^{+}}^{\B}X=1$, where 
$\tau^+=\tau\cup Id_{\beta\setminus \omega}(\in {}^\beta\beta)$. The join (**) holds, because $\s_{\tau^{+}}^{\B}$ is completely additive, since $\B$ is completely additive.

To prove (*), fix $k<\beta$. Then for all $l\in \beta$, we have ${\sf s}_l^kx\leq {\sf c}_kx$.
Conversely, assume that $y\in \B$ is an upper bound for $\{{\sf s}_l^kx: l\in \beta\}$. 
Let $l\in \beta\setminus (\Delta x\cup \Delta y)$; such an $l$ exists, because $|\Delta x|<\beta$, $|\Delta y|<\beta$ and $\beta$ is regular.
Hence, we get that ${\sf c}_lx=x$ and ${\sf c}_ly=y$.  But then ${\sf c}_l{\sf s}_l^kx\leq y$, and so ${\sf c}_kx\leq y$.
We have proved that (*) hold.

Let $\A=\Nr_n\D^*$. Let  $a\in \A$ be non--zero. We want to find a complete representation $f:\A\to \wp(V)$ ($V$ a unit 
of a $\sf Gs_{n}$, i.e a disjoint 
union of cartesian spaces)  such that $f(a)\neq 0$.  
Let $F$ be any Boolean ultrafilter of $\B$ generated by an atom below $a$. 
Then, like in the proof of theorem \ref{pa}, $F$
will preserve the family of joins in (*) and (**).

Next we proceed exactly like in the proof of theorem \ref{pa}. 
For $i, j\in \beta$, set $iEj\iff {\sf d}_{ij}^{\B}\in F$.
Then $E$ is an equivalence relation on $\beta$.
Define $f: \A\to \wp({}^n(\beta/E))$, via $x\mapsto \{\bar{t}\in {}^n(\beta/E): {\sf s}_{t\cup Id}^{\B}x\in F\},$
where $\bar{t}(i/E)=t(i)$ and $t\in {}^n\beta$. 
Then $f$ is well--defined, a homomorphism (from (*)) and atomic (from (**)).  Also
$f(a)\neq 0$ because $a\in F$, so $Id\in f(a)$.
\end{proof}

\begin{theorem}\label{paaa} If $\D\in \sf CPA_{\omega}$ is atomic and completely additive, then it is completely representable
\end{theorem}
\begin{proof} Replace $\D$ by its \de\ completion $\D^*=\Cm\At\D$. Then $\D$ is completely representable $\iff$ $\D^*$ is completely representable 
and furthermore $\D^*$ is complete.
It suffices thus to show that $\D^*$  is completely representable. 
One forms an atomic complete dilation $\B$ of $\D^*$ to a {\it regular cardinal $\beta>|\D^*|$} exactly as in lemma \ref{dilation}. 
For $\tau\in {}^{\omega}\beta$, let $\tau^+=\tau\cup Id_{\beta\setminus \omega}$.
Then like before $\D^*\subseteq_c \B$ and so    
the following family  of joins hold in $\B$: For all $i<\beta$, 
$b\in \B$  ${\sf c}_ib=\sum_{j\in \beta} {\sf s}_j^ix$
and for all $\tau\in {}^{\omega}\beta$, $\sum {\sf s}_{{\tau}^+}\At\D^*=1$. Let $a\in \D^*$ be non zero. 
Take any ultrafilter $F$ in the Stone space of $\B$ generated by an atom below $a$. 
Then   $f:\D^*\to \wp(^{\omega}\beta)$ defined 
via $d\mapsto \{\tau\in {}^{\omega}\beta: \s_{\tau^+}^{\B}d\in F\}$ 
is a complete representation of $\D^*$ such that $f(a)\neq 0$. 
\end{proof}

If the dilations are in $\QEA_{\omega}$ (an $\omega$ dimensional quasi--polyadic equality algebra) we have a weaker result. We do not know whether the result proved for $\PEA_{\omega}$ holds
when the $\omega$--dilation is an atomic $\QEA_{\omega}$. Entirely analogous results hold if we replace $\QEA_{\omega}$ by $\QA_{\omega}$.

\begin{theorem} \label{suf} Let $n<\omega$. Let $\D\in \QEA_{\omega}$ be atomic. 
Assume that  for all $x\in \D$ for all $k<\omega$, 
${\sf c}_kx=\sum_{l\in \omega}{\sf s}_l^kx$. 
If $\A\subseteq \Nr_{n}\D$  such that $\A\subseteq_c \D$ (this is stronger than $\A\subseteq_c \Nr_n\D$), 
then $\A$ is completely representable. 
\end{theorem}
\begin{proof} First observe that $\A$ is atomic, because $\D$ is atomic and $\A\subseteq_c \D$.
Accordingly, let $X=\At\A$. Let $a\in \A$ be non-zero. 
Like  before, one finds a principal ultrafilter $F$ such that $a\in F$ and $F$ preserves the family of joins ${\sf c}_{i}x=\sum^{\D}_{j\in \beta} {\sf s}_j^i x$, 
and $\sum {\sf s}^{\D}_{\tau}X=1$, where $\tau:\omega\to \omega$ is a finite transformation; that is $|\{i\in \omega: \tau(i)\neq i\}| <\omega$. 
The first family of joins exists by assumption, the second exists, since $\sum^{\D}X=1$ by 
$\A\subseteq_c \D$ and the $\s_{\tau}$s are completely additive.   Any principal ultrafilter $F$ generated by an atom below $a$ will do, as shown in the previous 
proof.  Again as before, the selected $F$ gives the required complete representation $f$ of $\A$. 
\end{proof}
The following example shows that the existence of the joins in theorem \ref{suf} is 
not necessary.

\begin{example}
Let $\D\in \QEA_{\omega}$ be the full weak set algebra with top element $^{\omega}\omega^{(\bold 0)}$ where $\bold 0$ is the constant $0$ sequence. 
Then, it is easy to show that for any $n<\omega$, $\Nr_n\A$
is completely representable.  
Let $X=\{\bold 0\}\in \D$. Then for all $i\in \omega$, we have ${\sf s}_i^0X=X$. 
But $(1, 0,\ldots )\in {\sf c}_0X$, so that $\sum_{i\in \omega} {\sf s}_i^0X=X\neq {\sf c}_0X$.  
Hence the joins in theorem \ref{suf} do not hold.
\end{example}

Now fix $1<n<\omega$ and let $\D$ be as in the previous example.  If we take $\D'=\Sg^{\D}\Nr_n\D$, then $\D'$ of course will still be a weak set algebra, and it will be locally finite, so that 
${\sf c}_ix=\sum^{\D}{\sf s}_j^ix$ for all $i<j<\omega$.
However, $\D'$ will be atomless as we proceed to show. Assume for contradiction  that it is not, and let $x\in \D'$ be an atom. 
Choose $k, l\in \omega$ with $k\neq l$ and ${\sf c}_kx=x$, this is possible since $\omega\setminus \Delta x$ is infinite. 
Then ${\sf c}_k(x\cdot  {\sf d}_{kl})=x$, so $x\cdot {\sf d}_{kl}\neq 0$. 
But $x$ is an atom, so $x\leq {\sf d}_{kl}$. This gives that $\Delta x=0$, 
and by \cite[Theorem 1.3.19]{HMT2} $x\leq -{\sf c}_k-{\sf d}_{kl}$. 
It is also easy to see that $({\sf c}_k -{\sf d}_{kl})^{\D'}={}^{\omega}\omega^{(\bold 0)}$, 
from which we conclude that $x=0$, 
which is a contradiction.

\section{Vaught's theorem and omitting types for globally guarded fragments}

Now we address complete representations with respect to {\it globally guarded relativized semantics} 
in terms of the existence of atomic $\omega$--dilations.  
We will be rewarded with some positive results on Vaught's theorem and omitting types for guarded versions of $L_n$ $(n\in \omega$)
and 
$L_{\omega, \omega}$.

As usual our approach is algebraic. Unless otherwise indicated, $n$ is a finite ordinal.
The class ${\sf G}_n$ is the class of set algebras having the same signature as $\PEA_n$; if $\A\in \sf G_n$, then  the 
top element of $\A$ is a set of $n$--ary sequences $V\subseteq {}^nU$ (some non-empty set $U$),
such that if $s\in V$ and $\tau: n\to n$ is a finite transformation, then $s\circ \tau\in V$. 
The operations of $\A$ with top element $V$, whose domain is a subset of $\wp(V)$, 
are like the operations in polyadic equality set algebras of dimension $n$,
but relativized to the top element $V.$  
This last class was  dealt with recently in \cite{Fer2} and is commonly referred to as ${\sf G}_n$.   It is known \cite{Fer2}, 
that for $n>1$,  $\CA_n\nsubseteq  \bold I{\sf D}_n$ 
and $\sf PEA_n\subseteq \bold I{\sf G}_n.$ It is clear by definition that ${\sf Gs}_n\subseteq {\sf D}_n$,  
hence ${\sf RCA}_n\subseteq  \bold I {\sf D_n}$. Units of $\sf D_n$, and $\sf G_n$ are unions of (not necessarily disjoint) weak spaces 
of dimension $n$.

Let ${\sf CQE}_{\omega}$ be the class of cylindric quasi-polyadic equality algebras in \cite[Definition 6.2.5]{Fer2}. 
Let ${\sf CRG}_n$ stand for the class of completely representable $\sf G_n$s. It is known \cite{Fer2} that ${\sf CRG}_n={\sf G}_n\cap \At.$
Whereas we do not know whether for $ 2<n<\omega$,  $\bold S_c\Nr_n\QEA_{\omega}^{\sf at}=\sf CRPEA_n$ or not, 
we have a `relativized version' of this equality, when we do not insists on full fledged commutativity of cylindrifiers:
\begin{theorem}\label{one}
For any $n<\omega$, $\bold S_c{\sf Nr}_n({\sf CQE}_{\omega}\cap \At)={\sf Gs}_n\cap \At={\sf CRG}_n$.
\end{theorem}
\begin{proof} Assume that $\C\in \sf G_{n}$ is atomic, hence completely representable. Assume further that ${\sf G}_n$ 
has unit
a union of the form $\bigcup_{i\in I}{}^nU_i$($I$ and $U_i$ non--empty sets). Units of $\sf G_n$s are of this form \cite[Theorem 6.2.8]{Fer2}.
The difference from $\sf Gs_n$s units is that the $U_i$s are not required to be pairwise disjoint.
But the proof is exactly like that of the first item of theorem \ref{iii}: For $i\in I$, let $E_i={}^nU_i$. Fix $f_i\in {}^{\omega}U_i$. Let $W_i={}^{\omega}U_i^{(f_i)}$.
Let ${\C}_i=\wp(W_i)$. Then $\C_i$ is atomic,
and $\Nr_{n}\C_i$ is atomic.
Define $h_i: \A\to \Nr_{n}\C_i$ by
$h_i(a)=\{s\in W_i: \exists a'\in \At\A, a'\leq a;  (s(i): i<n)\in a'\}$. 
Let $\D=\bold P _i \C_i$. Let $\pi_i:\D\to \C_i$ be the $i$th projection map.
Then $\D$ is atomic and $\A$ embeds into $\Nr_{n}\D$ via $J:a\mapsto (h_i(a) :i\in I)$.  The embedding is atomic, hence complete.
We have shown that $\A\in \bold S_c{\sf Nr}_{n}(\sf CQE_{\omega}\cap \sf At).$

Conversely,  assume that 
$\A\in  \bold S_c{\sf Nr}_n({\sf CQE}_{\omega}\cap \At)$. 
The variety ${\sf CQE}_{\omega}$ is definitionally equivalent to ${\sf TA}_{\omega}$ \cite{Fer} and \cite[p.149]{Fer2}, so 
$\A\in \bold S_c{\sf Nr}_n({\sf TA}_{\omega}\cap \At)$. 
Assume tha $\A\subseteq_c \Nr_n\D$ for some
atomic $\D \in \sf TA_{\omega}$. Then by \cite{Fer}, $\D$ is  completely representable.  
so $\Nr_n\D$ is completely repesentable, and this 
induces a complete representation of $\A(\subseteq_c \Nr_n\D)$,  
hence $\A\in \sf G_n\cap \At$.
\end{proof}

We consider  the class of algebras having the same signature as that of $\sf CPEA_{\omega}$ and its axiomatization is obtained from that of $\sf CPEA_{\omega}$ 
by weakening the axiom of 
commutativity of cylindrifiers and replacing  it by the weaker axiom $({\sf CP_9})^*$ as in \cite[Definition 6.3.7]{Fer2}. 
This class is studied in \cite{Fer2}, 
under the name of cylindric--polyadic equality algebras of dimension $\omega$ denoted in 
{\it op.cit} by $\sf CPE_{\omega}$ which is the notation we adopt here. 
\begin{theorem}\label{two} 
If $\D\in {\sf CPE}_{\omega}$ is atomic, then 
$\D$ is completely representable. 
\end{theorem}
\begin{proof}   
It suffices thus to show that $\D^*$  is completely representable. 
One forms an atomic dilation $\B$ of $\D^*$ to a {\it regular cardinal $\beta>|\D^*|$} exactly as in lemma \ref{dilation}. Since 
the Boolean reduct of $\B$, by construction, is isomorphic to a product of the complete Boolean reduct of the algebra $\D^*$, then $\B$ is complete
and $\D^*\subseteq_c \B$.

Analogous joins of (*) and (**) in the proof of theorem \ref{pa} remain valid in $\B$ but with the intervention of so--called {\it admissable substitutions}, $\sf adm$ for short.
The map $\tau:\beta\to \beta$ is {\it admissable}, if it does not move elements in $\beta\setminus\omega$, and that  those elements in $\omega$ 
that are moved by $\tau$ are mapped into elements in $\beta\setminus \omega$. The family joins analogous to (*),
which we continue to refer to as (*)  is the following family  of joins holding for all $i<\beta$, 
$b\in \B$, for all $\tau\in \sf adm$: 
${\sf s}_{\tau}{\sf c}_ib=\sum_{j\in \beta} {\sf s}_{\tau}{\sf s}_j^ix$.
And the family of joins in (**), which we also continue to denote by (**), take the form: 
$\sum^{\B}{\sf s}_{\tau}\At\D=1$, $\tau\in \sf adm$.

As above, one picks a principal ultrafilter $F$ of $\B$ generated by any atom below the given non--zero element, call it $a$--where $a\in \D$--required to be mapped into a non-zero element in 
a complete representation of $\D$. This is possible because $\B$ is atomic. Now $a\in F$, and like before this ultrafilter $F$  lies outside the nowhere dense sets 
in the Stone space of $\B$ corresponding to (*) and (**) as before. 
The required complete representation is induced by a complete representation of $\B$ 
defined via $F$ and $\sf adm$ restricted to $\D$.
In more detail,  proceeding like in \cite{mlq}: Let $\Gamma=\{i\in \beta:\exists j\in \omega: {\sf c}_i{\sf d}_{ij}\in F\}$.
Since $\c_i{\sf d}_{ij}=1$, then $\omega\subseteq \Gamma$. Furthermore the inclusion is proper, 
because for every $i\in \omega$, there is a $j\not\in \omega$ such that ${\sf d}_{ij}\in F$. 
Define the relation $\sim$ on $\Gamma$ via  $m\sim n\iff {\sf d}_{mn}\in F.$ 
Then $\sim$ is an equivalence relation because for all $i, j, k\in \omega$, ${\sf d}_{ii}=1\in F$, ${\sf d}_{ij}={\sf d}_{ji}$,
${\sf d}_{ik}\cdot {\sf d}_{kj}\leq {\sf d}_{lk}$ and filters are closed upwards.
Now we show that the required  representation will be a set algebra with top element a union of cartesian spaces (not necessarily disjoint) 
with base $M=\Gamma/\sim$.  
One defines the  homomorphism $f$ 
using the hitherto obtained perfect ultrafilter $F$ as follows:
For $\tau\in {}^{\omega}\Gamma$, 
such that $\rng(\tau)\subseteq \Gamma\setminus \omega$ (the last set is non--empty, because $\omega\subsetneq \Gamma$),
let  $\bar{\tau}: \alpha\to M$ be defined by $\bar{\tau}(i)=\tau(i)/\sim$  
and write $\tau^+$ for $\tau\cup Id_{\beta\setminus \omega}$. 
Then $\tau^+\in \sf adm$, because $\tau^+\upharpoonright \omega=\tau$, $\rng(\tau)\cap \omega=\emptyset,$ 
and $\tau^+(i)=i$ for all $i\in \beta\setminus \omega$. 

Let $V=\{\bar{\tau}\in {}^{\omega}M: \tau: \alpha\to \Gamma,  \rng(\tau)\cap \omega=\emptyset\}.$
Then $V\subseteq {}^{\omega}M$ is non--empty (because $\omega\subsetneq \Gamma$).
Now define $f$ with domain $D$ via: 
$d\mapsto \{\bar{\tau}\in V: \s_{\tau^+}^{\B}d\in F\}.$ 
Then $f$ is well defined, that is, whenever $\sigma, \tau\in {}^{\omega}\Gamma$ and 
$\tau(i)\setminus \sigma(i)$ for all $i\in \omega$, then for any $d\in \D$, 
$\s_{\tau^+}^{\B}d\in F\iff \s_{\sigma^+}^{\B}d\in F$.
The congruence relation just defined
on $\Gamma$  guarantees
that the hitherto defined homomorphism respects the diagonal elements.
For the other operations,
preservation of cylindrifiers
is guaranteed by (*) and complete  representability by (**). Finally $f(a)\neq 0$ because $a\in F$, so $Id\in f(a)$.
\end{proof}

Now we prove an omitting types theorem 
for a countable version of $\sf CPE_{\omega}$.
Consider the semigroup $\sf T$ generated by the set of 
transformations 
$$\{[i|j], [i,j], i, j\in \omega, \sf suc, \sf pred\}$$ defined on $\omega$ is a 
strongly rich subsemigroup of $(^\omega\omega, \circ)$,
where $\sf suc$ and $\sf pred$ are the successor and predecessor functions on $\omega$, respectively.
For a set $X$, let $\B(X)$ denote the Boolean set algebra $\wp(X), \cup, \cap, \sim)$. 
Let $\sf Gp_{\T}$ be the class of set algebras of the form 
$\langle\B(V), {\sf C}_i, {\sf D}_{ij}, {\sf S}_{\tau}\rangle_{i,j\in \omega, \tau\in \sf T},$
where $V\subseteq {}^{\omega}U,$  $V$ a non--empty union of cartesian spaces. 
Let $\Sigma$ be the set of equations defining cylindric polyadic equality algebras 
in \cite[Definition 6.3.7]{Fer2} restricted to the countable signature of ${\sf Gp}_{\T}$.
\begin{theorem}\label{three}  
If $\A\models \Sigma$ is countable 
and $\bold X=(X_i: i<\lambda)$, $\lambda<\mathfrak{p}$ is a family of subsets of
$\A$, such that $\prod X_i=0$ for all $i<\omega$, then there exists $\B\in \sf Gp_{\T}$ and an isomorphism $f:\A\to \B$ such that 
$\bigcap_{x\in X_i}f(x)=\emptyset$ for 
all $i<\omega$.
\end{theorem}
\begin{proof}\cite{mlq} 
For brevity let $\alpha=\omega+\omega$.
By strong richness of $\T$, it can be proved that $\A=\Nr_{\omega}\B$ where  $\B$ is an $\alpha$--dimensional
dilation with substitution operators coming from a countable subsemigroup 
$\sf S\subseteq ({}^{\alpha}\alpha, \circ)$ \cite{mlq}. 
It suffices to show that for any non--zero $a\in \A$, there exist a countable $\D\in {\sf Gp}_{\T}$ and a
homomorphism (that is not necessarily injective)
$f:\A\to \D$,  such that $\bigcap_{x\in X_i} f(x)=\emptyset$ for all $i\in \omega$ 
and $f(a)\neq 0$. The required $\B$ will be a subdirect product of the $\C_a$s.

So fix non--zero $a\in \A$. For $\tau\in \sf S$, set $\dom(\tau)=\{i\in \alpha: \tau(i)\neq i\}$ and $\rng(\tau)=\{\tau(i): i\in \dom(\tau)\}$.  
For the sake of brevity, let $\alpha=\omega+\omega$. Let $\sf adm$ be  the set of admissible substitutions in $\sf S$, where 
now $\tau\in \sf adm$ if $\dom\tau\subseteq \omega$ and $\rng\tau\cap \omega=\emptyset$.  Since $\sf S$ is countable, we have
$|\sf adm|\leq\omega$; in fact it can be easily shown that $|\sf adm|=\omega$. 
Then we have for all $i< \alpha$, $p\in \B$ and $\sigma\in \sf adm$,
\begin{equation}\label{tarek1}
\begin{split}
\s_{\sigma}{\sf c}_{i}p=\sum_{j\in \alpha} \s_{\sigma}{\sf s}_j^ip.
\end{split}
\end{equation}
By $\A=\Nr_{\omega}\B$ we also have, for each $i<\omega$, $\prod^{\B}X_i=0$, since $\A$ is a complete subalgebra of $\B$.
Because substitutions are completely additive, we have
for all $\tau\in \sf adm$ and all $i<\lambda$,
\begin{equation}\label{t1}
\begin{split}
\prod {\sf s}_{\tau}^{\B}X_i=0.
\end{split}
\end{equation}
For better readability, for each $\tau\in \sf adm$, for each $i\in \omega$, let
$$X_{i,\tau}=\{{\sf s}_{\tau}x: x\in X_i\}.$$
Then by complete additivity, we have:
\begin{equation}\label{t2}\begin{split}
(\forall\tau\in {\sf adm})(\forall  i\in \lambda)\prod {}^{\B}X_{i,\tau}=0.
\end{split}
\end{equation}

Let $S$ be the Stone space of $\B$, whose underlying set consists of all Boolean ultrafilters of
$\B$ and for $b\in B$, let $N_b$ denote the clopen set consisting of all ultrafilters containing $b$. 
Then from (\ref{t1}) and (\ref{t2}), it follows that for $x\in \B,$ $j<\alpha,$ $i<\lambda$ and
$\tau\in \sf adm$, the sets
$$\bold G_{\tau,j,x}=N_{{\sf s}_{\tau}{\sf c}_jx}\setminus \bigcup_{i} N_{{\sf s}_{\tau}{\sf s}_i^jx}
\text { and } \bold H_{i,\tau}=\bigcap_{x\in X_i} N_{{\sf s}_{\tau}x}$$
are closed nowhere dense sets in $S$.
Also each $\bold H_{i,\tau}$ is closed and nowhere
dense.
Let $$\bold G=\bigcup_{\tau\in \sf adm}\bigcup_{i\in \alpha}\bigcup_{x\in B}\bold G_{\tau, i,x}
\text { and }\bold H=\bigcup_{i\in \lambda}\bigcup_{\tau\in  \sf adm}\bold H_{i,\tau}.$$
Then $\bold H$ is meager, since it is a countable union of nowhere dense sets.  
By the Baire Category theorem  for compact Hausdorff spaces,
we get that $X=S\smallsetminus \bold H\cup \bold G$ is dense in $S$,
since $\bold H\cup \bold G$ is meager, because $\bold G$ is meager, too, since
$\sf adm$, $\alpha$ and $\B$ are all countable.
Accordingly, let $F$ be an ultrafilter in $N_a\cap X$, then 
by its construction $F$ is a {\it perfect ultrafilter} \cite[pp.128]{Sayedneat}.
Let $\Gamma=\{i\in \alpha:\exists j\in \omega: {\sf c}_i{\sf d}_{ij}\in F\}$.
Since $\c_i{\sf d}_{ij}=1$, then $\omega\subseteq \Gamma$. Furthermore the inclusion is proper, 
because for every $i\in \omega$, there is a $j\in \alpha\setminus \omega$ such that ${\sf d}_{ij}\in F$. 

Define the relation $\sim$ on $\Gamma$ via  $m\sim n\iff {\sf d}_{mn}\in F.$ 
Then $\sim$ is an equivalence relation because for all $i, j, k\in \alpha$, ${\sf d}_{ii}=1\in F$, ${\sf d}_{ij}={\sf d}_{ji}$,
${\sf d}_{ik}\cdot {\sf d}_{kj}\leq {\sf d}_{lk}$ and filters are closed upwards.
Now we show that the required  representation will be a $\Gp_{\T}$ with base
$M=\Gamma/\sim$. One defines the  homomorphism $f$ like in theorem \ref{paa}.
Using the hitherto obtained perfect ultrafilter $F$ as follows:
For $\tau\in {}^{\omega}\Gamma$, 
such that $\rng(\tau)\subseteq \Gamma\setminus \omega$ (the last set is non--empty, because $\omega\subsetneq \Gamma$),
let  $\bar{\tau}: \omega\to M$ be defined by $\bar{\tau}(i)=\tau(i)/\sim$  
and write $\tau^+$ for $\tau\cup Id_{\alpha\setminus \omega}$. 
Then $\tau^+\in \sf adm$, because $\tau^+\upharpoonright \omega=\tau$, $\rng(\tau)\cap \omega=\emptyset,$ 
and $\tau^+(i)=i$ for all $i\in \alpha\setminus \omega$. 
Let $V=\{\bar{\tau}\in {}^{\omega}M: \tau: \omega\to \Gamma,  \rng(\tau)\cap \omega=\emptyset\}.$
Then $V\subseteq {}^{\omega}M$ is non--empty (because $\omega\subsetneq \Gamma$).
Now define $f$ with domain $\A$ via: 
$a\mapsto \{\bar{\tau}\in V: \s_{\tau^+}^{\B}a\in F\}.$ 
Then $f$ is well defined, that is, whenever $\sigma, \tau\in {}^{\omega}\Gamma$ and 
$\tau(i)\setminus \sigma(i)$ for all $i\in \omega$, then for any $x\in \A$, 
$\s_{\tau^+}^{\B}x\in F\iff \s_{\sigma^+}^{\B}x\in F$.
Furthermore $f(a)\neq 0$, since ${\sf s}_{Id}a=a\in F$ 
and $Id$ is clearly admissable.

The congruence relation just defined
on $\Gamma$  guarantees
that the hitherto defined homomorphism respects the diagonal elements.
As before, for the other operations,
preservation of cylindrifiers
is guaranteed by the condition that
$F\notin G_{\tau, i, p}$ for all $\tau\in {\sf adm}, i\in \alpha$
and all $p\in A$. 
For omitting the given family 
of non--principal types, we use that 
$F$ is outside $\bold H$, too.  This means (by definition) that for each $i<\lambda$ and each  $\tau\in  \sf adm$
there exists $x\in X_i$, such that
${\sf s}_{\tau}^{\B}x\notin F$. Let $i<\lambda$. If $\bar{\tau}\in V\cap \bigcap _{x\in X_i}f(x)$, 
then $\s_{\tau^+}^{\B}x\in F$ which  is impossible because $\tau^+\in \sf adm$. 
We have shown that for each $i<\omega$, $\bigcap_{x\in X_i}f(x)=\emptyset.$
\end{proof}

\section{An algebraic analysis}

Throughout this section $n$ is fixed to be a finite ordinal $>2$. 
Here we reformulate and elaborate on the results established in theorems \ref{main} and \ref{can} in purely algebraic terms.
The next theorem shows that 
for $2<m\leq l<n<\omega$,  $\Psi(l, n)$ (as formulated in theorem \ref{main}) is `infinitely stronger' than $\Psi(l, n)_f$ for $n\geq m+2$.

\begin{theorem}\label{flatsquare} Let $2<m<\omega$ and $n$ be finite $\geq m+2$. Then the variety of $\CA_m$s 
having $n$--flat representations is not finitely axiomatizable over the variety of $\CA_m$s 
having $n$--square representations.
\end{theorem}

\begin{proof} 
Fix $2<m<n<\omega$. We use the relation algebras used in the proof of the first part of theorem \ref{OTT2}. 
Let $\mathfrak{C}(m,n,r)$ be the algebra $\Ca(\bold H)$ where $\bold H=H_m^{n+1}(\A(n,r), \omega)),$
is the $\CA_m$ atom structure consisting of all $n+1$--wide $m$--dimensional
wide $\omega$ hypernetworks \cite[Definition 12.21]{HHbook}
on $\A(n,r)$  as defined in \cite[Definition 15.2]{HHbook}.   Then $\mathfrak{C}(m, n, r)\in \CA_m$.

Write $\mathfrak{C}_r$ for   
$\mathfrak{C}(m,n,r)\in \CA_m$ as defined in \cite{HHbook} not to clutter notation. The parameters $m$ and $n$ will be clear from context.
Given positive $k$, then for any $r\geq  k^2$, \pe\ has a \ws\ in $G^k_r(\At(\A(n, r))$
\cite[Remark 15.13]{HHbook}. This implies using ultraproducts and an elementary chain argument that \pe\
has a \ws\ in the $\omega$--rounded game, in an elementary substructure of $\Pi_{r/U}\A(n,r)/F$,
hence the former is representable, and then so is the latter because
${\sf RRA}$ is a variety. 

Now \pe\ has a \ws\ in $G^k_{\omega}(\A(n,r))$ when $r\geq k^2$,
hence, $\A(n,r)$ embeds into a complete atomic relation algebra having
a $k$--dimensional relational basis by \cite[Theorem 12.25]{HHbook}. But this induces a \ws\ for \pe\ 
in  the  game $G^{k'}_{\omega}(\At(\mathfrak{C}_r))$ with $k'$ nodes and $\omega$ rounds, for $k'\geq k$, $k'\in \omega$ 
so that $\mathfrak{C}_r$ has a $k'$ --square represenation, when $r\geq k'{^2}$.
So if $n\geq m+2$, $k\geq 3$, and $r\geq k'^2$, then 
$\mathfrak{C}_r$ has an $n+1$--square representation, an $n$--flat one but does not have an $n+1$--flat one.
But $\Pi_{r/U} \mathfrak{C}_r/F\in {\sf RCA}_m(\supseteq \bold S\Nr_m\CA_{n+1}$)
by \cite[Corollaries 15.7, 5.10, Exercise 2, pp. 484, Remark 15.13]{HHbook}
and we are done.
\end{proof}
Observe that in the last part of the proof of theorem \ref{can}, we actually showed that $\bold S\Nr_n\CA_{n+3}$ is not atom--canonical. 
It is known that $\bold S \Nr_n\CA_{n+1}$ is atom--canonical because it admits a finite Sahlqvist (equational) 
axiomatization \cite{gn}.  The next theorem is  conditional covering 
also the only remaining 
case $\bold S\Nr_n\CA_{n+2}$:
\begin{theorem} Let $2<n<m<\omega$. If there exists a finite relation algebra $\R$ having $n$--blur $(J, E)$ and no 
infinite $m$--dimension hyperbasis, then the variety $\bold S\Nr_n\CA_m$ is not atom--canonical.
\end{theorem} 
\begin{proof} The algebra $\Bb_n(\R, J, E)(\in \CA_n)$ will be representable and using the same argument in the first item of theorem \ref{main} 
its \de\  completion will not be in $\bold S\Nr_n{\sf CA}_m$. This proves the required, since $\RCA_n\subseteq \bold S\Nr_n\CA_m$.
\end{proof}

In theorem  \ref{can} we showed that there is a countable $\RCA_n$ whose \de\ completion is outside $\bold S\Nr_n\CA_{n+3}$. 
More concisely, we showed that, although $\RCA_n\subseteq \bold S\Nr_n\CA_m$ for any $m\geq n$, 
$\Cm(\At\RCA_n)\nsubseteq \bold S\Nr_n\CA_{n+3}$.
In theorem \ref{main}, we showed that for $2<n\leq l<m\leq \omega$, $\Psi(l, m)_f$ holds if there exists $\A\in \Nr_n\CA_l\cap \RCA_n$ 
such that $\Cm\At\A$ does not have an $m$--flat representation, equivalently, $\Cm\At\A\notin \bold S\Nr_n\CA_m$.
This motivates the following (algebraic) definition, where we consider various variations on the operation of forming 
subalgebras by taking {\it restricted forms} thereof applied to the class $\Nr_n\CA_m$.

Recall that $\bold S_c$ denotes the operation of forming complete sublgebras. 
We let $\bold S_d$ denote the operation of forming {\it dense subalgebras}, and $\bold I$ denote the operation of forming 
isomorphic images.
\begin{definition} Let $2<n\leq l\leq m\leq \omega$. Let $\bold O\in \{\bold S, \bold S_d, \bold S_c, \bold I\}$.
\begin{enumarab}
\item An atomic algebra $\A\in \sf CA_n$ has the {\it strong $\bold O$ neat embedding property up to $m$}, if $\Cm\At\A\in \bold O\Nr_n\CA_m.$
\item An atomic algebra $\A\in {\sf RCA}_n$ is {\it strongly representable up to $l$ and $m$} if $\A\in \Nr_n\CA_l$ 
and $\Cm\At\A\in \bold S\Nr_n\CA_m$.
\end{enumarab}
\end{definition}  

Let $2<n\leq l<m\leq \omega$.
We denote the class of $\CA_n$s having the strong $\bold O$ neat embedding property up to $m$ by ${\sf SNPCA}_{n,m}^{\bold O}$,
and we let ${\sf RCA}_{n,m}^{\bold O}:={\sf SNPCA}_{n,m}^{\bold O}\cap {\sf RCA}_n$.
We denote the class of strongly representable $\CA_n$s up to $l$ and $m$ by ${\sf RCA}_n^{l,m}$.
Recall that ${\sf CRCA}_n$ denotes the class of completely representable $\CA_n$s with respect to generalized $n$--dimensional 
cartesian spaces. 
For a class $\bold K$ of $\sf BAO$s, $\bold K\cap {\sf Count}$ 
denotes the class of countable algebras in $\bold K$, $\bold K\cap {\sf Cat}$ denotes the class of algebras having at most countably many atoms in $\bold K$ 
and recall that $\bold K\cap \bf At$ denotes the class of atomic algebras in $\bold K$.
\begin{theorem}\label{main2} Let $2<n\leq l<m\leq \omega$ and $\bold O\in \{\bf S, S_c, S_d, I\}$. 
Then the following hold:
\begin{enumerate}
\item ${\sf RCA}_{n,m}^{\bold O}\subseteq {\sf RCA}_{n,l}^{\bold O}$ 
and  ${\sf RCA}_{n,l}^{\bf I}\subseteq {\sf RCA}_{n, l}^{\bold S_d}\subseteq {\sf RCA}_{n, l}^{\bold S_c}\subseteq {\sf RCA}_{n,l}^{\bold S}$. 
The  last inclusion is proper for $l\geq n+3$,
\item For $\bold O\in \{\bf S, S_c, S_d\}$, ${\sf SNPCA}_{n,l}^{\bold O}\subseteq \bold O\Nr_n\CA_l$
and for $\bold O=\bold S$, the inclusion is proper for $l\geq n+3$. But ${\sf SNPCA}_{n,l}^{\bold I}\nsubseteq \Nr_n\CA_l,$ 
\item If $\A$ is finite, then $\A\in {\sf SNPCA}_{n,l}^{\bold O}\iff \A\in \bold O\Nr_n\CA_l$ 
and $\A\in {\sf RCA}_{n, l}^{\bold O}\iff \A\in \RCA_n\cap \bold O\Nr_n\CA_l$.
Furthermore, for any positive $k$, ${\sf SNPCA}_{n, n+k+1}^{\bold O}\subsetneq {\sf SNPCA}_{n, {n+k}}^{\bold O},$
and ${\sf SNPCA}_{n,\omega}^{\bold O}\subsetneq \RCA_n,$
\item  $(\exists \A\in \RCA_n\cap {\bf At}\sim {\sf SNPCA}_{n,l}^{\bold S})\implies \bold S{\sf Nr}_n\CA_k$ is not 
atom--canonical for all $k\geq l$. In particular, $\bold S{\sf Nr}_n\CA_k$ is not atom--canonical for all $k\geq n+3$,
\item If $\bold S\Nr_n\CA_l$ is atom--canonical $\implies$ $\RCA_{n,l}^{\bold S}$ is first order definable. 
There exists a finite $k>n+1$, such that ${\sf RCA}_{n, k}^{\bold S}$ is not first order definable, 
\item For $\bold O\in \{\bold I, \bold S_c, \bold S_d\}$, $l\geq n+3$ and $\bold K$ any class 
such that $\RCA_{n, l}^{\bold O}\subseteq \bold K\subseteq {\sf SNPCA}_{n, l}^{\bold O}$, $\bold K$ 
is not first order definable,
\item Any class between $\Nr_n\CA_{\omega}\cap {\sf CRCA}_n$ and $\bold S_c\Nr_n\CA_{n+3}$ 
is not first order definable.  In particular, ${\sf CRCA}_n$ is not first order definable \cite{HH},
and $\Nr_n\CA_{n+k}$ for any $k\geq 3$ is not first order definable \cite[Theorem 5.1.4]{Sayedneat},

\item The classes  ${\sf CRCA}_n$ and ${\sf Nr}_n\CA_m$ for $n<m$ are pseudo--elementary but not elementary, nor pseudo--universal. 
Furthermore, their elementary 
theory is recursively enumerable. 
\item $\RCA_n^{l, \omega}\cap {\sf Count}\neq \emptyset$ 
and  ${\sf RCA}_n^{\omega, \omega}\cap {\sf Cat}\cap {\bf At}=\emptyset$.
\end{enumerate}
\end{theorem}
\begin{proof}
The inclusions in the first item follows from the definition and the strictness of the 
last inclusion in this item is witnessed by the algebra $\C$ used in the first part of theorem \ref{can},  since 
$\C$ 
satisfies $\C=\Cm\At\C\in \RCA_n$ but $\C\notin 
\bold S_c\Nr_n\CA_l$ for $l\geq n+3$.

Let 
$\bold O\in \{\bf S, S_c, S_d\}$. 
If $\Cm\At\A\in \bold O\Nr_n\CA_l$, then $\A\subseteq_d \Cm\At\A$, so $\A\in \bold S_d\bold O\Nr_n\CA_l\subseteq \bold O\Nr_n\CA_l$.
This proves the first part of item (2). 
The strictness of the last inclusion follows from the third part in theorem \ref{can} on squareness, since the atomic countable algebra $\A$ 
constructed in {\it op.cit} is in $\RCA_n$, but $\Cm\At\A\notin \bold S\Nr_n\CA_{l}$ for any $l\geq n+3$.

For the last non--inclusion in item (2), we use the same example in the third item of theorem \ref{iii} restricting our attention to dimension $n$. 
Let $V={}^n\Q$ and let ${\A}$ be the ${\sf Cs}_n$ with universe $\wp(V)$.
Then $\A\in {\sf Nr}_{n}\CA_{\omega}$.
Like before, let $y$ denote the $n$--ary relation
$\{s\in V: s_0+1=\sum_{i>0} s_i\}$ and let ${\B}=\Sg^{\A}(\{y\}\cup\{\{s\}: s\in V\}).$
Now $\B$ and $\A$ are in ${\sf Cs}_n$ and they share the same atom structure, namely, the singletons, 
so $\B$ is a dense sublgebra of $\A$,  and clearly $\Cm\At\B=\A(\in \Nr_n\CA_{\omega}$). As  proved in \cite{SL},  
$\B\notin {\bf  El}{\sf Nr}_{n}{\sf CA}_{n+1}$, so 
$\B\notin \Nr_n\CA_{n+1}(\supseteq \Nr_n\CA_l)$. But $\Cm\At\B\in \Nr_n\CA_{\omega}$, 
hence $\B\in {\sf RCA}_{n, l}^{\bold I}$.  We have shown that $\B\in {\sf RCA}_{n, l}^{\bold I}\sim \Nr_n\CA_l$, and we are 
through with the last required in item (2).

Item (3) follows by definition observing that if $\A$ is finite then 
$\A=\Cm\At\A$. The strictness of the first inclusion follows from the construction in \cite{t} 
where it shown that for an positive $k$, 
there is a {\it finite algebra $\A$} in 
$\bold \Nr_n\CA_{n+k}\sim \bold S\Nr_n\CA_{n+k+1}$.  
The inclusion ${\sf SNPCA}_{n,\omega}^{\bold O}\subseteq \RCA_n$ holds because if $\B\in {\sf SNPCA}_{n,\omega}^{\bold O}$, then 
$\B\subseteq \Cm\At\B\in \bold O\Nr_n\CA_{\omega}\subseteq \RCA_n$.
The $\A$ used in theorem \ref{can} witnesses  the strictness of the last inclusion.

Item (4) follows from the definition and the construction used in the third part of theorem \ref{can}.

Item (5) follows from that ${\bf S}\Nr_n\CA_l$ is canonical. So if it is atom--canonical too, 
then $\At(\bold S\Nr_n\CA_{l})=\{\F: \Cm\F\in \bold S\Nr_n\CA_l\},$ 
the former class is elementary \cite[Theorem 2.84]{HHbook}, and the last class is elementray $\iff \RCA_n^l$ is 
elementary.
Non--elementarity follows from  \cite[Corollary 3.7.2]{HHbook2} where it is proved that ${\sf RCA}_n^{\omega}$ is not elementary, together with the fact 
that  $\bigcap_{n<k<\omega}{\bold S}\Nr_n\CA_{k}=\RCA_n$. That $k$ has to be strictly greater than $n+1$, follows because 
$\bold S\Nr_n\CA_{n+1}$ is atom--canonical since it admits a finite Sahlqvist equational axiomatization \cite{gn}.  

For item (6), we take the rainbow algebra $\C$ based on the ordered structure $\Z$ and $\N$
used in the first item of theorem \ref{can}.
Then as proved therein $\C\notin \bold S_c\Nr_n\CA_{n+3}.$ 
But in \cite{mlq}, it is shown that $\C\equiv \B$ for some countable $\B\in \bold S_c\Nr_n\CA_{\omega}\cap {\sf CRCA}_n$.
This is proved by showing that \pe\ has a \ws\ in $G_k(\At\C)$ for all $k\in \omega$, hence using ultrapowers followed by an elementary chain argument
(like the argument used in the proof of item (5) of theorem \ref{iii}), we get that $\C\equiv \B$, and \pe\ has a \ws\ in $G_{\omega}(\At\B)$, hence 
$\B\in {\sf CRCA}_n\subseteq \bold S_c(\Nr_n\CA_{\omega}\cap \bf At)$.
This can be strengthened by showing that $\B$ can be chosen so that  $\B\in \bold S_d\Nr_n\CA_{\omega}\cap {\sf CRCA}_n$.
The idea is to define a $k$--rounded atomic game $\bold H_k$ for $k\leq \omega$, played on so--called {\it $\lambda$--neat hypernetworks} on an atom structure. 
This game besides the standard 
cylindrifier move (modified to $\lambda$--neat hypernetworks),  
offers \pa\ two new amalgamation moves. We omit the highly technical definitions.
One shows that \pe\ can still win the stronger $\bold H_k(\At\C)$ for all $k<\omega$, hence using 
ultrapowers followed by an elementary chain argument,  \pe\ has a \ws\ in $\bold H_{\omega}(\alpha)$ for a countable atom structure $\alpha$, such
that $\At\C\equiv \alpha$. 
The game $\bold H$ is designed so that the \ws\ of \pe\ in $\bold H_{\omega}(\alpha)$ implies that $\alpha\in \At\Nr_n\CA_{\omega}$ and 
that $\Cm\alpha\in \Nr_n\CA_{\omega}$.  Let $\B=\Tm\alpha$. Then $\B\subseteq_d \Cm\alpha\in \bold {\sf Nr}_n\CA_{\omega}$, $\B\in {\sf CRCA}_n$
and   $\C\notin \bold S_c{\sf Nr}_n\CA_{n+3}$. Hence 
$\C\notin {\sf SNPCA}_{n, m}^{\bold O}$, $\C\equiv \B$ 
and $\B\in \RCA_{n, m}^{\bold I}$ because $\Cm\At\B\in \Nr_n\CA_{\omega}$.

In item (6) we have shown that  any $\bold K$ between ${\sf CRCA}_n\cap \bold S_d{\sf Nr}_n\CA_{\omega}$ and $\bold S_c{\sf Nr}_n\CA_{n+3}$, $\bold K$ is not elementary: 
To prove item (7), 
we first need to slighty modify the construction in \cite[Lemma 5.1.3, Theorem 5.1.4]{Sayedneat}. The algebras $\A$ and $\B$ constructed in {\it op.cit} satisfy that
$\A\in {\sf Nr}_n\CA_{\omega}$, $\B\notin {\sf Nr}_n\CA_{n+1}$ and $\A\equiv \B$.
As they stand, $\A$ and $\B$ are not atomic, but they it 
can be  fixed that they are to be so giving the same result, by interpreting the uncountably many tenary relations in the signature of 
$\M$ defined in \cite[Lemma 5.1.3]{Sayedneat}, which is the base of $\A$ and $\B$ 
to be {\it disjoint} in $\M$, not just distinct. 
We work with $2<n<\omega$ instead of only $n=3$. The proof presented in {\it op.cit} lift verbatim to any such $n$.
Let $u\in {}^nn$. Write $\bold 1_u$ for $\chi_u^{\M}$ (denoted by $1_u$ (for $n=3$) in \cite[Theorem 5.1.4]{Sayedneat}.) 
We denote by $\A_u$ the Boolean algebra $\Rl_{\bold 1_u}\A=\{x\in \A: x\leq \bold 1_u\}$ 
and similarly  for $\B$, writing $\B_u$ short hand  for the Boolean algebra $\Rl_{\bold 1_u}\B=\{x\in \B: x\leq \bold 1_u\}.$
It can be shown that $\A\equiv_{\infty}\B$. 
We show that \pe\ has a \ws\ in an \ef-game over $(\A, \B)$ concluding that $\A\equiv_{\infty}\B$.
At any stage of the game,
if \pa\ places a pebble on one of
$\A$ or $\B$, \pe\ must place a matching pebble,  on the other
algebra.  Let $\b a = \la{a_0, a_1, \ldots, a_{n-1}}$ be the position
of the pebbles played so far (by either player) on $\A$ and let $\b
b = \la{b_0, \ldots, b_{n-1}}$ be the the position of the pebbles played
on $\B$.  \pe\ maintains the following properties throughout the game:
For any atom $x$ (of either algebra) with
$x\cdot \bold 1_{Id}=0$ then $x \in a_i\iff x\in b_i$ and $\b a$ induces a finite partion of $\bold 1_{Id}$ in $\A$ of $2^n$
 (possibly empty) parts $p_i:i<2^n$ and $\b b$ induces a partion of
 $\bold 1_{Id}$ in $\B$ of parts $q_i:i<2^n$.  Furthermore, $p_i$ is finite $\iff$ $q_i$ is
finite and, in this case, $|p_i|=|q_i|$.

That such properties can be maintained is fairly easy to show.
Using that $\M$ has quantifier elimination we get, using the same argument in {\it op.cit} 
that $\A\in \Nr_n\CA_{\omega}$.  The property that $\B\notin \Nr_n\CA_{n+1}$ is also still maintained.
To see why consider the substitution operator $_{n}{\sf s}(0, 1)$ (using one spare dimension) as defined in the proof of \cite[Theorem 5.1.4]{Sayedneat}.
Assume for contradiction that 
$\B=\Nr_{n}\C$, with $\C\in \CA_{n+1}.$ Let $u=(1, 0, 2,\ldots n-1)$. Then $\A_u=\B_u$
and so $|\B_u|>\omega$. The term  $_{n}{\sf s}(0, 1)$ acts like a substitution operator corresponding
to the transposition $[0, 1]$; it `swaps' the first two co--ordinates.
Now one can show that $_{n}{\sf s(0,1)}^{\C}\B_u\subseteq \B_{[0,1]\circ u}=\B_{Id},$ 
so $|_{n}{\sf s}(0,1)^{\C}\B_u|$ is countable because $\B_{Id}$ was forced by construction to be 
countable. But $_{n}{\sf s}(0,1)$ is a Boolean automorpism with inverse
$_{n}{\sf s}(1,0)$, 
so that $|\B_u|=|_{n}{\sf s(0,1)}^{\C}\B_u|>\omega$, contradiction. 

Now we prove the statement in item(7). Item (6) excludes any first order definable class between 
$\bold S_d\Nr_n\CA_{\omega}\cap \CRCA_n$ and $\bold S_c\Nr_n\CA_{n+3}$. So hoping for a contradiction, we can only 
assume that there is a class 
$\bold M$ between $\Nr_n\CA_{\omega}\cap {\sf CRCA_n}$ and $\bold S_d\Nr_n\CA_{\omega}\cap \CRCA_n$ that is first order definable. 
Then ${\bf El}(\Nr_n\CA_{\omega}\cap {\sf CRCA_n})\subseteq \bold M\subseteq \bold S_d\Nr_n\CA_{\omega}\cap \CRCA_n$. 
We have $\B\equiv \A$, and $\A\in \Nr_n\CA_{\omega}\cap {\sf CRCA_n}$, hence $\B\in {\bf El}(\Nr_n\CA_{\omega}\cap {\sf CRCA_n})\subseteq \bold S_d\Nr_n\CA_{\omega}\cap \CRCA_n$. 
We show that $\B$  is in fact outside $\bold S_d\Nr_n\CA_{\omega}\cap {\bf At}\supseteq \bold S_d\Nr_n\CA_{\omega}\cap \CRCA_n$  getting the hoped for contradiction,
and consequently the required.  
Take the cardinality $\kappa$ the signature of $\M$ to be $2^{2^{\omega}}$ and assume for contradiction that  
$\B\in \bold S_d\Nr_n\CA_{\omega}\cap \bf At$. 
Then $\B\subseteq_d \Nr_n\D$, for some $\D\in \CA_{\omega}$ and $\Nr_n\D$ is atomic. For brevity, 
let $\C=\Nr_n\D$. Then $\Rl_{Id}\B\subseteq_d \Rl_{Id}\C$.
Since $\C$ is atomic,  then $\Rl_{Id}\C$ is also atomic.  Using the same reasoning as above, we get that $|\Rl_{Id}\C|>2^{\omega}$ (since $\C=\Nr_n\CA_{\omega}$.) 
By the choice of $\kappa$, we get that $|\At\Rl_{Id}\C|>\omega$. 
By density, $\At\Rl_{Id}\C\subseteq \At\Rl_{Id}\B$, so $|\At\Rl_{Id}\B|\geq |\At\Rl_{Id}\C|>\omega$.   
But by the construction of $\B$, we have  $|\Rl_{Id}\B|=|\At\Rl_{Id}\B|=\omega$,   which is a  contradiction and we are done.

We have shown that the class ${\sf CRCA}_n$ is not elementary,  hence it is not pseudo--univeral. It is also not 
closed under $\bold S$: Take any representable algebra that is not completely representable,  
for example an infinite algebra that is not atomic. Other atomic examples are the algebras $\C$ and the term algebra $\Tm \bf At$  
dealt with in theorem \ref{can}.  
$\bold S{\sf Nr}_n\CA_{n+3}$. 
Call such an algebra lacking a complete representation $\A$. 
Then $\A^+$  is completely representable, a classical result of Monk's \cite{HH} and $\A$ embeds into $\A^+$.
For pseudo--elementarity one proceeds like the relation algebra case \cite[pp. 279--280]{HHbook} 
defining complete representability 
in a two--sorted theory, undergoing the obvious modifications.
For pseudo--elementarity  for the class ${\sf Nr}_n\CA_{\beta}$ for any $2<n<\beta$  one easily adapts \cite[Theorem 21]{r} by defining  ${\sf Nr}_n\CA_\beta$ 
in a two--sorted theory, when $1<n<\beta<\omega$, and a three--sorted one, when
$\beta=\omega$. The first part is easy.  For the second part; one uses a sort for a $\CA_n$
$(c)$, the second sort is for the Boolean reduct of a $\CA_n$ $(b)$
and the third sort for a set of dimensions $(\delta).$
For any infinite ordinal $\mu$, the defining theory for ${\sf Nr}_n\CA_{\mu}={\sf Nr}_n{\sf CA}_{\omega}$,
includes sentences requiring that the constants $i^{\delta}$ for $i<\omega$
are distinct and that the last two sorts define
a $\CA_\omega$. There is a function $I^b$ from sort $c$ to sort $b$ and sentences forcing  that $I^b$ is injective and
respects the $\CA_n$ operations. For example, for all $x^c$ and $i<n$,
$I^b({\sf c}_i x^c)= {\sf c}_i^b(I^b(x^c)).$ The last requirement is that $I^b$ maps {\it onto} the set of $n$--dimensional elements. This can  be easily expressed
via (*)
$$\forall y^b((\forall z^{\delta}(z^{\delta}\neq 0^{\delta},\ldots (n-1)^{\delta}\implies  c^b(z^{\delta}, y^b)=y^b))\iff \exists x^c(y^b=I^b(x^c))).$$
In all cases, it is clear that any algebra of the right type is the first sort of a model of this theory.
Conversely, a model for this theory will consist of  $\A\in \CA_n$  (sort $c$),
and a $\B\in \CA_{\omega}$;  the dimension of the last is the cardinality of
the $\delta$--sorted elements which is $\omega$, such that by (*) $\A=\Nr_n\B$.
Thus this three--sorted theory defines the class of neat reducts;
furthermore, it is clearly recursive. Recursive enumerability follows from \cite[Theorem 9.37]{HHbook}.

First part of the last item follows from the second part of theorem \ref{can} proving $\Psi(l, \omega)$; namely, there exists a countable 
$\A\in \Nr_n\CA_l\cap \RCA_n$ such  that $\Cm\At\A\notin \RCA_n$.
We prove the remaining part of (the last) item 8. 
Assume for contradiction that there is an $\A\in {\sf RCA}_n^{\omega, \omega}\cap \sf Cat\cap \bf At$. 
Then by definition  $\A\in \Nr_n\CA_{\omega}$, so by the first item of 
theorem \ref{iii},  $\A\in \CRCA_n$.  But this complete representation, like in the proof of $(2)\implies (3)$ 
of theorem \ref{main}, induces a(n ordinary) representation of $\Cm\At\A$ 
which is a 
contradiction. 
\end{proof}

In the following theorem ${\bf Up}$ denotes the operation of forming ultraproducts,
and ${\bf Ur}$  denotes the operation of forming {\it ultraroots}.
${\sf LCA}_n$ denotes the class of algebras introduced in the proof of theorem \ref{iii}; where it it was shown that ${\sf LCA}_n={\bf El}{\sf CRCA}_n.$
\begin{corollary}\label{main3} Let $2<n<\omega$. 
Then ${\sf Nr}_n\CA_{\omega}\cap {\bf At}\subsetneq {\bf El}{\sf Nr}_n\CA_{\omega}\cap {\bf At}
\subsetneq {\bf El}\bold S_d{\sf Nr}_n\CA_{\omega}\cap {\bf At}\subseteq  {\bf El}\bold S_c{\sf Nr}_n\CA_{\omega}\cap {\bf At}={\sf LCA}_n
\subsetneq {\sf RCA}_{n, \omega}^{\bold S} \subsetneq
{\bf Up}{\sf RCA}_{n, \omega}^{\bold S}={\bf Ur}{\sf RCA}_{n, \omega}^{\bold S}={\bf El}{\sf RCA}_{n, \omega}^{\bold S}\subsetneq {\bf S}{\sf Nr}_n\CA_{\omega}\cap {\bf At}.$
Furthermore, ${\bf El}\bold L$ for any $\bold L$  of the above classes  is an elementary subclass of $\RCA_n$ that is 
not finitely axiomatizable.
\end{corollary}
\begin{proof} The first strictness is witnessesd by the algebra $\B$ used in item (7) since $\B\in {\bf El}\Nr_n\CA_{\omega}\cap \bf At$, 
but $\B\notin \Nr_n\CA_{n+1}$. The second strictness is 
witnessed by the algebra (denoted also by) $\B$ used in item (2) taken from \cite{SL}, for in this case, 
$\B\notin {\bf El}\Nr_n\CA_{n+1}$,  but $\B\in \bold S_d\Nr_n(\CA_{\omega}\cap \bf At)\subseteq \bold S_d\Nr_n\CA_{\omega}\cap \bf At\subseteq {\bf El} \bold S_d\Nr_n\CA_{\omega}\cap 
\bf At$.
$\sf LCA_n\subset\sf {\sf RCA}_{n, \omega}^{\bold S}$, because if $\A\in \CA_n$ and \pe\ has a \ws\ in $G_k(\At\A)$ for all $k<\omega$, 
then \pe\ has  a \ws\ in $G_k(\At\Cm\At\A)=G_k(\At\A)$ for all
$k<\omega$. The inclusion is proper, because the first class is elementary by definition, while the second is not \cite{HHbook2}.
It is known \cite[Proposition 2.90]{HHbook} that
${\bf Up}{\sf RCA}_{n, \omega}^{\bold S}={\bf Ur}{\sf RCA}_{n, \omega}^{\bold S}={\bf El}{\sf RCA}_{n, \omega}^{\bold S}$.

The strictness of inclusion ${\bf El}{\sf RCA}_{n, \omega}^{\bold S}\subseteq {\sf RCA}_n\cap {\bf At}$ is not trivial to show. 
We give a sketch of the idea. Take  $\omega$--many disjoint copies of the
$n$ element graph with nodes $\{1,2,\ldots, n\}$ and edges
$0\to 1$, $1\to 2$, and $\dots n-1\to n$. Then the chromatic number of $\G$, in symbols $\chi(\G)$ is $<\infty$.  Now  $\G$  has an $n-1$ first order definable colouring.
Since ${\mathfrak M}(\G)$ with atom structure $\rho(\G)$, as defined in \cite[Definition 3.6.3, pp. 77-78]{HHbook2} is not representable by \cite[Proposition 3.6.8]{HHbook2}, then the first
order subalgebra $\F(\G)$ in the sense of \cite[pp.456 item (3)]{HHbook} is also not representable, because $\G$ is
first order interpretable in $\rho(\G)$. 
Here $\F(\G)$ is the subalgebra of ${\frak M}(\G)$
consisting of all sets of atoms in $\rho(\G)$ of the form $\{a\in \rho(\G)\}: \rho(\G)\models \phi(a, \bar{b})\}(\in {\frak M}(\G))$
for some first order formula $\phi(x, \bar{y})$ of the signature
of $\rho(\G)$ and some tuple $\bar{b}$ of atoms.  It is easy to check that $\F(\G)$ is indeed a subalgebra of ${\frak M}(\G)$, and that
$\Tm(\rho(\G))\subseteq \F(\G)\subseteq {\frak M}(\G)$.
But $\F(\G)$ is strictly larger than the  term algebra.
Indeed, the term algebra can be shown to be representable (this is not trivial).
We readily conclude that $\rho(\G)\notin {\At}{\bf El}{\sf RCA}_{n, \omega}^{\bold S}$, but $\rho(\G)\in {\At}(\sf RCA_n)$,
so ${\bf El}{\sf RCA}_{n, \omega}^{\bold S}\neq {\sf RCA}_n\cap {\bf At}$.

For non--finite axiomatizability: In \cite[Construction 3.2.76, pp.94]{HMT2}
the non--representable Monk algebras used are finite,  hence they atomic and
are outside ${\sf RCA}_n\supseteq {\bf El}{\sf RCA_{n, \omega}^{\bold S}}\supseteq {\sf LCA}_n$.  Furthermore, any non--trivial ultraproduct of such algebras
is also atomic and is in  
${\sf Nr}_n\CA_{\omega}\cap {\bf At}\subseteq {\bf El}{\sf Nr}_n\CA_{\omega}\cap {\bf At}\subseteq {\bf El}\bold S_c{\sf Nr}_n\CA_{\omega}\cap {\bf At}=
{\sf LCA}_n\subseteq {\bf El} {\sf  RCA}_{n,\omega} ^{\bold S}$ (Witness too the proof of the last item in theorem \ref{iii}).
 \end{proof}


\begin{thebibliography}{}

\bibitem{gn}H. Andr\'{e}ka, {\it A finite axiomatization of locally square cylindric-relativized set algebras,} Studia Sci. Math. Hungar. \textbf{38} (2001),  1-11.


\bibitem{1} H. Andr\'eka, M. Ferenczi and  I. N\'emeti, (Editors), {\bf Cylindric-like Algebras and Algebraic Logic},
Bolyai Society Mathematical Studies and Springer-Verlag, {\bf 22} (2012).


\bibitem{ANT}  H. Andr\'eka,  I. N\'emeti and T. Sayed Ahmed,  {\it Omitting types for finite variable fragments and complete representations.}
Journal of Symbolic Logic. {\bf 73} (2008) pp. 65--89.

\bibitem{v} J. V. Benthem, {\it {\sf Crs} and guarded logic, a fruitful contact.} In \cite{1}.


\bibitem {Biro} B. Bir\'o.  {\it Non-finite axiomatizability results in algebraic logic},
Journal of Symbolic Logic,   {\bf 57}(3)(1992), 832--843.



\bibitem{CF} E. Casanovas, R. Farre {\it Omitting types in incomplete theories},
Journal of Symbolic Logic, {\bf 61}(1)(1996), p. 236--245.




\bibitem{DM}  A. Daigneault and J.D. Monk,
{\it Representation Theory for Polyadic algebras}, Fundamenta  Mathematica, {\bf 52}(1963), p.151--176.

\bibitem{Fer}M. Ferenczi, {\it The polyadic generalization of the Boolean axiomatization of fields of sets},
Trans. of the Amer. Math. Society  {\bf 364}(2) (2011), p. 867--886.

\bibitem {Fer2} M. Ferenczi, {\it A new representation theory for cylindric-like algebras},
In \cite{1} p.106--135 .


\bibitem{Fre} D.H. Fremlin {\it Consequences of Martin's axiom}. 
Cambridge University Press, 1984.


\bibitem{HMT2}  L. Henkin, J.D. Monk and  A. Tarski {\it Cylindric Algebras Part I}.
North Holland, 1985.

\bibitem{HH} R. Hirsch and I. Hodkinson {\it Complete representations in algebraic logic},
Journal of Symbolic Logic, {\bf 62}(3)(1997) p. 816--847.

\bibitem{HHbook}  R. Hirsch and I. Hodkinson,  {\it Relation algebras by games.}
Studies in Logic and the Foundations of Mathematics, {\bf 147} (2002).

\bibitem{HHbook2} R. Hirsch and I. Hodkinson {\it  Completions and complete representations}, in \cite{1} pp. 61--90.

\bibitem{HHM}  R. Hirsch, I. Hodkinson, and R. Maddux,
{\it Relation algebra reducts of cylindric algebras
and an application to proof theory,} Journal of Symbolic Logic
{\bf 67}(1) (2002), p. 197--213.


\bibitem{t} R. Hirsch and T. Sayed Ahmed, {\it The neat embedding problem for algebras other than cylindric algebras
and for infinite dimensions.} Journal of Symbolic Logic {\bf 79}(1) (2014), pp .208--222.


\bibitem{Hodkinson} I. Hodkinson, {\it Atom structures of relation and cylindric algebras}. Annals of pure and applied logic,
{\bf 89}(1997), p.117--148.



\bibitem{Maddux}  R. Maddux  {\it Non finite axiomatizability results for cylindric and relation algebras},
Journal of Symbolic Logic (1989) {\bf 54}, pp. 951--974.



\bibitem{Monk} J.D. Monk {\it Non--finite axiomatizability of classes of representable cylindric algebras} Journal of Symbolic Logic {\bf 34} (1969) pp. 331--343.



\bibitem{Sain}  I. Sain, {\it Searching for a finitizable algebraization of first order logic}. 
Logic Journal of IGPL. Oxford University Press. {\bf 8}(4) (2000), 495--589.

\bibitem{Basim} B. Samir and T. Sayed Ahmed {\it A Neat Embedding 
Theorem for expansions of cylindric algebras.} Logic Journal of $IGPL$ {\bf 15}(2007) p.41--51.

\bibitem{r} R. Hirsch, {\it Relation algebra reducts of cylindric algebras and complete representations},
Journal of Symbolic Logic, {\bf 72}(2) (2007), p.673--703.



\bibitem{AU} T. Sayed Ahmed, {\it Amalgamation for reducts of polyadic algebras.} Algebra Universalis, {\bf 51} (2004), p. 301--359.




\bibitem{bsl} T. Sayed Ahmed,  {\it Neat embedding is not sufficient for complete representations}
Bulletin Section of Logic {\bf 36}(1) (2007) pp. 29--36.

\bibitem{Sayedneat}  T. Sayed Ahmed, {\it Neat reducts and neat embeddings in cylindric algebras}, in \cite{1}, pp. 105--134.

\bibitem{Sayed}  T. Sayed Ahmed  {\it Completions, Complete representations and Omitting types}, in \cite{1}, pp. 186--205.

\bibitem{au} T. Sayed Ahmed  {\it The class of completely representable polyadic algebras of infinite dimension
is elementary} Algebra Universalis (in press).


\bibitem{mlq} T. Sayed Ahmed, {\it On notions of representability for cylindric--polyadic algebras and a solution to the finitizability problem for first order logic with equality}. 
Mathematical Logic Quarterly, in press.




\bibitem{SL}  T. Sayed Ahmed  and I. N\'emeti,  {\it On neat reducts of algebras of logic}, Studia Logica. {\bf 68(2)} (2001), pp. 229--262.

\bibitem{Shelah} S. Shelah, {\it Classification theory: and the number of non-isomorphic models}
Studies in Logic and the Foundations of Mathematics.  (1990).



\end{thebibliography}
\end{document}